\documentclass[11pt,reqno]{amsart}
\usepackage{amsthm, amsmath,amsfonts,amssymb,euscript,hyperref,graphics,color}
\usepackage{graphicx}
\usepackage{import}
\usepackage{tikz}
\usepackage{latexsym}
\usepackage{cite}
\usepackage{verbatim}

\newcommand{\nnb}{\nonumber}

\newcommand{\p}{\partial}

\newcommand{\dtau}{\partial_{ \tau}}
\newcommand{\D}{\breve{D}}
\newcommand{\n}{\nabla}

\newcommand{\di}{\mathrm{d}} 
\newcommand{\tr}{\text{tr}}
\newcommand{\dive}{\text{div}}
\newcommand{\curl}{\text{curl}}
\newcommand{\lie}{\mathcal{L}} 
\newcommand\ti{\tilde}

\newcommand{\W}{\breve{W}}
\newcommand{\M}{\mathcal{M}}
\newcommand{\Wb}{\overline{ W}}

\newcommand{\Rmb}{\breve{R}}

\newcommand{\Ew}{\mathcal{E}}
\newcommand{\Hw}{\mathcal{H}}
\newcommand{\Jw}{\mathcal{J}}
\newcommand{\Kw}{\mathcal{K}}

\newcommand{\Gm}{\mathcal{G}}

\newcommand{\Hwb}{\breve{\mathcal{H}}}

\newcommand\al{\alpha}

\newcommand\ga{\gamma}

\newcommand\de{\delta}
\newcommand\De{\Delta}

\newcommand\Ld{\Lambda}

\newcommand\Si{\Sigma}

\newcommand\Ph{\Phi}

\newcommand\Ps{\Psi}

\newcommand{\vep}{\varepsilon}

 \def\<{\left\langle} \def\>{\right\rangle}
 \def\({\left(} \def\){\right)}

\newcommand{\bpf}{\begin{proof}}
 \newcommand{\epf}{\end{proof}}

\newcommand{\bg}{\begin}
 \newcommand{\ed}{\end}
 
\newcommand{\beq}{\begin{equation}}
\newcommand{\eeq}{\end{equation}}

\newcommand{\ali}[2]{
	\begin{align}\label{E:#1}
		#2
	\end{align}
}
\newcommand{\alis}[1]{
	\begin{align*}
		#1
	\end{align*}
}

\newtheorem{theorem}{Theorem}[section]
\newtheorem{lemma}[theorem]{Lemma}
\newtheorem{proposition}[theorem]{Proposition}
\newtheorem{corollary}[theorem]{Corollary}

\newtheorem{remark}[theorem]{Remark}

\numberwithin{equation}{section}

\allowdisplaybreaks[4]

\begin{document}
\title[Non-compact Einstein attractors in $1+n$]{Noncompact $n$-dimensional Einstein spaces as attractors for the Einstein flow}

\author[J. Wang]{Jinhua Wang} 
\email{wangjinhua@xmu.edu.cn}
\address{School of Mathematical Sciences, Xiamen University, Xiamen 361005, China}


\begin{abstract}
 We prove that  along with the Einstein flow, any small perturbations of an $n$($n\geq4$)-dimensional, non-compact negative Einstein space with some ``non-positive Weyl tensor” lead to a unique and global solution, and the solution will be attracted to a noncompact Einstein space that is close to the background one. The $n=3$ case has been addressed in \cite{W-Y-open-M}, while in dimension $n\geq 4$, as we know, negative Einstein metrics in general have non-trivial moduli spaces. This fact is reflected on the structure of Einstein equations, which further indicates no decay for the spatial Weyl tensor. Furthermore, it is suggested in the proof that the mechanic preventing the metric from flowing back to the original Einstein metric lies in the non-decaying character of spatial Weyl tensor. In contrary to the compact case considered in Andersson-Moncrief \cite{A-M-11-cmc}, our proof is independent of the theory of infinitesimal Einstein deformations. Instead, we take advantage of the inherent geometric structures of Einstein equations and develop an approach of energy estimates for a hyperbolic system of Maxwell type. 
\end{abstract}
\maketitle

\section{Introduction}\label{sec-intro}

 \subsection{Background}
Let $M$ be an $n$-dimensional, complete, non-compact Riemannian manifold admitting a negative Einstein metric $\gamma$\footnote{It follows from Myers' theorem \cite{Myers-35} that complete, non-compact Einstein metrics have a non-positive Einstein constant.}. We let the Einstein constant be $-(n-1)$ after rescaling.
There are a wealth of examples of complete, non-compact Einstein spaces, see for instance  \cite[\S 7 D, \S15]{Besse-Einstein}.  
Let $\mathcal M$ be a $(1+n)$-manifold of the form $\mathbb{R} \times M$. 
Then the \emph{Lorentz cone spacetime} $(\mathcal M, \bar\gamma)$ with $\bar\gamma$ given by $$\bar\gamma = -dt^2 + t^2 \gamma$$ is a solution to the vacuum Einstein equations in dimension $1+n$. When $n=3$, $(\mathcal M,\, \bar\gamma)$ is flat and known as the (open) \emph{Milne model}.

In the case of $n=3$,  the  (open) Milne spacetime is embedded into Minkowski spacetime. There were related stability results for Minkowski spacetime \cite{Friedrich-91, Christodoulou-K-93, Lind-Rod-05, Lind-Rod-10}. 

If the spatial manifold $M$ is closed (compact without boundary), Andersson and Moncrief \cite{A-M-03-local, A-M-11-cmc} first proved the stability of $(1+n)$-dimensional spacetime $(\mathcal{M}, \, \bar \ga)$ when assuming that the background, spatial manifold $(M, \, \ga)$ is stable (that is, the Einstein operator has non-negative eigenvalues) with a smooth moduli space. Therein \cite{A-M-11-cmc} they showed that the decay rates of gravity depend on the lowest eigenvalue of the Einstein operator on $(M, \ga)$. The proof was based on CMCSH gauge and energy estimates through a wave-type energy for the gravity.   In the particular case of $n=3$, when combined with a sharp estimate on the lower bound for the eigenvalues of Einstein operator \cite{Koiso-78, Kroncke-15}, the method of \cite{A-M-11-cmc} provided almost $t^{-1}$ decay estimates and had prompted more researches in non-vacuum context \cite{Andersson-Fajman-20, Branding-Fajman-Kroencke-18,Fajman-Wyatt-EKG, Barzegar-Fajman-charged, F-O-O-W-fluid23}.

Alternatively, when considering the $(1+3)$-dimensional Einstein Klein--Gordon system, we in \cite{Wang-J-EKG-2019} adopted the CMC gauge with zero shift and carried out the energy estimates through Bel--Robinson energy (cf. \cite{A-M-04,Christodoulou-K-93}). Remarkably, our proof suggests the decay rates of geometric quantities, such as the Weyl tensor and the second fundamental form, are essentially independent of the stability properties of (spatial) Einstein geometry. Instead, the proof exhibits that the decay rates depend on the expanding geometry of Milne spacetime and this has been reflected in the structure of Bianchi equations and other geometric structure equations. With the same decay mechanic, later  in \cite{Wang-KK-21}, we proved global existence for a nonlinear wave model in the Kaluza--Klein spacetime over the closed Milne model.

Although the above results are all focused on the case of compact spatial manifold, we in principle expect stability results hold for some noncompact cases as well. For this purpose, we intend to develop techniques that are independent of the lower bound for the eigenvalues of Einstein operator, or the theory for moduli spaces of Einstein metrics, or even the CMC foliations (which will involve constructing CMC data), since on these topics, little is known in the noncompact case. To begin with, we note that the decay mechanic disclosed in \cite{Wang-J-EKG-2019} has made it clear that the dynamic part of the proof (not including the construction of data) is also valid when the spatial manifold is noncompact. Apart from \cite{Wang-J-EKG-2019}, the proof in \cite{Wang-KK-21} holds as well if the closed Milne model were replaced by the open one. Given these hints, the author and Yuan \cite{W-Y-open-M} modified the framework proposed in \cite{Wang-J-EKG-2019} using Gaussian normal gauge\footnote{The local existence theorem for vacuum Einstein equations in Gaussian normal gauge had been implied by the work of Andersson--Rendall \cite{Andersson-Rendall-quiescent} in the analytic category, while for data with bounded energy, it was accomplished by Fournodavlos--Luk \cite{Fournodavlos:2021aa, Fournodavlos-Luk-Kasner}. }, and then addressed the stability of the open Milne spacetime. Meanwhile, the method of \cite{W-Y-open-M} has in turn provided a simplified proof for \cite{Wang-J-EKG-2019}.

From other perspectives, stability of noncompact spaces such as hyperbolic space appears in the context of Ricci flow \cite{Li-Yin-IMRN, Schulze-hyper-11, Bamler-GAFA} as well. These works are concerned with parabolic equations for which methods are quite different from that for hyperbolic equations and the proofs therein relied on analysis for the spectrum of Laplacian or Einstein operator, or some Poincar\'{e} type inequalities.

In this paper, we continue with the interest on the noncompact topic and investigate the stability problem for $n$($n\geq 4$)-dimensional, noncompact negative Einstein spaces for the Einstein flow. It turns out, with some non-positive requirement on the Weyl tensor of the background Einstein metric, the flow exists globally and the final attractor will be an Einstein space close to the background one. Technically, the mechanic behind lies in the non-decaying feature of (spatial) Weyl tensor. This phenomenon makes itself significantly different from the $n=3$ case in which the Weyl tensor vanishes.

\subsection{Main result}

Before the statement of our main result, we will introduce some notations. 

Throughout the paper, Greek indices $\alpha, \beta \cdots, \mu, \nu \cdots$ run over $0,\cdots,n$, and Latin indices $i,j, \cdots$ run over $1,\cdots, n$. On the spacetime manifold $(\mathcal M = \mathbb{R} \times M, \, \breve{g})$, the spacetime metric $\breve{g}$ in Gaussian normal coordinates (known as geodesic polar coordinates as well) takes the form of
\begin{equation}\label{metric-form-tau}
\breve{g}_{\mu\nu} = - \di t^2 + \tilde g_{ij} \di x^i \di x^j.
\end{equation}
The spatial metric $\tilde g_{ij}$ is the induced metric of $\breve g_{\mu \nu}$ on the spatial manifold \[ M_t := \{t\} \times M. \] Let
\beq\label{def-2nd}
\tilde k_{ij} :=-\frac{1}{2} \mathcal{L}_{\p_t} \tilde g_{ij}\eeq
be the second fundamental form. We define the normalized variables
\begin{equation}\label{rescale-metric}
g_{ij} =t^{-2} \tilde g_{ij},  \quad k_{ij}  =t^{-1} \tilde k_{ij}.
\end{equation}
Let us further decompose the second fundamental form $k_{ij}$ into the trace and traceless parts 
\beq\label{rescale-2}
 \Sigma_{ij} := k_{ij} - \frac{\tr_g k }{n} g_{ij}, \quad \eta := \frac{\tr_g k}{n} + 1,
\eeq
where $\tr_g k :=g^{ij} k_{ij} = t  \tilde g^{ij} \tilde k_{ij}$. The notation $\tr_g$ always refers to taking trace with respect to $g$. 

We denote $\nabla$ the connection corresponding to $g$, and $R_{ipjq}, \, R_{ij}, \, R$ the associated Riemann, Ricci curvature tensors and scalar curvature.

  The \emph{Weyl tensor} of a Riemannian (or pseudo-Riemannian ) manifold $(M, \, g)$ with dimension $n \geq 4$ is the Weyl part $W$ of its curvature tensor:
\[ R_{ikjl} = W_{ikjl} + \frac{1}{n-2} S \odot g + \frac{R}{2n(n-1)} g \odot g,  \] where $S_{ij} := R_{i j} - \frac{R}{n} g_{ij}$ is the traceless Ricci tensor.
In general, a Weyl tensor $W$ is a $(0,4)$-tensor that belongs to $S^2 \(\Ld^2 (T^\ast M) \)$. That is, \[ W_{ipjq}=W_{j q i p}, \quad W_{i p j q}=-W_{p i j q}=-W_{ipqj},\] and it satisfies \[ W_{[ipj]q} := \frac{1}{3} \( W_{i p j q} + W_{p j i q} + W_{j i p q} \) =0 \quad \text{and} \quad W_{i p j q} g^{p q}=0.\]

\bg{definition}[Non-positive Weyl tensor]\label{def-Non-positive-Weyl}
Let $(M, \, \gamma)$ be an $n$-dimensional Riemannian manifold with $n \geq 4$, and $W[\ga]$ its Weyl tensor. We call the Weyl tensor  $W[\ga]$ non-positive if for any symmetric $(0, 2)$-tensor $A_{ij} \in S^2(T^\ast M)$\footnote{Since Weyl tensors are always trace-free, it suffices to confine to symmetric, trace-free tensors.}with $A_{ij} \in L^2 (M, \gamma)$, 
\[ \int_M A_{i j} A_{p q}  \gamma^{i i^\prime} \gamma^{j j^\prime} \gamma^{p p^\prime} \gamma^{q q^\prime} W[\ga]_{i^\prime p^\prime j^\prime q^\prime} \, \di \mu_\gamma \leq 0.
\]
\ed{definition}

\bg{remark}
If $\ga$ is the hyperbolic metric, its Weyl tensor $W[\ga]$ vanishes. This provides a trivial example of Einstein metric with non-positive Weyl tensor.

\ed{remark}

The notation $x \lesssim y$ refers to $x \leq Cy$ for some \emph{universal} constant $C$, and $x \sim y$ means $x \lesssim y$ and $y \lesssim x$. We will also employ the notation $x\lesssim_{N, n} y$ to denote $x \leq C(N, n) y$ for some constant $C(N, n)$ that depends on the order of derivative $N$ and  the dimension $n$. Usually, we will drop the subscript in $\lesssim_{N, n}$ and simply denote it by $\lesssim$. $H_{k} (M, g)$ denotes the Sobolev norm with respect to $g$ on $M$. It is usually abbreviated as $H_k$. In particular, we use $\| \cdot \|$ to denote the $L^2(M,g)$ norm.

Now we are ready for the statement of the main theorem. 

\begin{theorem}\label{thm-global-existence-Ein}
Let $\ga$ be an Einstein metric on $M$ with Einstein constant $-(n-1)$, and $W[\ga]$ be the Weyl tensor of $(M,\,\ga)$. Suppose $W[\ga]$ is non-positive in the sense of Definition \ref{def-Non-positive-Weyl}.

Suppose $(M, \, g_0)$ is an $n(n\geq 4)$-dimensional smooth complete, non-compact Riemannian manifold with positive injective radius. 
Assume that $(M,  \, g_0, \,  k_0)$ is a rescaled data set for the vacuum Einstein equations and $(g_0, \, k_0)$ is close to $(\ga, \, - \ga)$. That is, for some fixed integer $N > \frac{n}{2}$, there is an $\varepsilon > 0$ such that
\begin{align} 
& \|g_0-\gamma\|^2_{H_{N+2}(M, \,g_0)} + \|  k_{0} + g_{0} \|^2_{H_{N+1}(M, \,g_0)} \nnb \\
& \qquad \qquad +   \| \tr_{g_0} k_{0} + n \|^2_{H_{N+2}(M, \, g_0)}  \leq \varepsilon^2. \label{intro-initial-data}
\end{align}
Then if $\varepsilon$ is small enough, there is a unique and global solution $(\mathcal{M}, \, \breve g)$ to the vacuum Einstein equations with $\breve g = -dt^2 + t^2 g$ for all $t \geq t_0$.

During the evolution of Einstein flow, we have the quantitative estimates
\beq\label{decay-E-gravity}
t^{1-\delta} \(   \| t\p_t g_{ij} \|_{H_{N+1}} +  \|R_{ij} + (n-1) g_{ij} \|_{H_N}  \)   \lesssim_{N, \, n} \varepsilon,
\eeq
and moreover, the solution $g$ remains close to the background Einstein metric $\ga$,
\beq\label{uni-E-metric}  
\|g_{ij} - \gamma_{ij} \|_{H_{N+2}}  \lesssim_{N, \, n} \varepsilon, 
\eeq
and
\beq\label{uni-E-Weyl}  
 \|W - W[\gamma] \|_{H_{N}}  \lesssim_{N, \, n} \varepsilon. 
 \eeq
Here $\delta \in (0, \,1/6)$ is a fixed constant. Therefore, as $t \rightarrow +\infty$, $g(t)$ tends to an Einstein metric $g_\infty$ whose Einstein constant is $-(n-1)$ and $g_\infty$ remains close to $\ga$.

\end{theorem}

\begin{remark}
The decay estimates \eqref{decay-E-gravity} imply $g(t)$ tends to an Einstein metric $g_\infty$ as $t \rightarrow +\infty$. However, as presented in \eqref{uni-E-Weyl}, the perturbation $W-W[\ga]$ does not decay and hence the Weyl tensor $W$ fails to settle down to $W[\ga]$. It tells that the target Einstein metric $g_\infty$ is not necessarily identical to the original Einstein metric $\gamma$, although $g_\infty$ still lies in a small neighbourhood of $\gamma$ \eqref{uni-E-metric}. In principle, the non-decaying Weyl tensor is the main obstruction that preventing $g(t)$ from settling down to $\ga$. 
\end{remark}

\begin{remark}
Suppose $(M, \,\ga)$ is the (non-compact) hyperbolic space, then due to the result of Graham and Lee \cite{GRAHAM1991186}, the hyperbolic space ${\bf H^n}$ admits deformations which are Einstein of the same Einstein constant, are not isometric to $\ga$, but remain close to $\ga$. 
This agrees with our result in the case of hyperbolic background which claims the attractor $g_{\infty}$ is merely an Einstein deformation close to $\ga$. Remarkably, what happens for the non-compact hyperbolic metric is entirely different from the compact case.
We note that although the non-compact hyperbolic space is not an isolated point of the moduli space, the compact hyperbolic space is indeed an isolated point \cite[Proposition 3.4]{Koiso-78}. Therefore, if $(M, \, \ga)$ is the compact hyperbolic space, the attractor $g_{\infty}$ is exactly $\ga$.
\end{remark}

\bg{remark}
Our proof works (in fact simplifies) even when the spatial manifold $M$ is closed, and hence Theorem \ref{thm-global-existence-Ein} automatically holds in the compact case. 

Now suppose $(M,\, \ga)$ is a compact Einstein manifold with Einstein constant $\al <0$, then the associated Riemann tensor $R[\ga]_{ikjl}$ is decomposed as
\[ R[\ga]_{ikjl} = \frac{\al}{2 (n-1)} \ga \odot \ga + W[\ga]_{ikjl}. \]
By a B\"{o}chner formula in Koiso \cite[Page 428]{Koiso-78} (see also in \cite[Page 355--356]{Besse-Einstein} and \cite[Page 87]{Kroncke-15}), the non-positive condition for the Weyl tensor of $\(M, \, \ga \)$ indicates \[ \<\mathcal{A} h,\, h \> \geq - \al  \frac{ n-2}{n-1} \|h \|^2_{L^2(M, \ga)}, \quad h \in S^2(T^\ast M), \] where $\mathcal{A}$ is the Einstein operator of $\ga$.
As a result, $\ga$ is strictly stable and hence rigid. 
This suggests the target Einstein metric $g_\infty$ is identical to $\ga$ in the compact case. 
\ed{remark}

\subsection{Related works} 
The work in $1+3$ dimension \cite{W-Y-open-M} serves as a lead-up for the general $(1+n)$-dimensional case. 
For clarity, let us give an overview for the main idea in \cite{W-Y-open-M}. 
\begin{itemize}
\item The proof of \cite{W-Y-open-M} was built on the Gaussian normal coordinates. In particular, the Gaussian time coordinate helps to eliminate borderline terms arising from the non-constant lapse (in the energy argument for the Klein--Gordon field), which greatly simplifies the proof; 

\item The equation for $\eta$ \eqref{rescale-2} admits a structure of saving regularity and hence lower orders (except the top order) of $\eta$ can be estimated a priori. This further allows us to cast the Codazzi equations into an elliptic system for $\Si$ \eqref{rescale-2} so that elliptic estimates for $\Si$ follow; 

\item Working with the Bianchi equations coupled with geometric structure equations for the second fundamental form ($\Si$ and $\eta$), we are able to close the energy argument 
without referring to the top order of $\eta$. In this way, the main energy estimates are not interrupted by the bad estimate for the top order of $\eta$, since the top order energy of $\eta$ fails to decay due to the Klein--Gordon source.
\end{itemize}

In contrary to the $n=3$ case, a general $n$-dimensional, negative Einstein metric with $n \geq 4$ has a non-trivial moduli space. Therefore, we have no reason to expect the metric would flow back to the background Einstein metric $\ga$, since it is highly possible that the metric would be attracted to a different point in the moduli space. To capture the attractor precisely, Andersson--Moncrief \cite{A-M-11-cmc} introduced for the flowing metric $g(t)$ an associated shadow metric $\ga(t)$ which lies in the deformation space of $\ga$ (the background metric). In the CMCSH gauge, the vacuum Einstein equations are roughly regarded as a system of wave type equations for $g(t) - \ga(t)$. The analysis in \cite{A-M-11-cmc} showed that $g(t) - \ga(t)$ decayed with the decay rates depending on the lowest eigenvalue of Einstein operator of $\ga$ and hence $g(t)$ eventually tended to the limitation of $\ga(t)$. We note that the proof of \cite{A-M-11-cmc} were tailed in particular for the compact case for which fruitful results on the theory of infinitesimal Einstein deformations had been established \cite{Besse-Einstein}. An attempt to extend the stability result of \cite{A-M-11-cmc} to the noncompact case serves as the main motivation for the present work.

\subsection{Comments on the proof}\label{sec-comm-pf}


In this paper, we aim to give a proof that is independent of the theory of infinitesimal Einstein deformations, or analysis for the eigenvalues of Laplacian or Einstein operator (or some Poincar\'{e} type inequalities). To this end, we, as in \cite{W-Y-open-M}, take advantage of the inherent structures of Bianchi equations to derive decay estimates, so that the proof holds whether the spatial manifold is compact or not. 
However, compared to the $1+3$ case, the interactions between different components of spacetime Weyl tensor are more involved in dimension $n\geq 4$. Thus, more observations for the hidden geometric structures are needed in high dimensional Einstein equations. 
 
\subsubsection{ Framework of energy estimates for the $(1+n)$-dimensional Bianchi equations} 
We have introduced the rescaled $(g_{ij}, \, k_{ij})$ in \eqref{rescale-metric}. For notational convenience, we also introduce the rescaled spacetime metric 
\beq\label{def-rescale-spacetime}\bar g_{\mu \nu} =t^{-2} \breve{g}_{\mu \nu} = - \di \tau^2 + g_{ij} \di x^i \di x^j,  \quad \bar g^{\mu \nu}=t^{2} \breve{g}^{\mu \nu}, 
\eeq
where 
\begin{equation}\label{def-tau}
\tau := \ln t,
\end{equation}
is  the logarithmic time so that \[\p_{\tau} = t\p_t.\] 
Let $\Wb$ be the Weyl tensor of $(\mathcal M, \, \bar g)$, then \[ \Wb_{\mu \alpha \nu \beta}  = t^{-2}  \W_{\mu \alpha \nu \beta} \] where $\W_{\mu \alpha \nu \beta}$ is the Weyl tensor of $\( \mathcal{M}, \, \breve g_{\mu \nu} \)$.
We define the projection tensor (onto the spatial manifold $M$) \[ \breve h^{\mu}_{ \nu} : = \breve g^{\mu}_{ \nu} + \p_t^{\mu} (\p_t)_{\nu}, \quad \text{with} \,\, (\p_t)_\nu = \p_t^{\rho} \breve g_{\nu \rho}, \]
and the $1+n$ splitting for the rescaled spacetime Weyl tensor 
\ali{def-1+n-Weyl}{
\Ew_{ij} : ={}& 
\breve h^p_i \breve h^q_j  \Wb_{p \tau q \tau},\nnb \\
\Hw_{i j l} : ={}& \breve h^p_i \breve h^q_j \breve h^k_l \Wb_{p q k \tau}, \nnb\\
\Kw_{iljk} : = {}& \breve h^p_i \breve h^q_j \breve h^m_l \breve h^n_k \Wb_{p m q n}. 
}

We note that, when viewed as a $(0, 4)$-tensor on the Riemannian manifold $(M, \, g)$, $\Kw_{iljk}$ fails to be a Weyl tensor. Denote $\Jw_{iljk}$ the Weyl part of $\Kw_{iljk}$ and $W_{ipjq}$ the Weyl tensor of $(M, \, g)$. Then $W$ and $\Jw$ are related by \eqref{E:relat-W-Jw}. Roughly speaking,  by the Gauss--Codazzi equations \eqref{Gauss-Ricci-hat-k}, \eqref{Codazzi-div-k}--\eqref{div-curl-k-hat}, $\Ew$, $\Hw$ and $W$ determine the full geometry on $(M, g)$, namely, the Riemann curvature and the second fundamental form. Therefore, the main body of this paper is devoted to the estimates for $\Ew$, $\Hw$ and $W$.

The $(1+n)$-dimensional Bianchi equations (plugged with the vacuum Einstein equations) are decomposed into the following two systems: the system of $\Ew$ and $\Hw$ (referring to Lemma \ref{lem-trans-Hw-Ew}),
\begin{subequations}
 \begin{align}
 & \lie_{\dtau} \Ew_{i  j} +(n-2)\Ew_{ij} + \nabla^p \Hw_{p j i} =\Si^{p q} W_{p i q j} + \text{l.o.q.},\label{sym-EH-E} \\
  &\lie_{\dtau} \Hw_{p i j} + \Hw_{p i j} +   \nabla_p \Ew_{i  j } - \nabla_i  \Ew_{p  j } = \text{l.o.q.}, \label{sym-EH-H}
\end{align}
\end{subequations}
and the system of $\Hw$ and $\Jw$ (referring to Lemma \ref{lem-trans-Hw-Jw}),
\begin{subequations}
\begin{align}
 & \lie_{\dtau} \Hw_{p i j} + (n-1) \Hw_{p i j} +  \frac{n-2}{n-3} \nabla^l \Jw_{l j p i} =\text{l.o.q.}, \label{sym-HJ-H} \\
&  \lie_{\dtau} \Jw_{ipjq} + \nabla_i \Hw_{j q p} - \nabla_p \Hw_{j q i} -  \frac{1}{n-2}(\dive \Hw \odot g)_{ipjq} =\text{l.o.q.}, \label{sym-HJ-J}
\end{align} 
\end{subequations}
where l.o.q. denotes lower--order quadratic terms.

We remark that in dimension $n=3$, the Weyl tensors, $\Jw$ and $W$, vanish and hence \eqref{sym-HJ-H}--\eqref{sym-HJ-J} is redundant. In fact, taking the constraint \eqref{eq-div-Jw} into account, we are able to reduce the system \eqref{sym-HJ-H}--\eqref{sym-HJ-J} to \eqref{sym-EH-H}. Moreover, performing Hodge dual on \eqref{sym-EH-H}, the Bianchi system \eqref{sym-EH-E}--\eqref{sym-EH-H} is recast into a system of Maxwell type in  $1+3$ dimension,
\alis{
\lie_{\dtau} E - \curl H + E={}& \text{l.o.q.}, \\
\lie_{\dtau} H + \curl E + H={}&  \text{l.o.q.}.
}
This system manifests itself a first order symmetric, hyperbolic system. The corresponding energy estimates on a spacetime foliated by Riemannian manifolds with negative curvatures were carried out in \cite{ A-M-04, Wang-J-EKG-2019} and later \cite{W-Y-open-M}.

In general dimension, we take \eqref{sym-EH-E}--\eqref{sym-EH-H} for instance to illustrate the strategy of energy estimates. Let us ignore the linear lower order (in derivative) terms for a moment, and consider the simplified system 
\begin{subequations} 
\begin{align}
  \lie_{\dtau} \Ew_{i j} + \n^p \Hw_{p j i} ={}& \text{l.o.q.} ,  \label{sym-symbolic-1} \\
  \lie_{\dtau} \Hw_{p q i} +2  \n_{[p} \Ew_{q] i} ={}& \text{l.o.q.},   \label{sym-symbolic-2}
\end{align}
\end{subequations}
where the bracket on indices refers to anti-symmetrization\footnote{see also notations in the subsection \ref{sec-bracket}.}: \[ 2 \n_{[p} \Ew_{q] i} :=  \n_{p} \Ew_{q i} - \n_{q} \Ew_{p i}, \] and $\Ew$ is a symmetric tensor, and $\Hw$ satisfies \[ \Hw_{[p q] i} =0, \quad \Hw_{[ijk] } = 0. \] In fact, the system \eqref{sym-symbolic-1}--\eqref{sym-symbolic-2} is essentially equivalent to a first order hyperbolic system, which was proved in \cite{LW2021a, LW2021b} by means of introducing some auxiliary variables. Since \eqref{sym-symbolic-1}--\eqref{sym-symbolic-2} are high dimensional analogue of Maxwell equations, we will refer this system  as a \emph{hyperbolic system of Maxwell type}.

In the following, we propose a straightforward energy method for the system \eqref{sym-symbolic-1}--\eqref{sym-symbolic-2}.
After multiplying $2E$ and $H$ on \eqref{sym-symbolic-1} and \eqref{sym-symbolic-2} respectively, as the case for a first order hyperbolic system, the summation of spatial derivative terms takes a divergence form,
\alis{
 & 2 \n^p \Hw_{p j i} \cdot \Ew^{i j}  +   2 \n_{[p} \Ew_{q] l} \cdot \Hw^{p q l} 
= 2 \nabla^p    \left(  \Hw_{p j i}   \cdot \Ew^{i j}  \right),
 } which vanishes after integration on $(M, \, g)$, due to the density theorems (see Corollary \ref{coro-density}). Then we are able to derive a zero--order energy identity for $\Ew$ and $\Hw$.
The higher--derivative analogue of the above identity is given by (see Lemma \ref{lemma-div-curl} for more details),
\ali{identity-intro}{
&   \nabla^{\mathring{k}} \Delta^{ [\frac{k}{2}] } \n^p \Hw_{p j i} \cdot \nabla^{\mathring{k}} \Delta^{ [\frac{k}{2}] } \Ew^{i j}  + \nabla^{\mathring{k}} \Delta^{ [\frac{k}{2}] }  \n_p \Ew_{ i j}  \cdot \nabla^{\mathring{k}} \Delta^{ [\frac{k}{2}] } \Hw^{p i j}  \nnb\\
={}& - \sum_{0 \leq m <k} C_k^m (n-3)^{k-m} \nabla^{\mathring{m}} \Delta^{ [\frac{m}{2}] } \n^p \Hw_{p j i} \cdot \nabla^{\mathring{m}} \Delta^{ [\frac{m}{2}] } \Ew^{i j} \nnb \\  
& + \text{divergence forms} + \text{lower--order cubic terms},
}
where $C_k^m$ are the combinatorial numbers and $\mathring{k}$ is defined as in \eqref{def-circ-l}. 
We remark that the quadratic terms on the right side of (the second line of) \eqref{E:identity-intro} vanish exactly when $n=3$ and thus one encounters no additional difficulty at this stage in $n=3$ case. However, in dimension $n \geq 4$, a straightforward estimate for these quadratic terms offers inadequate decay rates to close the energy argument. For this issue, we have to make use of the structure of the original equations \eqref{sym-symbolic-1}--\eqref{sym-symbolic-2}. The idea lies in the observation that using \eqref{sym-symbolic-1} to replace $\n^p \Hw_{p j i}$ in those quadratic terms by \[ -  \lie_{\dtau} \Ew_{i j} + \text{l.o.t.}, \] intuitively, we are able to transform these quadratic terms into boundary terms such as $\dtau |\nabla^{\mathring{k}} \Delta^{ [\frac{k}{2}] } \Ew^{i j} |^2 + \cdots$ It turns out, up to some lower-order cubic terms, these quadratic terms are then recast into boundary terms plus quadratic terms
\ali{id-good-sign}{
& - 4 t  C_k^m (n-3)^{k-m} \nabla^{\mathring{m}} \Delta^{ [\frac{m}{2}] } \n^p \Hw_{p (i j)} \cdot \nabla^{\mathring{m}} \Delta^{ [\frac{m}{2}] } \Ew^{i j} \nnb\\
={}& \p_t \( 2t^2 C_k^m (n-3)^{k-m} |\nabla^{\mathring{m}} \Delta^{ [\frac{m}{2}] } \Ew|^2 \) \nnb \\
&+ 4 t C_k^m (n-3)^{k+1-m} |\nabla^{\mathring{m}} \Delta^{ [\frac{m}{2}] } \Ew|^2 + \cdots  
} and they are of favourable signs to ensure a high--order energy inequality, see \eqref{E:id-good-sign} and \eqref{E:id-energy-ho-EHw-1}. 
In this way, we establish the main framework of energy estimates for the system \eqref{sym-symbolic-1}--\eqref{sym-symbolic-2}.

\subsubsection{Decay mechanics and nonlinear coupling structures} 
Working with the Bianchi equations, the decay rate for each variable naturally comes from the linear structure. 
Actually, in the proof leading to the main theorem, we will see that the non-decaying feature of the Weyl tensor is reflected on the linear structure of Bianchi equations. More specifically, because of the linear terms $(n-2)\Ew_{ij}$ and $\Hw_{p j}$ that take favourable signs in \eqref{sym-EH-E}--\eqref{sym-EH-H}, we expect to prove almost $t^{-1}$ decay for $\|\Ew\|_{H_N}$ and $\|\Hw\|_{H_N}$; On the contrary, since there is no linear term $\Jw$ with favourable sign in the dynamic equation of $\Jw$ \eqref{sym-HJ-J}, we only derive uniform bound for $\|\Jw\|$ from the transport system \eqref{sym-HJ-H}--\eqref{sym-HJ-J}. In other words, neither $\Jw$ or the spatial Weyl tensor $W$ decays which will definitely give rise to difficulties in the energy estimates. 

Let us proceed to more details of the energy estimates. We can always begin with the system \eqref{sym-HJ-H}--\eqref{sym-HJ-J} to derive uniform bound for the zero--order energy $\|\Jw\|$ (and hence $\|W\|$), since all of the other variables decay better. After that, to obtain decay estimates for the zero--order energies of $\|\Ew\|$ and $\|\Hw\|$, we note that the nonlinear coupling $\Si^{p q} W_{p i q j}$ in \eqref{sym-EH-E}  would be subtle for it involves the non-decaying $W$. As a remark, in the case of $1+3$ dimension, the spatial Weyl tensor $W$ vanishes and thus this kind of dangerous coupling does not occur.

We now focus on handling the leading nonlinear term $\Si^{p q} W_{p i q j}$ in \eqref{sym-EH-E}, for which the interaction between the Bianchi equation \eqref{sym-EH-E} and the geometric structure equation \eqref{eq-evolution-Si} comes into play. Particularly, we have to take the relation between the Weyl component $\Ew$ and the second fundamental form $\Si$ into account. In practice, after multiplying $t \Ew$ on \eqref{sym-EH-E}, the leading term becomes $t \Si^{p q} W_{p i q j} \Ew^{ij}$.
Then with the help of \eqref{eq-evolution-Si} \[  \lie_{\p_t} ( t \Si ) = \lie_{\dtau} \Si + \Si = \Ew +\text{l.o.q.}, \] we replace  $t \Si^{p q} W_{p i q j} \Ew^{ij}$ by $t \Si^{p q} W_{p i q j} \( \lie_{\p_t} ( t \Si ) + \text{l.o.q.}\)$ so that we can further manipulate the principle part as follows
\alis{
 & \int^t_{t_0} \int_{M} 2 t \Si^{p q} W_{p i q j} \lie_{\p_t} ( t \Si )^{ij}  \di \mu_g  \, \di t \\
 ={}& \int^t_{t_0} \p_t \( \int_{M}  t^2 \Si^{p q}   \Si^{i j}  W_{ p i q j} \, \di \mu_g \) \di t  + \text{l.o.t.}.
 }
 After integration by parts, this boundary term gives rise to an additional term 
 \beq\label{modified term} 
 \int_{M}  t^2 \Si^{p q}   \Si^{i j}  W_{ p i q j}\, \di \mu_g
 \eeq in the energy.
It motivates us to introduce the non-positive condition for the background Weyl tensor (see Definition \ref{def-Non-positive-Weyl}) so that up to lower order terms, \eqref{modified term} admits a favourable sign. One can refer to subsection \ref{sec-ee-EH-0} for more details. 
Once the positivity of energy is addressed, the decay estimates (almost $t^{-1}$) for the zero--order energies of $\Ew, \, \Hw$ follows easily. 

\subsubsection{Elliptic estimates for the spatial Weyl tensor $W$}  The argument for the zero--order energy estimates fails for the higher--order case, as we can expect that the higher--order version $\n^k \( \Si^{p q} W_{ p i q j} \)$ ($k\geq 1$) yields a large number of nonlinear terms with inadequate decay rates. Therefore, instead of sticking to the transport system \eqref{sym-HJ-H}--\eqref{sym-HJ-J}, we turn to an elliptic system for $W$ (see Lemma \ref{lem-div-E-curl}), which is originated from the $n$-dimensional (spatial) Bianchi identities $\n_{[i} R_{pj]q}=0$.  The elliptic estimates help to reduce the high--order bound $\|W\|_{H_N}$ to $\|W\|$, up to lower--order terms, and we know that $\|W\|$  has been bounded in the preceding step. In other words, using the elliptic estimates makes $\n^k \( \Si^{p q} W_{ p i q j} \)$ ($k\geq 1$) essentially linear terms. This procedure of linearization enables us to close the higher--order energy argument for $(\Ew,\Hw)$ by induction.

\subsection{Outline of the paper}

The paper is organized as follows. In Section  \ref{preliminaries}, we introduce some relevant notations, and geometric structure equations and Einstein equations. In Section \ref{sec-Main-eq}, we derive the main equations, including a hyperbolic system of Maxwell type and elliptic systems. Section \ref{sec-Energy-estimates} is devoted to establishing the energy estimates. In the end, we collect the local existence theorem, the density theorem and some geometric identities in the appendix.  

{\bf Acknowledgement} J.W. thanks Wei Yuan a lot for helpful suggestions. This project is supported by NSFC (Grant No. 12271450).

\section{Preliminary}\label{preliminaries}

\subsection{Geometric notations}\label{subsec-notation}

On the $(1+n)$-dimensional spacetime $(\M,  \, \breve{g})$, the Lorentz metric $\breve{g}$ in Gaussian normal coordinates (i.e. geodesic polar coordinates) takes  the form of \eqref{metric-form-tau} and the second fundamental form is defined by \eqref{def-2nd}. We let $\D$ and $\ti \n$ be the covariant derivatives with respect to the spacetime metric $\breve{g}_{\mu \nu}$ and the spatial metric $\ti g_{ij}$ respectively.
 Relative to a frame $\{e_i\}_{i=1}^n$ that is tangent to $M_t := \{t\} \times M$, we have the formulae for connection
\ali{eq-connection}{
&\D_{\p_t} \p_t= 0, & \D_{i} e_j = \ti \n_i e_j - \ti k_{i j} \p_t,  \nnb \\
& \D_{i} \p_t = - \ti k_{i}^{j} e_j, & \breve D_{\p_t} e_i = \ti \n_{\p_t} e_i,  
}
where $\ti \n_{\p_t} e_i$ is the projection of $\breve D_{\p_t} e_i$ onto $M_t$.

Let us recall the rescaled variables defined in \eqref{rescale-metric}--\eqref{rescale-2}. In particular, \[ g_{ij}=t^{-2} \tilde g_{ij}\] is the rescaled metric, and $\nabla$ is the corresponding connection. Then
\[g^{ij}=t^{2} \tilde g^{ij}, \quad   \di \mu_g =t^{-n} \di \mu_{\tilde g},\]
and the rescaled Riemann, Ricci and scalar curvatures are related to the original ones by
\[
 R_{imjn} = t^{-2} \tilde R_{imjn}, \quad  R_{ij } = \tilde R_{ij}, \quad   R = t^2 \tilde R.
\]
The notations $R_{imjn}, \, \tilde R_{imjn}$ denote the Riemann tensors with respect to $g$ and $\tilde g$ respectively. 

The rescaled spacetime metric $\bar g_{\mu \nu}$ is defined in  \eqref{def-rescale-spacetime}.
Its associated Weyl tensor is given by 
\[ \Wb_{\mu \alpha \nu \beta}  = t^{-2}  \W_{\mu \alpha \nu \beta} \]
where $\W_{\mu \alpha \nu \beta}$ is the Weyl tensor of $\(\mathcal{M}, \, \breve g \)$. For notational convenience, we define 
\[
\bar R_{\mu \alpha \nu \beta}  := t^{-2}  \Rmb_{\mu \alpha \nu \beta}, \quad \bar R_{\mu  \nu } := \Rmb_{\mu  \nu}
\]
where $\Rmb_{\mu \alpha \nu \beta}$, $\Rmb_{\mu  \nu}$ denote the Riemann and Ricci tensors with respect to $\breve g$.

The variables $\Ew$, $\Hw$ and $\Kw$ are defined in \eqref{E:def-1+n-Weyl} as the $1+n$ splitting for the rescaled spacetime Weyl tensor $\Wb$.
Noting that
 \beq\label{trace-Kw-Ew}
g^{p q} \Kw_{i p j q} = \Ew_{i j}, \quad g^{i j} \Ew_{i j}=0,
\eeq
then we define the Weyl part of $\Kw$ by
 \beq\label{def-Jw} 
 \Jw_{imjn} := \Kw_{imjn} - \frac{1}{n-2} \(\Ew \odot g \)_{imjn},
 \eeq
where $\odot$ is the \emph{Kulkarni-Nomizu product}  
\begin{equation*}
( \xi \odot \zeta )_{imjn} := \xi_{ij} \zeta_{mn} - \xi_{jm} \zeta_{in}  + \zeta_{ij} \xi_{mn} - \zeta_{jm} \xi_{in},
\end{equation*}
for any symmetric $(0, 2)$-tensors  $\xi,\, \zeta$.
By virtue of their definitions, $\Ew,\,\Hw,\,\Jw$ are all $M$-tensors (referring to \ref{def-M-tensor}) and they can be regarded as tensor fields on the spatial manifold $(M, g_{ij})$. More properties of $\Ew,\,\Hw,\,\Jw$ are stated in the following proposition.
\bg{proposition}\label{prop-Weyl}
We have the following properties for $\Ew$, $\Hw$ and $\Jw$:
\begin{itemize}
\item $\Ew_{ij}$ is  symmetric and trace-free, \[\Ew_{i j} = \Ew_{j i}, \quad \Ew_{ij} g^{ij}=0. \]

\item $\Hw_{ijl}$ is anti-symmetric in the first two indices, trace-free, and satisfies an algebraic identity, \[ \Hw_{ijl}=-\Hw_{jil}, \quad \Hw_{i p q} g^{pq}=0, \quad \Hw_{[ijl]}=0. \]

\item $\Jw_{ipjq}$ is a Weyl tensor which satisfies \[\Jw_{ipjq}=-\Jw_{pijq}=\Jw_{piqj}, \quad \Jw_{ipjq}=\Jw_{jqip}, \quad \Jw_{ipjq} g^{p q} =0, \quad \Jw_{[ipj]q} = 0.\]
\end{itemize}
\ed{proposition}

\bpf
These properties follow from the feature of the rescaled spacetime Weyl tensor $\Wb$ and \eqref{trace-Kw-Ew}--\eqref{def-Jw}.
\epf

To do the $1+n$ splitting for Bianchi equations, we need the following calculations.
\begin{proposition}\label{prop-decomp-Bianchi}
We have the following identities,
\begin{align}
\D_{\p_t} \W_{i j p t} ={} &  \lie_{\dtau} \Hw_{i j p} + \Hw_{i j p} + k_i^l \Hw_{l j p} + k_j^l \Hw_{i l p}    + k_p^l  \Hw_{i j l},  \label{id-DT-W-T} \\
\D_{\p_t} \W_{i p q j} ={}& t  \( \lie_{\dtau} \Kw_{i p q j}  + 2  \Kw_{i p q j} \) \nnb \\
&+ t \(  k_{i}^l \Kw_{l p q j} + k_{p}^l \Kw_{i l q j} + k_{q}^l \Kw_{i p l j} + k_{j}^l \Kw_{i p q l} \), \label{id-DT-W-p} \\
\D_{\p_t} \W_{i t j t}  ={}& t^{-1}  \( \lie_{\dtau} \Ew_{i j} +  k_i^p \Ew_{p j} + k_j^p \Ew_{i p }\),  \label{id-DT-W-TT} \\
 \D_p \W_{q t i j} = {}&  t \( \nabla_p \Hw_{i j q} + k_{p}^l \Kw_{q l i j} - k_{p i} \Ew_{q j} + k_{p j} \Ew_{q i}\),  \label{id-Dp-W-T} \\
\D_p \W_{i m j n}  ={}& t^2\( \nabla_p \Kw_{i m j n} - k_{p i} \Hw_{j n m }  + k_{p m} \Hw_{ j n i} \) \nnb \\
& + t^2 \( - k_{p j} \Hw_{i m n} + k_{p n} \Hw_{i m j}\),  \label{id-Dp-W-p} \\
\D_p \W_{i t j t} = {}& \nabla_p \Ew_{i j} +  k_{p}{}^{\!q} \Hw_{i q j} + k_{p}{}^{\!q} \Hw_{j q i }. \label{id-Dp-W-TT} 
\end{align}
Note that, $\nabla$ is the connection associated to the rescaled spatial metric $g_{ij}$.
\end{proposition}
The proof is collected in Appendix \ref{sec-proof-conn}.

\subsection{Lorenzian geometric equations}\label{sec-einstein-eq}
Before presenting the Einstein equations, we recall \eqref{rescale-metric}--\eqref{rescale-2} for the definitions of $g_{ij}$, $\Si_{ij}$ and $\eta$.
In the Gaussian normal gauge, there are the following transport equations
\begin{subequations}
\begin{align}
\mathcal{L}_{\dtau} g_{ij} &=-2 \eta g_{ij} - 2 \Sigma_{ij},  \label{eq-evolution-g}  \\
\p_{\tau} \eta + \eta &= \eta^2 + \frac{1}{n} |\Sigma|^2, \label{eq-evolution-eta}  \\
\mathcal{L}_{\dtau} \Sigma_{ij} + \Sigma_{ij} &=  \Ew_{i  j } - \Sigma_{ip} \Sigma_{j}^p -  \frac{1}{n}   |\Sigma|^2 g_{ij}, \label{eq-evolution-Si}
\end{align}
\end{subequations}
where the vacuum Einstein equations $\breve R_{\mu \nu} = 0$ are employed in \eqref{eq-evolution-eta}--\eqref{eq-evolution-Si}.
We have as well the Gauss--Codazzi equations 
\alis{
R_{imjn}  ={}& -  \frac{1}{2} ( k \odot k )_{imjn} + \bar R_{imjn}, \\
\nabla_i k_{jm} - {}& \nabla_j k_{im} =-\bar R_{ijm \tau},
} 
which, combined with the vacuum Einstein equations, read in terms of the rescaled variables $\eta,\, \Si,\,\Ew, \, \Hw$ and $\Jw$ as below
\begin{subequations}
\begin{align}
R_{imjn} =& -  \frac{1}{2} ( g \odot g )_{imjn} + \Jw_{imjn} - \frac{1}{n-2}  \Ew \odot g  \nnb \\
& +  (1-\eta) \Si \odot g - \frac{1}{2} \eta (\eta -2) g \odot g - \frac{1}{2} \Si \odot \Si, \label{Gauss-Riem-hat-k} \\
\nabla_i \Si_{jm} -\nabla_j \Si_{im} & + \nabla_i \eta g_{jm} - \nabla_j \eta g_{i m} =-\Hw_{ijm}. \label{Codazzi-curl-k}
\end{align}
\end{subequations}
Taking contraction on \eqref{Gauss-Riem-hat-k} leads to
\begin{subequations}
\begin{align}
R_{ij} +  (n-1) g_{ij} &=  \Ew_{ij} + (n-1) ( 2 \eta - \eta^2 ) g_{ij} \nnb \\
& \quad + \Sigma_{ip} \Sigma_{j}^p + (n-2) ( 1 - \eta ) \Sigma_{ij},\label{Gauss-Ricci-hat-k} \\
R +  n(n-1) &= 2 n(n-1)\eta - n(n-1) \eta^2  + \Sigma_{ij} \Sigma^{ij}. \label{Gauss-Ricci-trace-hat -k} 
\end{align}
\end{subequations}
The Codazzi equation \eqref{Codazzi-curl-k} yields an elliptic system for $\Si$ (with $\n \eta$ regarded as a source term)
\begin{subequations}
\begin{align}
\n^j \Sigma_{i j} &= (n-1) \nabla_i \eta,  \label{Codazzi-div-k}  \\
\n_i\Sigma_{j l} - \n_j \Si_{il} &= - \nabla_i \eta g_{jl} + \n_j \eta g_{i l} - \Hw_{ijl}.  \label{div-curl-k-hat}
\end{align}
\end{subequations}
In addition, using \eqref{eq-evolution-Si} and \eqref{Gauss-Ricci-hat-k}, we obtain the following wave equation for $\Si$,
\ali{wave-Sigma}{
\dtau^2 \Sigma_{ij} - \Delta \Sigma_{ij} & =  \Delta \eta g_{i j} -n \nabla_i \nabla_j \eta - (n-1) \dtau \Sigma_{ij} + 2 \Sigma_{ij} + \Jw \ast \Si \nnb \\
& + (\dtau \Sigma, \Sigma, \eta) * (\Sigma, \eta)  +  (\Sigma, \eta) * (\Sigma, \eta) * (\Sigma, \eta).
}
The notation $*$ is defined in the subsection \ref{def-contraction}.

Denote $W_{imjn}$ the Weyl part of $R_{imjn}$, that is,
\[ 
R_{imjn} =  W_{imjn} + \frac{1}{n-2} Ric \odot g - \frac{R}{2(n-1)(n-2)} g\odot g.
\]
Then in terms of $W_{imjn}$, \eqref{Gauss-Riem-hat-k} becomes
\begin{align}
R_{imjn} =& -  \frac{1}{2} ( g \odot g )_{imjn} + W_{imjn} + \frac{1}{n-2}  \Ew \odot g  - \frac{1}{2} \eta (\eta -2) g \odot g \nnb \\
& +  (1-\eta) \Si \odot g  + \frac{1}{ n-2} (\Si \cdot \Si)\odot g - \frac{|\Si|^2}{2(n-1)(n-2)} g\odot g, \label{eq-Riem-W} 
\end{align}
and therefore $W_{imjn}$ and $\Jw_{imjn}$ are related as  follows,
\ali{relat-W-Jw}{
W_{imjn} ={}&   \Jw_{imjn}  - \frac{2}{n-2}  \Ew \odot g  - \frac{1}{2} \Si \odot \Si   \nnb \\
& - \frac{1}{ n-2} (\Si \cdot \Si)\odot g + \frac{|\Si|^2}{2(n-1)(n-2)} g\odot g,
}
where $(\Si \cdot \Si)_{i j} :=\Si_{i }^{l}  \Si_{j l}$.

\subsection{Sobolev norms}\label{sec-Sobolev}

For any $(p, q)$-tensor $\Psi \in T^p_q(M)$, we define  \[|\Psi|_g^2 : =  g^{i_1 j_1} \cdots g^{i_q j_q} g_{i^\prime_1 j^\prime_1} \cdots g_{i^\prime_p j^\prime_p}  \Psi_{i_1\cdots i_q}^{i^\prime_1 \cdots i^\prime_p} \Psi_{j_1 \cdots j_q}^{j^\prime_1 \cdots j^\prime_p}. \] 
In what follows, we also use $|\Psi|$ to denote $|\Psi|_g$ for simplicity. 

Let $H^p_k(M)$ be the Sobolev space of tensors with respect to the norm \[\|\Psi\|_{H^p_k} = \sum_{j=0}^k \left( \int_M |\nabla^j \Psi|^p \di \mu_g \right)^{\frac{1}{p}}. \]
Let $H_{0,k}^p(M)$ be the closure of the space of smooth tensors with compact support in $M$. 
Let $m$ be an integer and denote by $C_B^m(M)$ the space of functions of class $C^m$ for which the norm $$\|\Psi\|_{C^m} = \sum_{j=0}^m \sup_{x \in M} |\nabla^j \Psi (x)|$$ is finite. 


We recall the following Sobolev inequalities \cite{Hebey}. 
\begin{proposition}\label{prop-Sobolev}
Let $(M,\, g)$ be a smooth, complete Riemannian $n$-manifold with Ricci curvature bounded from below. Assume that for any $x \in M$, 
\beq\label{volum-non-collapsing}
\text{Vol}_g (B_x(1)) >\kappa,
\eeq
 where $\kappa>0$ is a positive constant, and $\text{Vol}_g (B_x(1))$ stands for the volume of unit ball centred at $x$, $B_x(1)$, with respect to $g$. 

Let $k >m$ be two integers.

\begin{itemize}
\item  For any $1 \leq q <n$ and $q < p$, $\frac{1}{p} \geq \frac{1}{q} - \frac{k-m}{n}$, $H_k^q (M) \subset H_m^p(M)$.
\item For any $q \geq 1$, if $\frac{1}{q} < \frac{k-m}{n}$, then $H_k^q (M) \subset C^m_B(M)$.
\end{itemize}
\end{proposition}

Based on the Sobolev inequalities, we have the multiplication rules.
\begin{proposition}\label{prop-Multiplication}
For any $p\geq 1$, there are,
 \ali{multiplication}{
 \|f_1 f_2\|_{H^p_k} \lesssim {}& \|f_1\|_{H^p_k} \|f_2\|_{H^p_k}, & \text{if} \,\, k > \frac{n}{p}, \nnb \\
  \|f_1 f_2\|_{H^p_k} \lesssim {}& \|f_1\|_{H^p_{N-l}} \|f_2\|_{H^p_{k+l}}, & \text{if} \,\, N > \frac{n}{p} \,\, \text{and} \,\, 0\leq l \leq N -k.
  }
\end{proposition}
The estimates in Proposition \ref{prop-Multiplication} are in fact particular cases of the following general rules \cite{Palais}. Let $k \leq \min\{ s_1, \, s_2, \, s_1 + s_2-\frac{n}{p} \}$, $s_i \geq0$, $i=1,\,2$, then the following estimate holds \[ \|f_1 f_2\|_{H^p_k} \lesssim \|f_1\|_{H^p_{s_1}} \|f_2\|_{H^p_{s_2}}. \]

For any $\Psi \in  T^p_q (M)$, we define $H^{\prime p}_{ k}(M)$ be the Sobolev space with respect to the norm 
\begin{equation}\label{def-norm-w}
\|\Psi\|_{H^{\prime p}_{ k}} := \sum_{l=0}^k \left( \int_M |\nabla^{\mathring{l}} \Delta^{ [\frac{l}{2}] }  \Psi|^p \di \mu_g \right)^{\frac{1}{p}}, 
\end{equation}
where $\mathring{l}$ is an integer such that 
\begin{equation}\label{def-circ-l}
\mathring{l} =
\begin{cases}
0, &\text{if} \,\, l \,\, \text{is even}, \\
1, &\text{if} \,\, l \,\, \text{is odd}.
\end{cases}
\end{equation}

We will abbreviate $H^{\prime 2}_{k}(M)$ by $H^\prime_{k}(M)$, and $\|\cdot\|_{H^{\prime 2}_{k}}$ by $\|\cdot\|_{H^{\prime}_{k}}$. The following proposition shows the equivalence between $\|\cdot\|_{H_{k}}$ and $\|\cdot\|_{H^\prime_{k}}$.

 \begin{proposition}\label{prop-elliptic-Delta}
Let $(M, \, g)$ satisfy the assumption of Proposition \ref{prop-Sobolev}.
Fix an integer $N > \frac{n}{2}$. Suppose \[ \|R_{imjn} \|_{L^\infty} \quad \text{and} \quad \|\nabla R_{ipjq} \|_{H_{N-1}}\] are bounded. 
  For any $\Psi \in  T^p_q (M)$ with compact support, we have
\[ \|\Psi\|^2_{H_k} \lesssim \|\Psi\|^2_{H^\prime_{k}}, \quad k \leq N+2. \]
\end{proposition}
The proof is collected in Appendix \ref{subsec-equi-norms}.

\subsection{More conventions}

\subsubsection{$M$-Tensors}\label{def-M-tensor}
Let $\Psi_{\alpha_1 \cdots  \alpha_l}$ be a $(0, l)$-tensor on $\mathcal M$ satisfying
\begin{equation*}
\Psi_{\alpha_1 \cdots \alpha_{i-1} \beta \alpha_{i+1} \cdots \alpha_l} \cdot \p_t^\beta =0, \quad \forall \, i \in \{1, \cdots, l\},
\end{equation*}
where if $i=1$ or $l$, then $\alpha_0 $ and $\alpha_{l+1}$ are interpreted as being absent.
We can restrict $\Psi$ on $M_t :=\{t\} \times M$, and naturally interpret it as a tensor field on the Riemannian manifold $(M_t,\, g)$.

\subsubsection{Multi index} For notational convenience, we use $\nabla_{I_l} \Psi$ to denote the $l^{\text{th}}$ order covariant derivative $\nabla_{i_1} \cdots \nabla_{i_l}\Psi$ where the multi index $I_l=\{ i_1 \cdots i_l\}$ is used. 

\subsubsection{Index brackets}\label{sec-bracket} Round and square brackets on tensor indices are employed to identify the
symmetric and anti-symmetric, respectively, components of a tensor. For example\footnote{If $\Psi_{ijk} = -\Psi_{j i k}$, then $\Psi_{[ijk]}$ reduces to $\Psi_{[i j k]} = \frac{1}{3} \( \Psi_{ijk} + \Psi_{j k i} + \Psi_{k i j} \)$.} \[ \Phi_{(ij)} := \frac{1}{2} \( \Phi_{ij} + \Phi_{j i} \), \quad \Psi_{[i j k]} := \frac{1}{6} \( \Psi_{ijk} + \Psi_{j k i} + \Psi_{k i j} - \Psi_{i k j} - \Psi_{k j i} - \Psi_{j i k} \). \]

\subsubsection{Contractions}\label{def-contraction} Unless indicated otherwise, we use the metric $g_{i j}$ and its inverse to raise and lower indices. 
Throughout, we use $A*B$ to denote a linear combination of products of $A$ and $B$, with each product being a contraction (with respect to $g$) between the two $M$-tensors $A$ and $B$. We also use the notation $(A, B) \ast (D, E)$ to denote all possible linear combinations $A\ast D + B \ast D + A \ast E + B \ast E$. Similar rule applies to $(A, B, \cdots) \ast (D, E, \cdots)$ as well. Besides, we use $\pm A$ to denote any linear combinations of $A$.
In the estimates, we will only employ the formula $\|A * B\| \lesssim \|A\| \|B\| $ which allows ourselves to ignore the detailed product structure at this point. $\|\cdot \|$ denotes some Sobolev norm associated to $g$. 

\subsubsection{Some abbreviations} We always use $\Gm$ to denote $\Ew$ or $\Hw$ (defined in \eqref{E:def-1+n-Weyl}) for simplicity, if there is no need to distinguish them.

\subsubsection{Some constants} We let $C$ denote some universal constant which may vary from line to line. We use $I_k$ to denote constants depending on the initial energy (not on $\vep$). The constant $k$ indicates the number of derivatives used in the energy norms.

\section{The main equations}\label{sec-Main-eq}

In this section, we will derive the main equations for energy estimates.
\subsection{Bianchi equations}
To begin with, we collect the $1+n$ Bianchi equations which eventually lead to hyperbolic or elliptic systems when taking the vacuum Einstein equations into account.
Contracting the following Bianchi equations
\alis{
\D_\mu \Rmb_{i t \nu t} + \D_i \Rmb_{t \mu \nu t} + \D_t \Rmb_{\mu i \nu t} ={}& 0, \\
\D_\nu \Rmb_{t i \mu j} + \D_t \Rmb_{i \nu \mu j} + \D_i \Rmb_{\nu t \mu j} ={}& 0, \\
\D_\mu \Rmb_{i j \nu p} + \D_i \Rmb_{j \mu \nu p} + \D_j \Rmb_{\mu i \nu p} ={}& 0, 
}
  with $\breve g^{\mu \nu}$ leads to 
\alis{
\D^\mu \Rmb_{i t \mu t} - \D_i \Rmb_{t t} + \D_t \Rmb_{i t} = {}& 0, \\
\D^\mu \Rmb_{t i \mu j} - \D_t \Rmb_{i j} + \D_i \Rmb_{t j} ={}& 0, \\
\D^\mu \Rmb_{i j \mu p} - \D_i \Rmb_{j p} + \D_j \Rmb_{i p} ={}& 0.
} 
Consequently, combining with the vacuum Einstein equations, we obtain
\begin{align*}
\D^\mu \W_{i t \mu t} =0, 
\quad \D^\mu \W_{i t \mu j} = 0, \quad \D^\mu \W_{j i \mu p}=0.
\end{align*}
For the $1+n$ splitting, we define the projection to be 
\beq\label{def-projection}
\breve h^{\mu \nu} := \breve g^{\mu \nu} +  \p_t^\mu \p_t^\nu. 
\eeq
It then follows that
\begin{align*}
 \breve h^{\mu \nu}  \D_\mu \W_{i t \nu t} - \D_t \W_{i t t t}  &=0,\\
 \breve h^{\mu \nu}  \D_\mu \W_{t i \nu j} -  \D_t \W_{t i t j} &=0, \\
 \breve h^{\mu \nu}  \D_\mu \W_{j i \nu p} -  \D_t \W_{j i t p} &=0.
\end{align*}
Since $\D_t  \p_t =0$, and hence $\D_t \W_{i t t t} = 0$, and $\D_t \W_{j i t t} = 0$, the above identities reduce to
\begin{align}
 \breve h^{\mu \nu} \D_\mu \W_{i t \nu t} &=0,  \label{Bianchi-TT} \\
-  \D_t \W_{t i t q} +   \breve h^{\mu \nu} \D_\mu \W_{t i \nu q}  &=0,  \label{Bianchi-T-1} \\ 
-  \D_t \W_{j i t q} +  \breve h^{\mu \nu}  \D_\mu \W_{j i \nu q}  &=0.  \label{Bianchi-no-T}
\end{align}
Besides, we present the other Bianchi equations
\begin{align}
  \D_t \W_{p i q t} + \D_p \W_{i t q t} - \D_i \W_{p t q t} &=0, \label{Bianchi-TT-dual} \\
\D_t \W_{i p q j} + \D_i \W_{p t q j} + \D_p \W_{t i q j} &=0, \label{Bianchi-extra} \\
\D_p \W_{i j q l} + \D_i \W_{j p q l} + \D_j \W_{p i q l} &=0.\label{Bianchi-no-T-dual}
\end{align}

Based on these Bianchi equations, we will derive a couple of elliptic or transport systems for $\Ew$, $\Hw$, $\Jw$ or $W$.
More precisely, we find that \eqref{Bianchi-T-1} and \eqref{Bianchi-TT-dual} constitute a hyperbolic system for $\Ew$, $\Hw$; 
\eqref{Bianchi-no-T} and \eqref{Bianchi-extra} yield a transport system for $\Hw$ and  $\Jw$, see Lemmas \ref{lem-trans-Hw-Ew} and \ref{lem-trans-Hw-Jw}.
In addition,
\eqref{Bianchi-no-T} and \eqref{Bianchi-TT-dual} imply the following equation
\begin{equation}\label{duality-TT-noT} 
\D_p \W_{i t j t} - \D_i \W_{p t j t} =  \breve h^{\mu \nu} \D_\mu \W_{p i \nu j}.
\end{equation}
Then \eqref{Bianchi-no-T-dual}--\eqref{duality-TT-noT} leads to an elliptic system for $\Jw$, while
\eqref{duality-TT-noT} together with \eqref{Bianchi-TT} composes an elliptic system for $\Ew$, see Lemma \ref{lem-div-E-curl}.

\subsection{Elliptic equations}
\bg{lemma}\label{lem-div-E-curl}
The spatial Weyl tensor $W$ satisfies the following elliptic system,
\begin{subequations}
\begin{align}
\n^l W_{ljpi}={}& \pm \nabla \Ew \pm \n \Si \pm  \n  \eta  + \n \Si \ast (\Si, \, \eta) + \n \eta \ast (\Si, \, \eta), \label{eq-div-W} \\
  \nabla_{[p} W_{i m] j n} ={}& \pm \n \Ew \pm \n \Si \pm \n \eta +  \n \Si \ast (\Si, \, \eta) + \n \eta \ast (\Si, \, \eta). \label{eq-curl-W}
\end{align}
\end{subequations}
The symmetric and trace-free tensor $\Ew$ satisfies the following elliptic system as well,
\begin{subequations}
\begin{align}
 \nabla^p \Ew_{j p} ={} & - \Si^{q p} \Hw_{j q p}, \label{eq-div-Ew} \\
 \nabla_p \Ew_{i  j } - \nabla_i  \Ew_{p  j } = {}&  \frac{n-2}{n-3} \nabla^l \Jw_{l j p i}  + (n-2) \Hw_{p i j} \nnb \\
-{}& \frac{1}{n-3} \(   g_{ij} \Si^{q l} \Hw_{p q l} - g_{p j} \Si^{q l} \Hw_{i q l} \)\nnb \\
+{}&  \frac{n-2}{n-3}  \( \Si^{l}_{j} \Hw_{p i l} + \Si_{p}{}^{\!l} \Hw_{l i j} + \Si_{i}{}^{\!l} \Hw_{p l j} -n \eta \Hw_{p i j} \). \label{eq-div-Jw}
\end{align}
\end{subequations}
\ed{lemma}
\bg{remark}
The system \eqref{Bianchi-no-T-dual}--\eqref{duality-TT-noT} leads to an elliptic system for $\Jw$:
\begin{align*}
  \nabla^l \Jw_{l j p i} ={}&  \nabla \Ew \pm \Hw  + ( \Si, \eta) \ast \Hw, \\
 \nabla_{[p} \Jw_{i m] j n} ={}&  \nabla \Ew \pm \Hw  + ( \Si, \eta) \ast \Hw.
\end{align*}
As a further remark, by \eqref{div-curl-k-hat}, $\Hw = \n \Ew \pm \n \eta$ and the above system for $\Jw$ is equivalent to the elliptic system of $W$.
\ed{remark}

\bpf[Proof of Lemma \ref{lem-div-E-curl}]
Due to the Bianchi identity \[ \n_{[p} R_{im]jn} = 0,\] \eqref{eq-Riem-W} implies \eqref{eq-curl-W}.
Taking contraction on \eqref{eq-curl-W}, we derive \eqref{eq-div-W}.

By \eqref{Bianchi-TT} and Proposition \ref{prop-decomp-Bianchi},
\[
0=\breve h^{\mu \nu} \D_\mu \W_{i t \nu t}= \breve h^{p q} \D_p \W_{i t q t} \stackrel{\eqref{id-Dp-W-TT}}{=}  t^{-2} \( \nabla^p \Ew_{i p} + k^{p l} \Hw_{i l p} \) =0.
\]
In view of the fact that $\Hw$ is trace-free,  \eqref{eq-div-Ew} follows immediately.


Consider the identity \eqref{duality-TT-noT} and 
make use of the computing identities \eqref{id-Dp-W-p} and \eqref{id-Dp-W-TT} (in Proposition \ref{prop-decomp-Bianchi}),
\begin{align*}
 \breve h^{\mu \nu} \D_\mu \W_{p i \nu j}  &=  \nabla^l \Kw_{p i l j} - k^{l}_{p} \Hw_{l j i }  + k^{l}_{i} \Hw_{ l j p} - \tr  k \Hw_{p i j}  + k^{l}_{j} \Hw_{p i l},\\
\D_p \W_{i t j t} - &\D_i \W_{p t j t} = \nabla_p \Ew_{i  j } - \nabla_i  \Ew_{p  j }  \\
&+ k_{p}{}^{\!l} \Hw_{i l j} + k_{p}{}^{\!l} \Hw_{j l i} - k_{i}{}^{\!l} \Hw_{p l j} - k_{i}{}^{\!l} \Hw_{j l p}.
\end{align*}
Therefore the identity \eqref{duality-TT-noT} yields
\begin{align*}
& \nabla_p \Ew_{i  j } - \nabla_i  \Ew_{p  j } \\
= {}& \nabla^l \Kw_{p i l j} - \tr  k \Hw_{p i j}  + k^{l}_{j} \Hw_{p i l} + k_{p}{}^{\!l} \Hw_{l i j} + k_{i}{}^{\!l} \Hw_{p l j}. 
\end{align*}
Next, we intend to reformulate $\nabla^l \Kw_{p i l j} $ in terms of $\nabla^l \Jw_{p i l j}$.
By virtue of the definition for $\Jw$ \eqref{def-Jw},
\alis{
\nabla^q \Kw_{qjpi} = \nabla^q \Jw_{qjpi} + \frac{1}{n-2} \( \nabla_p \Ew_{ij} - \nabla_i \Ew_{pj} + g_{ij} \nabla^q \Ew_{pq} - g_{pj} \nabla^q \Ew_{qi} \),
} 
we obtain
\begin{align*}
& \nabla_p \Ew_{i  j } - \nabla_i  \Ew_{p  j } \\
= {}&  \nabla^l \Jw_{l j p i} + \frac{1}{n-2} \( \nabla_p \Ew_{i j} - \nabla_i \Ew_{pj} + g_{ij} \nabla^l \Ew_{p l} - g_{p j} \nabla^l \Ew_{l i} \) \\
& - \tr  k \Hw_{p i j} - 3 \Hw_{p i j} + \Si^{l}_{j} \Hw_{p i l} + \Si_{p}{}^{\!l} \Hw_{l i j} + \Si_{i}{}^{\!l} \Hw_{p l j} . 
\end{align*}
As a consequence,
\begin{align*}
&  \frac{n-3}{n-2} \( \nabla_p \Ew_{i  j } - \nabla_i  \Ew_{p  j } \) \\
= {}&  \nabla^l \Jw_{l j p i} + \frac{1}{n-2} \(   g_{ij} \nabla^l \Ew_{p l} - g_{p j} \nabla^l \Ew_{l i} \) \\
& + (n-3) \Hw_{p i j}  + \Si^{l}_{j} \Hw_{p i l} + \Si_{p}{}^{\!l} \Hw_{l i j} + \Si_{i}{}^{\!l} \Hw_{p l j} -n \eta \Hw_{p i j}. 
\end{align*}
Then \eqref{eq-div-Jw} follows from substituting \eqref{eq-div-Ew} into the above identity.

\epf

\subsection{Hyperbolic systems of Maxwell type}

In this subsection, we employ the Bianchi equations \eqref{Bianchi-T-1}--\eqref{Bianchi-extra} to derive the two hyperbolic systems of Maxwell type for $(\Ew, \,\Hw)$ and $(\Hw, \, \Jw)$.

\bg{lemma}\label{lem-trans-Hw-Ew}
The Bianchi equations \eqref{Bianchi-T-1} and \eqref{Bianchi-TT-dual} indicate the following system for $(\Ew, \,\, \Hw)$
\begin{subequations}
 \begin{align}
 & \lie_{\dtau} \Ew_{i  j} +(n-2)\Ew_{ij} + \nabla^p \Hw_{p (i j)} \nnb \\
={}&  \Si^{p q} W_{p i q j} + \(\Si, \, \eta \) \ast \Ew + \Si \ast \Si \ast \Si, \label{trans-EHw-Ew} \\
  &\lie_{\dtau} \Hw_{p i j} + \Hw_{p i j} +   \nabla_p \Ew_{i  j } - \nabla_i  \Ew_{p  j } \nnb \\
  ={}& \Si\ast \Hw. \label{trans-EHw-Hw} 
\end{align}
\end{subequations}
\ed{lemma}

\bpf

To derive \eqref{trans-EHw-Ew}, we appeal to the Bianchi identity \eqref{Bianchi-T-1}. 
In more details, we use the computing identities \eqref{id-Dp-W-T} and \eqref{id-DT-W-TT} (in Proposition \ref{prop-decomp-Bianchi}) to replace $\D^p \W_{t i p j}$ and $ \D_t \W_{t i t j}$ in \eqref{Bianchi-T-1} respectively. It then follows that
 \alis{
 & t^{-1} \( \lie_{\dtau} \Ew_{i j} + k_i^p \Ew_{p j} + k_j^p \Ew_{i p }\) \\
  & + t^{-1} \nabla^q \Hw_{q j i} + t^{-1} \( k^{p q} \Kw_{i p q j} - \tr k\Ew_{i j} + k^q_{ j} \Ew_{ i q}\)=0.
 }
 That is,
\alis
{& \lie_{\dtau} \Ew_{i  j}  + \nabla^p \Hw_{p j i} - k^{p q} \Kw_{p i q j}  - \tr  k \Ew_{i j} + 2 k_{p j} \Ew^p_{i} + k_{i}^{p}  \Ew_{p j}  =0.
}
 Noting that \[ k_{pq}=-g_{pq}+\Si_{pq}+\eta g_{pq} \] and taking the symmetric part, we obtain, 
\alis
{& \lie_{\dtau} \Ew_{i  j} +(n-2)\Ew_{ij} + \nabla^p \Hw_{p (i j)} \nnb \\
& - \Si^{p q} \Kw_{p i q j} - (n-2) \eta \Ew_{ij} + 3  \Si_{(i}^{p}  \Ew_{j) p} =0.
}
In addition, we use \eqref{def-Jw} to substitute $\Kw$ by $\Jw$, then the above equation reads
\begin{align}\label{eq-pt-Ew-Hw}
& \lie_{\dtau} \Ew_{i  j} +(n-2) \Ew_{ij} + \nabla^p \Hw_{p (i j)} \nnb \\
& - \Si^{p q} \Jw_{p i q j} - (n-2) \eta \Ew_{ij} \nnb \\
& -  \frac{1}{n-2} \Si^{p q} \Ew_{p q} g_{i j}  + \( 3 + \frac{2}{n-2} \)  \Si_{(i}^{p}  \Ew_{j) p} =0.
\end{align}
Furthermore, using \eqref{E:relat-W-Jw} to re-express $\Jw$ in terms of $W$, we arrive at \eqref{trans-EHw-Ew}.

As for \eqref{trans-EHw-Hw}, we turn to the Bianchi equation \eqref{Bianchi-TT-dual}. 
Making use of the computing identities \eqref{id-DT-W-T} and \eqref{id-Dp-W-TT}  (in Proposition \ref{prop-decomp-Bianchi}),  we rewrite \eqref{Bianchi-TT-dual} as
\[ 0= \nabla_p \Ew_{i  j } - \nabla_i  \Ew_{p  j }  + \lie_{\dtau} \Hw_{p i j} + \Hw_{p i j} + k_j^l  \Hw_{p i l} + k_{p}^{l} \Hw_{j l i} + k_i^l \Hw_{ljp}. \]
With the help of the algebra identity $\Hw_{[ijk]}=0$, the above equation reduces to
\beq\label{eq-tr-H}
  \lie_{\dtau} \Hw_{p i j} + \Hw_{p i j} +   \nabla_p \Ew_{i  j } - \nabla_i  \Ew_{p  j } = - \(  \Si_j^l  \Hw_{p i l} + \Si_{p}^{l} \Hw_{j l i} + \Si_i^l \Hw_{ljp} \),
\eeq
which is further abbreviated to  \eqref{trans-EHw-Hw}.
\epf

Similar to the case for $\( \Ew, \, \Hw \)$, we have the following lemma for $(\Hw, \, \Jw)$  as well.
\bg{lemma}\label{lem-trans-Hw-Jw}
We infer from the Bianchi equations \eqref{Bianchi-no-T} and \eqref{Bianchi-extra} the following transport system for $(\Hw, \, \Jw)$
\begin{subequations}
\begin{align}
 & \lie_{\dtau} \Hw_{p i j} + (n-1) \Hw_{p i j} +  \frac{n-2}{n-3} \nabla^l \Jw_{l j p i} \nnb \\
  = {} & ( \Si, \eta) \ast \Hw,  \label{trans-HJw-Hw} \\
&  \lie_{\dtau} \Jw_{ipjq} + \nabla_i \Hw_{j q p} - \nabla_p \Hw_{j q i} -  \frac{1}{n-2}\(\dive \Hw \odot g\)_{ipjq}  \nnb \\
 = {}&  (\Si, \eta) \ast \Jw  + (\Si, \eta) \ast \Ew, \label{trans-HJw-Jw}
\end{align}
\end{subequations}
where we set \[ \dive \Hw_{(i j)} :=\nabla^l \Hw_{l (ji) }. \]  
\ed{lemma}

\bg{remark}
Due to the relation between $\Jw$ and $W$ \eqref{E:relat-W-Jw}, \eqref{trans-HJw-Jw} (or more precisely \eqref{pt-HJw-Jw}) suggests a transport equation for $W$,
\begin{align}\label{eq-pt-W}
&  \lie_{\dtau} W_{ipjq} + \nabla_i \Hw_{j q p} - \nabla_p \Hw_{j q i} -  \frac{3}{n-2}(\dive \Hw \odot g)_{ipjq} - 2 \( \Ew \odot g \)_{ipjq} \nnb \\
 = {}& \Si \ast W  + (\Si, \, \eta) \ast \Si + (\Si, \, \eta) \ast \Ew + (\Si, \, \eta) \ast \Si \ast \Si.
\end{align}
\ed{remark}
\bg{remark}
The two systems of $(\Ew,\, \Hw)$ and $(\Hw, \, \Jw)$ in Lemmas \ref{lem-trans-Hw-Ew} and \ref{lem-trans-Hw-Jw} respectively are never independent to each other. For instance, \eqref{trans-HJw-Hw} together with \eqref{eq-div-Jw} implies \eqref{trans-EHw-Hw}.
\ed{remark}

\bpf[Proof of Lemma \ref{lem-trans-Hw-Jw}]
For \eqref{trans-HJw-Hw}, we combine \eqref{eq-div-Jw} with \eqref{eq-tr-H} to derive
\begin{align}
 & \lie_{\dtau} \Hw_{p i j} + (n-1) \Hw_{p i j} +  \frac{n-2}{n-3} \nabla^l \Jw_{l j p i}  \nnb\\
  = {} & \frac{1}{n-3} \(   g_{ij} \Si^{q l} \Hw_{p q l} - g_{p j} \Si^{q l} \Hw_{i q l} \) +  \frac{(n-2)n}{n-3}  \eta \Hw_{p i j} \nnb \\
&- 2 \Si^{l}_{j} \Hw_{p i l} +  \Si_{p}{}^{\!l} \Hw_{i j l} - \Si_{i}{}^{\!l} \Hw_{p j l}  \nnb \\
& +  \frac{1}{n-3} \( \Si_{i}{}^{\!l} \Hw_{l p j} - \Si_{p}{}^{\!l} \Hw_{l i j} +  \Si^{l}_{j} \Hw_{p i l}  \). \label{pt-HJw-Hw}
\end{align}
This equation is simplified as \eqref{trans-HJw-Hw} when the detailed product structure is ignored.
Alternatively, \eqref{trans-HJw-Hw} can be inferred from the $1+n$ decomposing of the Bianchi identity \eqref{Bianchi-no-T}.

To derive \eqref{trans-HJw-Jw}, we turn to the Bianchi identity \eqref{Bianchi-extra}. 
With the help of the calculations \eqref{id-DT-W-p} and \eqref{id-Dp-W-T} (in Proposition \ref{prop-decomp-Bianchi}), \eqref{Bianchi-extra} becomes
\alis {
&  t  \lie_{\dtau} \Kw_{i p j q}  + 2 t \Kw_{i p j q} + t \( \nabla_i \Hw_{j q p} - \nabla_p \Hw_{j q i} \)  \\
& + t \( k_{i}^l \Kw_{l p j q} + k_{p}^l \Kw_{i l j q} + k_{j}^l \Kw_{i p l q} + k_{q}^l \Kw_{i p j l} \) \\
&+ t \( k_{i}^l \Kw_{p l j q} - k_{i j} \Ew_{p q} + k_{i q} \Ew_{p j} \) \\
&+ t \( - k_{p}^l \Kw_{i l j q} + k_{p j} \Ew_{i q} - k_{p q} \Ew_{i j} \) =0,
}
which, via
$k_{pq}=-g_{pq}+\Si_{pq}+\eta g_{pq}$, can be further turned into
\ali{eq-transport-K}
{& \lie_{\dtau} \Kw_{i p j q}  + \nabla_i \Hw_{j q p} - \nabla_p \Hw_{j q i} \nnb \\
&+ \Si_{j}^l \Kw_{i p l q} + \Si_{q}^l \Kw_{i p j l} +2\eta \Kw_{ipjq} \nnb \\
& + \(g \odot \Ew \)_{i p j q} - \(\Si \odot \Ew \)_{i p j q} - \eta \(g \odot \Ew \)_{i p j q}
=0.
}
In what follows, we mean to transform \eqref{E:eq-transport-K} into a transport equation for $\Jw$ through
\alis{
\lie_{\dtau} \Jw_{ipjq} ={}& \lie_{\dtau} \Kw_{ipjq} - \frac{1}{n-2} ( \Ew \odot \lie_{\dtau} g)_{i p j q} - \frac{1}{n-2} ( \lie_{\dtau} \Ew \odot g)_{i p j q},
} 
which is suggested by  the relation \eqref{def-Jw}.
As a further step, we make use of \eqref{eq-evolution-g} and \eqref{eq-pt-Ew-Hw} to arrive at
\begin{align}
& \lie_{ \dtau} \Jw_{ipjq} + \nabla_i \Hw_{j q p} - \nabla_p \Hw_{j q i} -  \frac{1}{n-2}(\dive \Hw \odot g)_{ipjq}  \nnb \\
 = {}& - \Si_{j}^l \Jw_{i p l q} + \Si_{q}^l \Jw_{i p l j} - 2\eta \Jw_{ipjq} - \frac{1}{n-2} \( (\Si \cdot \Jw) \odot g \)_{i p j q}  \nnb \\
 & + \frac{2}{n-2} \eta ( \Ew \odot g)_{i p j q} + \frac{n}{n-2}  ( \Ew \odot \Si)_{i p j q} \nnb \\
&- \frac{1}{(n-2)^2} \Si^{kl} \Ew_{kl} ( g \odot g)_{i p j q}  +  \( \frac{3}{n-2} + \frac{2}{(n-2)^2} \)  \( (\Si \cdot \Ew  ) \odot g \)_{ipjq}. \label{pt-HJw-Jw}
\end{align}
Here we set $\dive \Hw_{(i j)} :=\nabla^l \Hw_{l (ji) }$,  and \[ (\Si \cdot \Jw )_{i j} := \Si^{k l}\Jw_{i k j l}, \quad (\Si \cdot \Ew)_{i j} :=\Si_{( i }^{l}  \Ew_{j) l}. \] In particular, in \eqref{pt-HJw-Jw}, \[ ( \dive \Hw \odot g)_{ipjq}= \nabla^l \Hw_{l (ji)} g_{pq} - \nabla^l \Hw_{l (jp)} g_{qi} + g_{ij} \nabla^l \Hw_{l (qp)} - g_{pj} \nabla^l \Hw_{l (iq)}. \]

The equations \eqref{pt-HJw-Hw} and \eqref{pt-HJw-Jw} lead to the conclusion of this lemma.
\epf

\section{Energy estimates}\label{sec-Energy-estimates}
In this section, we will prove the following theorem.
\begin{theorem}\label{thm-EE}
Suppose the background Einstein space $(M, \, \ga)$ and the initial data $(M, \, g_0, \, k_0)$ satisfy the assumptions in Theorem \ref{thm-global-existence-Ein}. 
Then along with the Einstein flow, the solution $(M, \, g(t))$ exists for all $t \in [t_0, + \infty)$ and we have the following estimates
\alis{
& t^{1-\delta} \( \|\Sigma_{ij} \|_{H_{N+1}} + \| \Ew \|_{H_N} + \|\Hw \|_{H_N} \) +  t \|\eta \|_{H_{N+2}}  \lesssim  \varepsilon, \\
& \|g_{ij} - \gamma_{ij} \|_{H_{N+2}} + \|W - W[\gamma] \|_{H_{N}}  \lesssim \varepsilon.
}

\end{theorem}

\subsection{Bootstrap assumptions}
Recall the fixed numbers $0<\delta < \frac{1}{6}$ and $N > \frac{n}{2}$.
We start with the following weak assumptions: Suppose  $\Lambda$ is a large constant to be determined, and
\begin{align}
t\| \Si \|_{H_{N+1}} + t \|\eta\|_{H_{N+2}}  \leq \vep \Ld t^{\de}, &\quad \| \Si \|_{C^0} + \| \eta \|_{C^0}\leq \vep \Ld,  \label{bt-Si-eta} \\
 t\|\Ew \|_{H_N} + t\|\Hw \|_{H_N}  \leq \vep \Ld t^{\de}, & \quad \| \Ew \|_{C^0}\leq \vep \Ld,  \label{bt-Ew-Hw} \\
 \|g_{ij} - \gamma_{i j}\|_{H_{N+2}} \leq \vep \Ld, & \quad \| g_{ij} - \gamma_{i j} \|_{C^1} \leq \varepsilon \Lambda, \label{bt-g} \\
 \|W \|_{H_N}  & \leq \Ld. \label{bt-Jw}
\end{align}
We will improve these bootstrap assumptions by showing that \eqref{bt-Si-eta}--\eqref{bt-Jw} implies the same inequalities hold with the constant $\Lambda$ replaced by $\frac{1}{2} \Lambda$.
\bg{remark}
The data assumption \eqref{intro-initial-data} tells
\[ \| R[g_0]_{im}{}^{j}{}_{n} - R[\ga]_{im}{}^{j}{}_{n} \| ^2_{H_N(M,\, g_0)}  \lesssim \varepsilon^2, \] and thus
\alis{
 \| R[g_0]_{ij} + (n-1) g_{0ij} \|^2_{H_N(M, \, g_0)} + \|W[g_0] - W[\ga] \|^2_{H_N(M,\, g_0)}  \lesssim \varepsilon^2.
 }
Moreover, with the help of Gauss-Codazzi equations \eqref{Codazzi-curl-k}--\eqref{Gauss-Ricci-hat-k}, it follows that
\[ \|\Ew[g_0]  \|^2_{H_N(M,\, g_0)} + \|\Hw[g_0]  \|^2_{H_N(M,\, g_0)}  \lesssim \varepsilon^2. \] In the end, we infer from the relation between $W$ and $\Jw$ \eqref{E:relat-W-Jw} that
\[ \|\Jw[g_0] - W[g_0]  \|^2_{H_N(M,\, g_0)}  \lesssim \varepsilon^2. \]
Note that, $\Ew[g_0]$, $\Hw[g_0]$ and $\Jw[g_0]$ denote the initial values of $\Ew,\,\Hw,\,\Jw$ on the initial hypersurface. All the above estimates show that the bootstrap assumptions \eqref{bt-Si-eta}--\eqref{bt-Jw} hold initially.

For notational simplicity, we will denote 
\beq\label{def-Jw0-W0}
 \Jw_0 := \Jw[g_0], \quad W_0 : =W[g_0].
\eeq
\ed{remark}

As the large constant $\Ld$ is independent of $\varepsilon$, for $\vep>0$ small enough, \eqref{bt-g} implies that 
$g$ and $\gamma$ are equivalent as bilinear forms, and $g$ is close to $\gamma$. Therefore,  the spatial manifold $\(M, \,g\)$ satisfies the volume non-collapsing condition \eqref{volum-non-collapsing}, since $\(M, \, \ga\)$ has positive injective radius and hence is volume non-collapsing. 
Moreover, under the bootstrap assumptions \eqref{bt-Si-eta}--\eqref{bt-Jw}, we know from the Gauss equation \eqref{Gauss-Ricci-hat-k} that the Ricci tensor of $g_{ij}$ is bounded from below.
Therefore, by Proposition \ref{prop-Sobolev}, there are uniform Sobolev inequalities on $(M, g)$ under the bootstrap assumptions \eqref{bt-Si-eta}--\eqref{bt-Jw}.  
Furthermore, due to \eqref{bt-g}, the density theorem follows from Proposition \ref{pro-density}.
\begin{corollary}\label{coro-density} 
Suppose the bootstrap assumptions  \eqref{bt-Si-eta}--\eqref{bt-Jw} hold, then we have
\alis{
H_{0, k}(M) &= H_{k}(M), \quad k\leq N+2, \,\, N > \frac{n}{2},
} 
where $H_{0, k}(M)$, $H_{k}(M)$ are defined in Section \ref{sec-Sobolev}.
\end{corollary}

Under the bootstrap assumptions \eqref{bt-Si-eta}--\eqref{bt-Jw}, we know from \eqref{eq-Riem-W} that the Riemann tensor of $g$ has the following bound\[ \| R_{imjn} + \frac{1}{2} (g \odot g)_{imjn} \|_{H_{N}} < C, \] and hence
\ali{Reim-bound}{
 \|R_{imjn}\|_{L^\infty} + \|\nabla R_{imjn} \|_{H_{N-1}} < C, \quad N > \frac{n}{2}.
}
As a consequence of Corollary \ref{coro-density} and the boundness of Riemann tensor \eqref{E:Reim-bound}, we have
\begin{corollary}\label{prop-elliptic-Delta-1}
The conclusion of Proposition \ref{prop-elliptic-Delta} holds for any $\Psi \in H_{k}(M)$, $k \leq N+2$, provided the bootstrap assumptions \eqref{bt-Si-eta}--\eqref{bt-Jw}.
\end{corollary}

\subsection{Estimates (without the top order derivative) for $ \eta $}
\begin{lemma}\label{prop-energy-estimate-eta-low}
Under the bootstrap assumptions \eqref{bt-Si-eta}--\eqref{bt-Jw}, we have
\[ t^2 \|\eta\|^2_{H_{k}}   \lesssim \varepsilon^2  I_{k+1}^2 + \varepsilon^4 \Lambda^4, \quad k \leq N+1. \]
\end{lemma} 

\begin{proof}
The transport equation of $\eta$ \eqref{eq-evolution-eta} involves only terms with the same regularity as $\eta$, like $\eta^2$, $|\Sigma|^2$,  on the right side.
This structure costs no extra regularities in the estimates. In fact, we deduce for $k \leq N+1$, 
\alis{ 
\p_{\tau} \|\eta\|_{H_{k}}^2 +2 \|\eta\|_{H_{k}}^2 & \lesssim \left( \|\eta\|_{L^\infty} +  \|\Sigma\|_{L^\infty} \right) \|\eta\|_{H_{k}}^2 \\
&+\left(  \|\eta^2\|_{H_{N}}  +  \|\Sigma * \Sigma\|_{H_{N}}   \right)  \|\eta\|_{H_{k}}.
}
With the bootstrap assumptions \eqref{bt-Si-eta}--\eqref{bt-Jw}, we obtain
\alis{ 
\p_{t} (t^2 \|\eta\|_{H_{k}}^2)  & \lesssim \varepsilon \Lambda t^{-2+ \delta}  \cdot t^2 \|\eta\|^2_{H_{k}}+   \varepsilon^2 \Lambda^2 t^{-2+2 \delta}   \cdot t \|\eta\|_{H_{k}},
}
which concludes this lemma by the Gr\"{o}nwall's inequality.
\end{proof}

\subsection{ Estimates for $W$}
\subsubsection{Zero-order estimates for $W$ and $W-W[\ga]$}

\begin{corollary}\label{coro-Jw-W-0} 
Under the bootstrap assumptions \eqref{bt-Si-eta}--\eqref{bt-Jw}, the following estimates hold
\begin{subequations}
\begin{align}
  \|W \|^2 \lesssim {}&  I_2^2 +   \vep^2 \Ld^2,   \label{ee-W-0} \\
 \|W  - W[\ga]\|^2 \lesssim {}& \vep^2 I_2^2 + \vep^2 \Ld^2. \label{ee-W-W0-0}
\end{align}
\end{subequations}
\end{corollary} 

\bpf
We will first verify the following estimates for $\Jw$
\begin{subequations}
\begin{align}
 \|\Jw \|^2 \lesssim {}&  I_2^2 +   \vep^2 \Ld^2,   \label{ee-Jw-0}\\
  \|\Jw - \Jw_0\|^2 \lesssim {}& \vep^2 I_2^2 + \vep^2 \Ld^2, \label{ee-Jw-Jw0-0}
\end{align}
\end{subequations}
which then indicates \eqref{ee-W-0}--\eqref{ee-W-W0-0} with the help of \eqref{E:relat-W-Jw}.

We multiply $\frac{4(n-3)}{n-2} \Hw$ and $2\Jw$ on both sides of the transport system for $(\Hw, \, \Jw)$ \eqref{trans-HJw-Hw}--\eqref{trans-HJw-Jw}, and notice that, 
\alis{
 & \frac{n-2}{n-3} \frac{4(n-3)}{n-2}  \nabla^l \Jw_{l j p i} \Hw^{pij} + 2 \( \nabla_i \Hw_{j q p} - \nabla_p \Hw_{j q i} \) \Jw^{ipjq} \\
 & =  4 \nabla^l \Jw_{l j p i} \Hw^{pij} + 4 \nabla_i \Hw_{j q p}  \Jw^{ipjq} = 4 \nabla_i \( \Hw_{j q p}  \Jw^{ipjq} \),
}
and since $\Jw_{ipjq}$ is a Weyl tensor
\[  -  \frac{1}{n-2}(\dive \Hw \odot g)_{ipjq} \cdot \Jw^{ipjq} = 0. \]
Putting all these together leads to
\alis{
& \p_t \( \|\Jw\|^2 + \frac{2(n-3)}{n-2}  \|\Hw\|^2 \) + (n-1) \frac{4(n-3)}{n-2} t^{-1} \|\Hw\|^2 \\
\lesssim {}& t^{-1} \( \|\Si\|_{L^\infty} +n \|\eta\|_{L^\infty} \)  \(\|\Hw\|^2 +  \|\Jw\|^2 \) +  t^{-1} \|\Si\|_{L^\infty} \| \Ew \| \|\Jw\|   \\ 
\lesssim {}& \vep \Ld t^{-2+\de} \( \|\Hw\|^2 + \|\Jw\|^2 \) + \vep^2 \Ld^2 t^{-3+2\de} \| \Jw \|, 
}
where the bootstrap assumptions \eqref{bt-Si-eta}--\eqref{bt-Jw} are used in the estimates.
An application of the Gr\"{o}nwall's inequality yields
\[\|\Hw\|^2 + \|\Jw \|^2 + (n-1) \frac{4(n-3)}{n-2} \int_{t_0}^t t^{\prime -1} \|\Hw\|^2 \di t' \lesssim  I_2^2 + \vep^2 \Ld^2, \] 
and this justifies \eqref{ee-Jw-0}. Meanwhile, \eqref{trans-HJw-Jw} can be alternatively written as
\alis{
 \dtau \(  \Jw- \Jw_0 \)_{ipjq} = {}& \nabla_i \Hw_{j q p} - \nabla_p \Hw_{j q i} -  \frac{1}{n-2}(\dive \Hw \odot g)_{ipjq} \\
 &  + ( \Si, \, \eta ) \ast \Jw + (\eta, \, \Si) \ast \Ew,
}
which gives
\alis{
\|\Jw - \Jw_0\| \lesssim {}& \int_{t_0}^t t^{\prime -1} \( \|\n \Hw \| + \( \|\Si\|_{L^\infty} + \|\eta\|_{L^\infty} \)  \( \| \Jw\| +  \Ew\| \) \) \di t' \\
\lesssim {}&  \int_{t_0}^t \vep \Ld t^{\prime -2+\de} \, \di t' \lesssim \vep \Ld.
}
That is, we conclude \eqref{ee-Jw-Jw0-0}.

With the transport system for $\(\Hw, \, W\)$ \eqref{trans-HJw-Hw} and \eqref{eq-pt-W}, the estimate \eqref{ee-W-0} can be verified in an analogous way.
Moreover, due to \eqref{E:relat-W-Jw} and the bootstrap assumptions \eqref{bt-Si-eta}--\eqref{bt-Jw}, the estimate \eqref{ee-W-W0-0} follows from \eqref{ee-Jw-Jw0-0}.
\epf

\subsubsection{Higher-order estimates for $ W$ and $ W-W_0$}

\bg{lemma}\label{lemm-elliptic}
Fix an integer $N > \frac{n}{2}$. Suppose on an $n$-dimensional Riemannian manifold $(M, \, g)$, the curvature tensor is bounded \[ \|R_{imjn} \|_{L^\infty} + \|\nabla R_{ipjq} \|_{H_{N-1}} < C. \] 

Let $\Ph_{imjn}$ be a $(0,4)$-tensor on $M$ with compact support such that 
\beq\label{antisymm-Ph}
\Ph_{imjn}=-\Ph_{mijn}. 
\eeq
 If $\Ph$ satisfies the following elliptic system
\begin{align}\label{eq-elliptic-Ph} 
\n^p \Ph_{pimj} ={}& B_{imj}, \nnb \\
\n_{[p} \Ph_{im]jn} ={}& A_{pimjn},
\end{align}
 then it holds that
\beq\label{elliptic-est-Ph} 
\|\Ph\|_{H_{k}}  \lesssim \|A\|_{H_{k-1}} + \|B\|_{H_{k-1}} + \|\Ph\|, \quad 1 \leq k \leq N+2.
\eeq

Similarly, if  $\Ps_{ij}$ is a symmetric and trace-free $(0, 2)$-tensor on $M$ with compact support, and satisfies 
\ali{div-curl-system}{
\n^i \Ps_{ij} & =B_j, \nnb\\
\n_i \Ps_{j k} - \n_j \Ps_{i k} &= A_{ijk},
} 
then the following estimates hold,
\beq\label{elliptic-est-Ps} 
\|\Ps\|_{H_{k}}  \lesssim \|A\|_{H_{k-1}} + \|B\|_{H_{k-1}} + \|\Ps\|, \quad 1 \leq k \leq N+2.
\eeq
\ed{lemma}
The proof of Lemma \ref{lemm-elliptic} is collected in Appendix \ref{sec-proof-lem-ellip}.

Since the spatial Weyl tensor $W$ satisfies the elliptic system \eqref{eq-div-W}--\eqref{eq-curl-W}, we can apply the elliptic estimates in Lemma \ref{lemm-elliptic} which are combined with the density corollary \ref{coro-density} to demonstrate the following estimates for $W$ and $W-W_0$.
\bg{proposition}\label{prop-W}
The bootstrap assumptions \eqref{bt-Si-eta}--\eqref{bt-Jw} suggest that
\begin{subequations}
\begin{align}
  \|W\|_{H_N} \lesssim {}&  I_{N+2} + \vep  \Ld, \label{ee-W-N} \\
 \|W  - W[\ga] \|_{H_N} \lesssim {}& \vep \Ld, \label{ee-W-W0-N}
\end{align}
\end{subequations}
for $N > \frac{n}{2}$.
\ed{proposition}

\bpf
Applying Lemma \ref{lemm-elliptic} to the elliptic system \eqref{eq-div-W}--\eqref{eq-curl-W} yields
\alis{
\|W \|_{H_N}  
  \lesssim {}& \|W\|  + \|\Ew\|_{H_N} + \|\Si \|_{H_N} + \|\eta \|_{H_N} + \|\Si \|^2_{H_N} + \|\eta \|^2_{H_N}  \nnb \\ 
\lesssim {}& I_{2} + \vep  \Ld, \quad  \text{by} \,\, \eqref{ee-W-0}.
}
This concludes \eqref{ee-W-N}.

To prove \eqref{ee-W-W0-N}, we first note that the Weyl tensor $W[\ga]$ of the Einstein manifold $(M, \, \ga)$ obeys
\begin{subequations}
\begin{align}
\n[\ga]^l W[\ga]_{ljpi} ={}& 0, \label{elliptic-W[ga]-1} \\
  \nabla[\ga]_{[p} W[\ga]_{i m] j n} ={}& 0, \label{elliptic-W[ga]-2}
\end{align}  
\end{subequations}
 where $\n[\ga]$ denotes the covariant derivative with respect to $\ga$.
Due to the fact \[ \n[\ga] W[\ga] - \n W[\ga] = \n \ga \ast W[\ga] = \n (g-\ga) \ast W[\ga], \] we infer from the systems \eqref{eq-div-W}--\eqref{eq-curl-W} and \eqref{elliptic-W[ga]-1}--\eqref{elliptic-W[ga]-2} the following elliptic system for $ W-W[\ga]$,
\begin{align*}
  \nabla^l \( W-W[\ga] \)_{l j p i}  ={} & \pm \nabla \Ew \pm \n \Si \pm  \n  \eta  + \n \Si \ast (\Si, \, \eta) + \n \eta \ast (\Si, \, \eta)\\
  & +  \n (g-\ga) \ast W[\ga] + (g-\ga) \ast \n[\ga] W[\ga], \\
 \nabla_{[p} \( W- W[\ga] \)_{i m] j n}  ={} &  \pm \nabla \Ew \pm \n \Si \pm  \n  \eta  + \n \Si \ast (\Si, \, \eta) + \n \eta \ast (\Si, \, \eta)\\
   & + \n (g-\ga) \ast W[\ga].
\end{align*}
Applying Lemma \ref{lemm-elliptic} to the above elliptic system, we obtain
\ali{ee-elliptic-W-W[ea]}{
& \|W-W[\ga] \|_{H_N} \lesssim \|W-W[\ga]\|  + \|\Ew\|_{H_N} + \|\Si \|_{H_N} + \|\eta \|_{H_N} + \|\Si \|^2_{H_N}   \nnb \\ 
& \quad + \|\eta \|^2_{H_N}  + \| \n (g-\ga) \ast W[\ga] \|_{H_{N-1}} + \| (g-\ga) \ast \n[\ga] W[\ga] \|_{H_{N-1}}.
}
Note that  by Proposition \ref{prop-Multiplication}
\alis{ 
& \| \n (g-\ga) \ast W[\ga] \|_{H_{N-1}} + \| (g-\ga) \ast \n[\ga] W[\ga] \|_{H_{N-1}} \\
  \lesssim {}&  \| \n (g-\ga)   \|_{H_{N-1}} \| W[\ga] \|_{H_{N}} + \| (g-\ga) \|_{H_{N}} \| \n[\ga] W[\ga] \|_{H_{N-1}}.
}
Then combined with the above inequality, 
 \eqref{ee-W-W0-0} and  the bootstrap assumptions \eqref{bt-Si-eta}--\eqref{bt-Jw}, the estimate \eqref{E:ee-elliptic-W-W[ea]} is further sharpen as
\alis{
\|W-W[\ga] \|_{H_N} \lesssim {}&  \vep I_{N+2} + \vep \Ld + \vep \Ld \( \| W[\ga] \|_{H_{N}} +  \| \n[\ga] W[\ga] \|_{H_{N-1}} \).
}
Moreover, due to the bootstrap assumption  \[ \|g - \gamma \|_{H_{N+2}} \leq \vep \Ld, \quad N > \frac{n}{2}, \] the two norms $\| \cdot \|_{H^k}$ and $\| \cdot \|_{H^k(M, \, \gamma)}$, $k \leq N+2$, are equivalent (see the proof leading to Proposition \ref{pro-density}).
It then follows that
\alis{
\|W-W[\ga] \|_{H_N} \lesssim {}&  \vep I_{N+2} + \vep \Ld + \vep \Ld \| W[\ga] \|_{H_{N} (M, \, \ga)} \\
\lesssim {}& \vep \Ld.
}

\epf

We next proceed to the energy estimates for  $\Ew$ and $\Hw$, based on the hyperbolic system \eqref{trans-EHw-Ew}--\eqref{trans-EHw-Hw}.

\subsection{Estimates for $\Ew$ and $\Hw$}
In this section, we aim to prove the following proposition.
\bg{proposition}\label{pro-est-Ew-Hw}
We have 
\alis{
 & t^2 \|\Ew\|^2_{H_{N}}  + t^2 \|\Hw\|^2_{H_{N}} \lesssim  \vep^2 I_{N+2}^2  + \varepsilon \Ld \( \vep I_{N+2} + \vep^2 \Ld^2\) t^{2 \delta},
 }
provided the bootstrap assumptions \eqref{bt-Si-eta}--\eqref{bt-Jw}.

\ed{proposition}

To develop an approach of energy estimates for the hyperbolic system of Maxwell type \eqref{trans-EHw-Ew}--\eqref{trans-EHw-Hw} on a spacetime foliated by spatially Riemannian manifolds with negative curvature, the following lemma plays a crucial role.
\begin{lemma}\label{lemma-div-curl}
 Let  $k \in \mathbb{Z}$, $k \geq 0$, and $E_{ij}$ be a symmetric, trace-free $(0, 2)$-tensor on $M$, $H_{ijl}$ be a $(0, 3)$-tensor on $M$ satisfying 
 \beq\label{cyclic-H}
 H_{ijl}=-H_{jil}, \quad H_{ijl} g^{jl}=0, \quad H_{[ijl]}=0.
 \eeq  
 Then the following identity holds,
\ali{identity}{
&   \nabla^{\mathring{k}} \Delta^{ [\frac{k}{2}] } \n^p H_{p j i} \cdot \nabla^{\mathring{k}} \Delta^{ [\frac{k}{2}] } E^{ij}  + \nabla^{\mathring{k}} \Delta^{ [\frac{k}{2}] }  \n_p E_{i j}  \cdot \nabla^{\mathring{k}} \Delta^{ [\frac{k}{2}] } H^{p i j}  \nnb\\
=& \sum_{0 \leq l \leq k}  \nabla^p    \left( \nabla^{\mathring{l}} \Delta^{ [\frac{l}{2}] } H_{p (i j)}  * \nabla^{\mathring{l}} \Delta^{ [\frac{l}{2}] } E^{i j}  \right) \nnb \\
& - \sum_{0 \leq m <k} C_k^m (n-3)^{k-m} \nabla^{\mathring{m}} \Delta^{ [\frac{m}{2}] } \n^p H_{p (i j)} \cdot \nabla^{\mathring{m}} \Delta^{ [\frac{m}{2}] } E^{i j} \nnb \\  
& +  \sum_{0 \leq l \leq k} \nabla_{I_{l-1}} ( O_{ipjq}  \ast H) *  \nabla^{\mathring{l}} \Delta^{ [\frac{l}{2}] } E +  \nabla_{I_{l-1}} ( O_{ipjq}  \ast E) *  \nabla^{\mathring{l}} \Delta^{ [\frac{l}{2}] } H, 
}
where \[ O_{imjn} : = R_{imjn} + \frac{1}{2} ( g \odot g )_{imjn} \] is the error term of the Riemann curvature $R_{imjn}$ (cf. the principle part of $R_{imjn}$ is $- \frac{1}{2} ( g \odot g )_{imjn}$). In the formula of \eqref{E:identity}, the last line vanishes when $l=0$, and moreover, if $k=0$, the last two lines are both absent. We also remark that the constants $C_k^m$ are the combinatorial numbers. 
\end{lemma}
We postpone the proof of Lemma \ref{lemma-div-curl} to Appendix \ref{sec-id-Bianchi}.

\subsubsection{Zero-order estimates for $\Ew$ and $\Hw$}\label{sec-ee-EH-0}
First of all, we need to establish an energy identity for $\Ew$ and $\Hw$. For this purpose, we multiply $4\Ew$ and $2\Hw$ on  the hyperbolic system \eqref{trans-EHw-Ew}--\eqref{trans-EHw-Hw} and note that \eqref{E:identity} with $k=0$ yields
\beq\label{spatial-div-0} 
\n^p H_{p j i}  E^{ij}  +    \n_p E_{i j}  H^{p i j} =  \nabla^p    \left(  H_{p (i j)}    E^{i j}  \right).
\eeq 
It then follows that
\ali{energy-id-EwHw-0}{
& \p_t \( 2 t^2\|\Ew\|^2 + t^2 \|\Hw\|^2 \) + 4 (n-3) t \|\Ew\|^2 \nnb \\
={} & \int_{M_t} 2 t \Si_{p q} W^{p i q j} \Ew_{i j} \, \di \mu_g + f_1,
}
where 
\alis{
f_1= {}& \int_{M_t} t \(\Si, \, \eta \) \ast \( \Ew \ast \Ew, \, \Hw \ast \Hw \) \di \mu_g + \int_{M_t} t \Si \ast \Si \ast \Si \ast \Ew \,  \di \mu_g.
}
The bootstrap assumptions \eqref{bt-Si-eta}--\eqref{bt-Jw} enable us to bound $f_1$ as below,
\[  |f_1| \lesssim  \vep^3 \Ld^3 t^{-2+3\de}. \]

In the sequel, to treat the leading nonlinear term \[ \int_{M_t} 2 t \Si_{p q} W^{p i q j} \Ew_{i j} \, \di \mu_g, \] we have to find out more hidden structures in the Einstein equations. Here we appeal to the geometric structure equation \eqref{eq-evolution-Si}, \[ \lie_{\p_t} ( t \Si_{i j} ) = \mathcal{L}_{\dtau} \Sigma_{ij} + \Sigma_{ij} =  \Ew_{i  j } - \Sigma_{ip} \Sigma_{j}^p -  \frac{1}{n}   |\Sigma|^2 g_{ij}, \] which allows us to replace $\Ew_{ij}$ by $\lie_{\p_t} ( t \Si_{i j} ) + \text{l.o.q.}$. As a consequence,
\alis{
 & \int_{M_t}  2 t \Si_{p q} W^{p i q j} \Ew_{i j} \, \di \mu_g \\
 ={}&  \int_{M_t} 2 t \Si_{p q} W^{p i q j} \( \lie_{\p_t} ( t \Si)_{i j}  + \Sigma_{i m} \Sigma_{j}^m +  \frac{1}{n}   |\Sigma|^2 g_{ij} \) \di \mu_g \\
 ={}&  \int_{M_t} \( \p_t ( t^2 \Si_{p q}   \Si_{i j} W^{p i q j} )  -  t^2 \Si_{p q}   \Si_{i j}  \lie_{\p_t}  W^{p i q j} + 2 t \Si_{p q} W^{p i q j} \Sigma_{i m} \Sigma_{j}^m \) \di \mu_g \\ 
    ={}& \p_t \int_{M_t}   t^2 \Si_{p q}   \Si_{i j} W^{p i q j} \, \di \mu_g + \int_{M_t}  t  \Si \ast   \Si  \ast \( \n \Hw, \, \Ew \)  \di \mu_g \\
    &  + \int_{M_t} t  \Si \ast   \Si \ast (\Si, \,\eta) \ast (W, \, \Si, \,\Ew, \, \Si \ast \Si) \, \di \mu_g.
} We note that \eqref{eq-pt-W} is used in the above calculations.
After integrating over $t' \in [t_0, \, t]$ and inserting the bootstrap assumptions \eqref{bt-Si-eta}--\eqref{bt-Jw}, we achieve
\alis{
&  2 t^2\|\Ew\|^2 + t^2 \|\Hw\|^2  - \int_{M_t} t^2 \Si^{p q}   \Si^{i j}  W_{ p i q j}\, \di \mu_g + 4 (n-3) \int_{t_0}^t t' \|\Ew\|^2 \, \di t' \\
\lesssim {}& \varepsilon^2 I_2^2 + \int_{t_0}^t \varepsilon^3 \Ld^3 t^{\prime -2+3 \de} \, \di t'  + \int_{t_0}^t \varepsilon^3 \Ld^3 I_{N+2} t^{\prime -2+3 \de} \,\di t'    \\
\lesssim {}& \varepsilon^2 I_2^2 +  \varepsilon^3 \Ld^3,
} which will be the desired energy inequality if we manage to demonstrate the positivity of $- \int_{M_t} t^2 \Si^{p q}   \Si^{i j}  W_{ p i q j}\, \di \mu_g$.
Since we have imposed the non-positive assumption on $W[\ga]$ (see Definition \ref{def-Non-positive-Weyl}), that is, \[ - \int_M \Si_{i j} \Si_{p q}  \ga^{i i^\prime} \ga^{j j^\prime} \ga^{p p^\prime} \ga^{q q^\prime} W[\ga]_{  i^\prime p^\prime j^\prime q^\prime} \, \di \mu_{\ga} \geq 0, \] it suffices to estimate the following error term 
\alis{
 & \Big| \int_{M }  \Si^{p q}   \Si^{i j}  W_{ p i q j}  \, \di \mu_g - \int_M \Si_{i j} \Si_{p q}  \ga^{i i^\prime} \ga^{j j^\prime} \ga^{p p^\prime} \ga^{q q^\prime} W[\ga]_{  i^\prime p^\prime j^\prime q^\prime} \, \di \mu_{\ga} \Big| \\
\lesssim {} & \|\Si\|^2  \|W-  W[\ga] \|_{L^\infty}  + \|\Si\|^2   \|g -  \ga \|_{L^\infty} \( \| W \|_{L^\infty} + \| W[\gamma] \|_{L^\infty} \) \\ \lesssim {} &  \vep \Ld   \|\Si\|^2 + \( I_{N+2}  +  \vep \Ld \)   \vep \Ld   \|\Si\|^2 \\
 \lesssim {} &  \vep \Ld   \|\Si\|^2,
 }
 where we have used the improved estimates for $\|W\|_{H_N}$ and $\|W-  W[\ga] \|_{H_N}$ \eqref{ee-W-N}--\eqref{ee-W-W0-N}. As a summary, we arrive at
 \alis{
& 2 t^2\|\Ew\|^2 + t^2 \|\Hw\|^2  +4 (n-3) \int_{t_0}^t t \|\Ew\|^2 \, \di t \\
& - \int_{M} t^2 \Si_{i j} \Si_{p q}  \ga^{i i^\prime} \ga^{j j^\prime} \ga^{p p^\prime} \ga^{q q^\prime} W[\ga]_{  i^\prime p^\prime j^\prime q^\prime} \, \di \mu_{\ga} \\
\lesssim {} &  \varepsilon^2 I_2^2 + \varepsilon^3 \Ld^3 +  \vep \Ld \|t \Si\|^2  \\
\lesssim  {} & \varepsilon^2 I_2^2  + \varepsilon^3 \Ld^3 t^{2 \de},
}
and therefore
\beq\label{esti-Ew-Hw-0}
 t^2\|\Ew\|^2 + t^2 \|\Hw\|^2  \lesssim  \varepsilon^2 I_2^2  + \varepsilon^3 \Ld^3 t^{2 \de}.
 \eeq

We conclude Proposition \ref{pro-est-Ew-Hw} with $N=0$. In what follows, we will complete the proof of Proposition \ref{pro-est-Ew-Hw} by induction. 

\subsubsection{Higher-order estimates of $\Ew$ and $\Hw$}
The inductive proof for higher-order estimates of $\Ew$ and $\Hw$ composes of three steps.

{\bf 1) Inductive assumption for $\|\Ew\|_{H_{k-1}}$ and $\|\Hw\|_{H_{k-1}}$}.
Suppose it holds that
\beq\label{induction-Ew-Hw}
\|t \Ew\|^2_{H_{k-1}} + \|t \Hw\|^2_{H_{k-1}} \lesssim \varepsilon^2 I^2_{k+1} + \varepsilon \Ld \( \vep I_{N+2} + \vep^2 \Ld^2\) t^{2 \delta},
\eeq
 we will prove the same estimate holds if we replace $k-1$ by $k$, for $k\leq N$. Notice that, the estimate \eqref{esti-Ew-Hw-0} indicates \eqref{induction-Ew-Hw} holds with $k=1$.

{\bf 2) Improved estimates for $\|\Si\|_{H_k}$}.
Based on the updated estimate for $\|\Ew\|$ \eqref{esti-Ew-Hw-0}, we can improve $\|\Si\|$ as follows.
Taking advantage of the transport equation of $\Sigma$  \eqref{eq-evolution-Si}, we have
\alis{
 \dtau \|\Sigma\|^2  +  2 \|\Sigma\|^2 &\lesssim \left( \|\eta\|_{L^\infty} +  \|\Sigma\|_{L^\infty} \right) \|\Sigma\|^2 + \|\Ew \| \|\Sigma\|.
}
It follows from the bootstrap assumptions \eqref{bt-Si-eta}--\eqref{bt-Jw} and the improved estimate \eqref{esti-Ew-Hw-0} that
\alis{ 
\p_{t} (t^2 \|\Sigma\|^2)  & \lesssim \varepsilon  \Lambda  t^{-2+ \delta} \cdot t^2 \|\Sigma\|_{L^2}^2 + \left(\varepsilon  I_{2} +  \varepsilon^{\frac{3}{2}} \Lambda^{\frac{3}{2}}   \right)  t^{-1+\de}  \cdot t \|\Sigma\|_{L^2}.
}
Then the Gr\"{o}nwall's inequality allows us to deduce
\beq\label{ee-Sigma}
t \|\Sigma\|_{L^2} \lesssim  \left(\varepsilon  I_{2} +  \varepsilon^{\frac{3}{2}} \Lambda^{\frac{3}{2}}   \right) t^{\delta}.
\eeq

Furthermore, the inductive assumption \eqref{induction-Ew-Hw} and the sharpen estimate \eqref{ee-Sigma} help us to improve $\|\Si\|_{H_k}$. Indeed, with $\|\eta\|_{H_{N+1}}$ being bounded (Lemma \ref{prop-energy-estimate-eta-low}) and $\nabla \eta$ being viewed as a source term, we regard the system \eqref{Codazzi-div-k}--\eqref{div-curl-k-hat} as an elliptic system for $\Sigma$. 
Making use of the elliptic estimates in Lemma \ref{lemm-elliptic},  we can prove that 
\ali{esti-Si}{
\| t\Sigma\|^2_{H_{k}} \lesssim {}& \| t\Sigma\|^2 + \| t\Hw\|^2_{H_{k-1}} + \|t \eta\|^2_{H_{k}}, \quad 1 \leq k \leq N, \nnb\\
\lesssim {}& \left(\varepsilon^2  I^2_{k+1}  + \varepsilon^2 \Ld I_{N+2} +  \varepsilon^3 \Lambda^3 \right) t^{2\delta} +  \varepsilon^2 I^2_{k+1}  +  \varepsilon^4 \Lambda^4\nnb \\
\lesssim {}&  \varepsilon^2 I^2_{k+1}  +  \varepsilon \Ld \( \vep I_{N+2} + \vep^2 \Ld^2\) t^{2 \delta}.
}

{\bf 3) Estimates for $\|\Ew\|_{H_{k}}$ and $\|\Hw\|_{H_{k}}$, $k\leq N$}. Before the analysis, we take $\nabla^{\mathring{k}} \Delta^{ [\frac{k}{2}] } $ derivative on the hyperbolic system of $\( \Ew,\, \Hw \)$ and obtain the higher-order equations,
\alis
{& \lie_{\dtau} \nabla^{\mathring{k}} \Delta^{ [\frac{k}{2}] }  \Ew_{i  j} +(n-2) \nabla^{\mathring{k}} \Delta^{ [\frac{k}{2}] } \Ew_{ij} + \nabla^{\mathring{k}} \Delta^{ [\frac{k}{2}] }  \nabla^p \Hw_{p (i j)} \\
={}& \nabla^{\mathring{k}} \Delta^{ [\frac{k}{2}] } \( \Si^{p q} W_{p i q j} + (\eta, \Si ) \ast \Ew  + \Si \ast \Si \ast \Si\) \\
& + \sum_{a+b=k-1} \n_{I_a} (\n \Si, \n \eta) \ast \n_{I_b} \Ew, \\
  &\lie_{\dtau} \nabla^{\mathring{k}} \Delta^{ [\frac{k}{2}] }  \Hw_{p i j} + \nabla^{\mathring{k}} \Delta^{ [\frac{k}{2}] }  \Hw_{p i j} + \nabla^{\mathring{k}} \Delta^{ [\frac{k}{2}] }   \nabla_p \Ew_{i  j } -  \nabla^{\mathring{k}} \Delta^{ [\frac{k}{2}] } \nabla_i  \Ew_{p  j } \\
  ={}& \nabla^{\mathring{k}} \Delta^{ [\frac{k}{2}] } \( \Si \ast  \Hw\) + \sum_{a+b=k-1} \n_{I_a} (\n \Si, \n \eta) \ast \n_{I_b} \Hw.
}
To do the higher--order energy estimates, we multiply $4t \nabla^{\mathring{k}} \Delta^{ [\frac{k}{2}] } \Ew^{ij}$ and $2 t \nabla^{\mathring{k}} \Delta^{ [\frac{k}{2}] } \Hw^{pij}$ on both sides of the above system. Now let us focus on the summation of spatial derivative terms for a moment, 
\alis{
& 4 t \nabla^{\mathring{k}} \Delta^{ [\frac{k}{2}] } \n^p \Hw_{p (i j)} \cdot \nabla^{\mathring{k}} \Delta^{ [\frac{k}{2}] } \Ew^{ij} + 2t  \( \nabla^{\mathring{k}} \Delta^{ [\frac{k}{2}] }   \nabla_p \Ew_{i  j } -  \nabla^{\mathring{k}} \Delta^{ [\frac{k}{2}] } \nabla_i  \Ew_{p  j } \)  \nabla^{\mathring{k}} \Delta^{ [\frac{k}{2}] } \Hw^{pij} \\
={}& 4 t \nabla^{\mathring{k}} \Delta^{ [\frac{k}{2}] } \n^p \Hw_{p (i j)} \cdot \nabla^{\mathring{k}} \Delta^{ [\frac{k}{2}] } \Ew^{ij}  + 4t \nabla^{\mathring{k}} \Delta^{ [\frac{k}{2}] }  \n_p \Ew_{i j}  \cdot \nabla^{\mathring{k}} \Delta^{ [\frac{k}{2}] } \Hw^{p i j}.
}
Notice that, when $k=0$, the above formula reduces to \eqref{spatial-div-0}. For general $k \in \mathbb{N}$, we apply Lemma \ref{lemma-div-curl} to  achieve the following identity
\ali{id-ho-maxwell}{
&  4t\( \nabla^{\mathring{k}} \Delta^{ [\frac{k}{2}] } \n^p \Hw_{p (i j)} \cdot \nabla^{\mathring{k}} \Delta^{ [\frac{k}{2}] } \Ew^{ij}  + \nabla^{\mathring{k}} \Delta^{ [\frac{k}{2}] }  \n_p \Ew_{i j}  \cdot \nabla^{\mathring{k}} \Delta^{ [\frac{k}{2}] } \Hw^{p i j} \) \nnb \\
=& \sum_{0 \leq l \leq k}  \nabla^p    \left( 4 t \nabla^{\mathring{l}} \Delta^{ [\frac{l}{2}] } \Hw_{p (i j)}  * \nabla^{\mathring{l}} \Delta^{ [\frac{l}{2}] } \Ew^{i j}  \right) \nnb \\ 
 - & 4 t\sum_{0 \leq m <k} C_k^m (n-3)^{k-m} \nabla^{\mathring{m}} \Delta^{ [\frac{m}{2}] } \n^p \Hw_{p (i j)} \cdot \nabla^{\mathring{m}} \Delta^{ [\frac{m}{2}] } \Ew^{i j} \nnb \\
 + &  \sum_{0 \leq l \leq k} \nabla_{I_{l-1}} ( O_{imjn}  \ast \Hw) *  t \nabla^{\mathring{l}} \Delta^{ [\frac{l}{2}] } \Ew +  \nabla_{I_{l-1}} ( O_{imjn}  \ast \Ew) * t \nabla^{\mathring{l}} \Delta^{ [\frac{l}{2}] } \Hw.
}
Recall that $O_{imjn} = R_{imjn} + \frac{1}{2} ( g \odot g )_{imjn} $ is the error term of the Rieman curvature $R_{imjn}$, and in view of \eqref{eq-Riem-W},
\alis{
O_{imjn} ={}&  W_{imjn} + \frac{1}{n-2}  \Ew \odot g  + \Si \odot g + \eta  g \odot g \\
 & +\Si \ast \Si + \Si \ast \eta + \eta^2 \ast g. 
}
 
The third line in \eqref{E:id-ho-maxwell} contains quadratic terms, whose decay is inadequate if they are estimated straightforwardly. Fortunately, we observe that these quadratic terms involve $\n^p \Hw_{p (i j)}$ which happens to occur in the transport equation of $\Ew_{ij}$ \eqref{trans-EHw-Ew}. Therefore we are able to replace $\n^p \Hw_{p (i j)}$ by \[ \lie_{ \dtau} \Ew_{i j} + (n-2) \Ew_{ij} + \cdots \] and then these quadratic terms can be further calculated as follows,
\alis{
& - 4 t  C_k^m (n-3)^{k-m} \nabla^{\mathring{m}} \Delta^{ [\frac{m}{2}] } \n^p \Hw_{p (i j)} \cdot \nabla^{\mathring{m}} \Delta^{ [\frac{m}{2}] } \Ew^{i j} \\
={} & 4 t C_k^m (n-3)^{k-m} \nabla^{\mathring{m}} \Delta^{ [\frac{m}{2}] }  \Ew^{ij} \\
& \qquad  \cdot \nabla^{\mathring{m}} \Delta^{ [\frac{m}{2}] } \( \lie_{ \dtau} \Ew_{i j} + (n-2) \Ew_{ij} - \Si^{p q} W_{piqj} + (\Si, \, \eta) \ast \Ew + \Si \ast \Si \ast \Si \) \\
={}&  2 t C_k^m (n-3)^{k-m} \dtau |\nabla^{\mathring{m}} \Delta^{ [\frac{m}{2}] } \Ew|^2 + 4 t C_k^m  (n-2) (n-3)^{k-m} |\nabla^{\mathring{m}} \Delta^{ [\frac{m}{2}] } \Ew|^2 \\
& - 4 t C_k^m (n-3)^{k-m}  \nabla^{\mathring{m}} \Delta^{ [\frac{m}{2}] } \Ew^{ij} \cdot \nabla^{\mathring{m}} \Delta^{ [\frac{m}{2}] } \( \Si^{p q} W_{piqj} \) \\
&+ \nabla^{\mathring{m}} \Delta^{ [\frac{m}{2}] }  \( (\Si, \, \eta) \ast \Ew + \Si \ast \Si \ast \Si \) \ast t \nabla^{\mathring{m}} \Delta^{ [\frac{m}{2}] } \Ew.
}
Note that, in the above identity, the first line on the right hand side of the second equality can be rearranged as
\alis{
&  2 t C_k^m (n-3)^{k-m} \dtau |\nabla^{\mathring{m}} \Delta^{ [\frac{m}{2}] } \Ew|^2 + 4 t C_k^m  (n-2) (n-3)^{k-m} |\nabla^{\mathring{m}} \Delta^{ [\frac{m}{2}] } \Ew|^2 \\
={}& C_k^m (n-3)^{k-m} \( \p_t \(2 t^2 |\nabla^{\mathring{m}} \Delta^{ [\frac{m}{2}] } \Ew|^2 \)  + 4 t (n-3) |\nabla^{\mathring{m}} \Delta^{ [\frac{m}{2}] } \Ew|^2  \).
}
In summary, we manage to reformulate the original quadratic terms as a divergence form plus some quadratic terms with positive signs and lower order terms
\ali{id-good-sign}{
& - 4 t  C_k^m (n-3)^{k-m} \nabla^{\mathring{m}} \Delta^{ [\frac{m}{2}] } \n^p \Hw_{p (i j)} \cdot \nabla^{\mathring{m}} \Delta^{ [\frac{m}{2}] } \Ew^{i j} \nnb\\
={}& \p_t \( 2t^2 C_k^m (n-3)^{k-m} |\nabla^{\mathring{m}} \Delta^{ [\frac{m}{2}] } \Ew|^2 \) + 4 t C_k^m (n-3)^{k+1-m} |\nabla^{\mathring{m}} \Delta^{ [\frac{m}{2}] } \Ew|^2 \nnb \\
& - 4 t C_k^m (n-3)^{k-m} \nabla^{\mathring{m}} \Delta^{ [\frac{m}{2}] } \Ew^{ij} \cdot \nabla^{\mathring{m}} \Delta^{ [\frac{m}{2}] } \( \Si^{p q} W_{piqj} \) \nnb\\
&+ \nabla^{\mathring{m}} \Delta^{ [\frac{m}{2}] }  \( (\Si, \, \eta) \ast \Ew + \Si \ast \Si \ast \Si \) \ast t \nabla^{\mathring{m}} \Delta^{ [\frac{m}{2}] } \Ew.
}
Moreover, substituting the above formulation back into the energy identity, we find that the divergence form provides boundary terms which will contribute extra positive energies, and the positive quadratic terms will afford spacetime integrals with favourable signs, after integrating over time. We collect all these informations in the following energy identity,
\ali{id-energy-ho-EHw-1}{
& \p_t \(2 t^2\|  \nabla^{\mathring{k}} \Delta^{ [\frac{k}{2}] } \Ew\|^2 + t^2 \|  \nabla^{\mathring{k}} \Delta^{ [\frac{k}{2}] } \Hw\|^2 \) \nnb \\
& \quad + \p_t \(   \sum_{l<k}  2 C_k^l  (n-3)^{k-l} t^2 \| \nabla^{\mathring{l}} \Delta^{ [\frac{l}{2}] } \Ew \|^2 \) \nnb \\
 & \quad +  \sum_{m \leq k}   4 C_k^m  (n-3)^{k+1-m} t \|\nabla^{\mathring{m}} \Delta^{ [\frac{m}{2}] } \Ew \|^2 \nnb \\
={} &
  f^k_1 + f^k_2 + f^k_3,
}
where
\[ f^k_1 =  \int_{M_t} \sum_{m \leq k} 4 C_k^m (n-3)^{k-m} t \nabla^{\mathring{m}} \Delta^{ [\frac{m}{2}] } \Ew^{ij} \cdot \nabla^{\mathring{m}} \Delta^{ [\frac{m}{2}] } \( \Si^{p q} W_{piqj} \)  \di \mu_g, \]
and
\alis
{|f^k_2| \lesssim {} &  \int_{M_t} \sum_{l \leq k} \big| \nabla_{I_{l}} \( \( \Si, \, \eta \)  \ast \Gm \) *  t \nabla^{\mathring{l}} \Delta^{ [\frac{l}{2}] } \Gm\big| \, \di \mu_g   \\
&+  \int_{M_t} \sum_{l \leq k} \big| \nabla_{I_{l-1}} \( \( \Si, \, \eta \)  \ast \Gm \) *  t \nabla^{\mathring{l}} \Delta^{ [\frac{l}{2}] } \Gm \big| \, \di \mu_g   \\
&+ \int_{M_t} \sum_{l \leq k} \big| \nabla^{\mathring{l}} \Delta^{ [\frac{l}{2}] }  \(   \Si \ast \Si \ast \Si \) \ast t \nabla^{\mathring{l}} \Delta^{ [\frac{l}{2}] } \Ew \big| \, \di \mu_g \\
& + \int_{M_t} \sum_{l \leq k} \nabla_{I_{l-1}} \big| \( \Gm   \ast \Gm \) *  t \nabla^{\mathring{l}} \Delta^{ [\frac{l}{2}] } \Gm \big| \, \di \mu_g \\
& + \int_{M_t} \sum_{l \leq k} \nabla_{I_{l-1}} \big| \( \( \Si, \, \eta \) \ast \(\Si, \, \eta \)  \ast \Gm \) *  t \nabla^{\mathring{l}} \Delta^{ [\frac{l}{2}] } \Gm \big| \, \di \mu_g,} 
with $\Gm \in \{\Ew, \, \Hw \}$, and 
\alis
{ f^k_3 = {} & \int_{M_t} \sum_{l \leq k} \nabla_{I_{l-1}} ( W \ast \Hw) *  t \nabla^{\mathring{l}} \Delta^{ [\frac{l}{2}] } \Ew \, \di \mu_g \\
&  +  \int_{M_t}  \sum_{l \leq k} \nabla_{I_{l-1}} ( W  \ast \Ew) * t \nabla^{\mathring{l}} \Delta^{ [\frac{l}{2}] } \Hw \, \di \mu_g. 
}

The treatment for $f_1^k$ will be postponed to the last.
The second term $f_2^k$ involves only lower-order terms. It can be estimated straightforwardly as follows,
\ali{ee-f1k}{
|f_2^k| \lesssim &  \( \|\eta \ast \Gm\|_{H_k} + \|\Si \ast \Gm\|_{H_k} \)  \|t \Gm\|_{H_{k}} \nnb \\
& +  \|\Si \ast \Si \ast \Si\|_{H_k}  \|t\Ew\|_{H_{k}} \nnb \\
&+  \| \Gm \ast \Gm\|_{H_{k-1}} \|t\Gm\|_{H_{k}}  \nnb \\
&+  \| \Gm \ast \(\Si, \, \eta\) \ast \(\Si, \, \eta \) \|_{H_{k-1}} \|t\Gm\|_{H_{k}} \nnb \\
\lesssim {}& \varepsilon^3 \Ld^3  t^{-2+3\de}.
}
The third term $f_3^k$ involves the non-decaying Weyl tensor $W$ whose estimate has been updated in \eqref{ee-W-N}. However, a straightforward estimate as below
\[
\int_{t_0}^t |f_3^k| \di t' \lesssim   \int_{t_0}^t \| \Gm \ast W\|_{H_{k-1}} \|t \Gm\|_{H_{k}} \di t' \lesssim \vep^2 \Ld^2 t^{2\de}
\]
would prevent us from improving the estimates for $\|\Ew\|_{H_k}$ and $\|\Hw\|_{H_k}$. Apart from this failure, the following manipulation that combines the improvement for $\|W\|_{H_N}$ \eqref{ee-W-N} and the inductive assumption \eqref{induction-Ew-Hw} is invalid as well, 
\ali{ee-direct-leading}{
 & \int_{t_0}^t t'  \| \Gm \ast W\|_{H_{k-1}} \| \Gm\|_{H_{k}} \di t' \lesssim \int_{t_0}^t t' \| W\|_{H_{N}}  \|\Gm\|_{H_{k-1}}  \|\Gm\|_{H_{k}} \di t' \nnb \\
 \lesssim {} & \int_{t_0}^t \( I_{N+2} +  \vep \Ld \)  \cdot  \( \varepsilon I_{k+1} + \varepsilon^{\frac{1}{2}} \Ld^{\frac{1}{2}} \( \vep^{\frac{1}{2}} I^{\frac{1}{2}}_{N+2} + \vep \Ld \) t^{\prime \delta} \)  \vep \Ld t^{\prime -1+\de} \di t' \nnb \\
 \lesssim {} &   \varepsilon^{2} \Ld^{\frac{3}{2}} I^{\frac{3}{2}}_{N+2}  t^{2 \delta} + \vep^\frac{5}{2} \Ld^\frac{5}{2} I_{N+2} t^{2 \de} + \varepsilon^3  I^{\frac{1}{2}}_{N+2} \Ld^{\frac{5}{2}} t^{2\de} +  \varepsilon^3 \Ld^{\frac{7}{2}} t^{2\de} \nnb \\
 \lesssim {} &  \varepsilon^{2} \Ld^{\frac{3}{2}} t^{2 \delta},
}
since, to close the inductive argument, our expected bound should be\footnote{This bound should be the right hand side of \eqref{induction-Ew-Hw} with $k$ replaced by $k+1$.}  $\varepsilon^2 I^2_{k+1} + \varepsilon \Ld \( \vep I_{N+2} + \vep^2 \Ld^2\) t^{2 \delta}$, while the bound derived in \eqref{E:ee-direct-leading} $\varepsilon^{2} \Ld^{\frac{3}{2}} t^{2 \delta}$ is obviously not proper for 
\[   \varepsilon^{2} \Ld^{\frac{3}{2}}  t^{2 \delta} \not\le C \( \varepsilon^2 I^2_{k+1} + \varepsilon \Ld \( \vep I_{N+2} + \vep^2 \Ld^2\) t^{2 \delta} \). \] 
Therefore, instead of estimating $f_3^k$ directly, we apply integration by parts to the second term in $f_3^k$, so that 
\alis{
 f_3^k ={}& \int_{M_t}  \sum_{l \leq k} \nabla_{I_{l-1}} ( W_{ipjq}  \ast \Hw) *  t \nabla^{\mathring{l}} \Delta^{ [\frac{l}{2}] } \Ew \, \di \mu_g \\
 & + \int_{M_t}  \sum_{l \leq k}  \nabla_{I_{l}} ( W_{ipjq}  \ast \Ew) * t \nabla_{I_{l-1} } \Hw \, \di \mu_g. 
 }
As a result, $f_3^k$ can be estimated in the following way,
\ali{ee-f2k}{
|f_3^k| \lesssim {}&   \| \Hw \ast W\|_{H_{k-1}} \|t\Ew\|_{H'_{k}} +  \| \Ew \ast W\|_{H_{k}} \|t\Hw\|_{H_{ k-1}} \nnb \\
\lesssim {} & t \| W\|_{H_{N}}  \|\Hw\|_{H_{k-1}}  \|\Ew\|_{H'_{k}} + t \| W\|_{H_{N}}   \|\Ew\|_{H_{k}} \|\Hw\|_{H_{ k-1}} \nnb \\
\lesssim {} &  \| W\|_{H_{N}}  \|\Hw\|_{H_{k-1}}  \|t \Ew\|_{H'_{k}},
}
where in the last inequality we have used the equivalence between the two norms $\| \cdot \|_{H_k}$ and $\| \cdot \|_{H'_{k}}$.
It is important to remark that $f_3^k$ \eqref{E:ee-f2k} 
is a leading nonlinear term, which will be further analyzed in \eqref{E:ee-regularity-solve}, with the help of  the extra positive terms on the left side of the energy identity \eqref{E:id-energy-ho-EHw-1}.


The remaining term $f_1^k$ on the right side of the energy identity \eqref{E:id-energy-ho-EHw-1}, can be handled in an analogous manner as the $k=0$ case. We first separate $f_1^k$ into the top-order (in $\Si$) parts and lower-order parts as follows,
 \ali{split-leading-term}{
 f_1^k= {}& \int_{M_t} \sum_{m \leq k} 4 C_k^m (n-3)^{k-m} t \nabla^{\mathring{m}} \Delta^{ [\frac{m}{2}] } \Ew^{ij} \nnb \\
 &\quad \cdot \( \nabla^{\mathring{m}} \Delta^{ [\frac{m}{2}] } \Si^{p q} W_{piqj}  +  \sum_{a+b=m-1} \n_{I_a} \Si \ast \n_{I_b} \n W \) \di \mu_g.
 }
For the top-order (in $\Si$) terms \[ \int_{M_t}   \sum_{m \leq k} 4 C_k^m (n-3)^{k-m} t  \nabla^{\mathring{m}} \Delta^{ [\frac{m}{2}] } \Si_{p q} W^{p i q j}  \cdot  \nabla^{\mathring{m}}\Delta^{ [\frac{m}{2}] } \Ew_{i j}  \, \di \mu_g, \] we use the higher-order version of \eqref{eq-evolution-Si}, which reads 
\alis{
& \lie_{\p_t} ( t  \nabla^{\mathring{m}} \Delta^{ [\frac{m}{2}] } \Si_{i j} ) = \mathcal{L}_{\dtau} \nabla^{\mathring{m}} \Delta^{ [\frac{m}{2}] } \Sigma_{ij} + \nabla^{\mathring{m}} \Delta^{ [\frac{m}{2}] } \Sigma_{ij} \\
={}&  \nabla^{\mathring{m}} \Delta^{ [\frac{m}{2}] } \Ew_{i  j } - \nabla^{\mathring{m}} \Delta^{ [\frac{m}{2}] } \(\Sigma_{ip} \Sigma_{j}^p \) -  \frac{1}{n}  \nabla^{\mathring{m}} \Delta^{ [\frac{m}{2}] } |\Sigma|^2 g_{ij}\\
& + \sum_{a+b=m-1} \n_{I_a} (\n \Si, \n \eta) \ast \n_{I_b} \Si,
} to replace $\nabla^{\mathring{m}} \Delta^{ [\frac{m}{2}] } \Ew_{i j}$ by $\lie_{\p_t} ( t \nabla^{\mathring{m}} \Delta^{ [\frac{m}{2}] } \Si_{i j} ) + \text{l.o.t.}$ so that
\alis{
 & \int_{M_t}  \sum_{m \leq k} 4 C_k^m (n-3)^{k-m} t \nabla^{\mathring{m}} \Delta^{ [\frac{m}{2}] } \Si^{p q}    W_{p i q j} \cdot \nabla^{\mathring{m}} \Delta^{ [\frac{m}{2}] } \Ew_{i j} \, \di \mu_g \\
 ={}&  \int_{M_t}  \sum_{m \leq k} 4 C_k^m (n-3)^{k-m} t  \nabla^{\mathring{m}} \Delta^{ [\frac{m}{2}] } \Si_{p q} W^{p i q j} \cdot \( \lie_{\p_t} ( t \nabla^{\mathring{m}} \Delta^{ [\frac{m}{2}] } \Si_{i j} )  \) \di \mu_g \\
 & + \int_{M_t}  \sum_{m \leq k} 4 C_k^m (n-3)^{k-m} t  \nabla^{\mathring{m}} \Delta^{ [\frac{m}{2}] } \Si_{p q} W^{p i q j} \ast \(  \text{l.o.t.}  \) \di \mu_g \\
  ={}& \sum_{m \leq k} 2 C_k^m (n-3)^{k-m} \cdot \p_t   \int_{M_t}  \( t^2  \nabla^{\mathring{m}} \Delta^{ [\frac{m}{2}] } \Si_{p q}   \nabla^{\mathring{m}} \Delta^{ [\frac{m}{2}] } \Si_{i j}  W^{p i q j} \) \di \mu_g \\
  & + f_4^k + f_5^k,
  } 
  with
\alis{
  f_4^k ={}&  \sum_{m \leq k} \int_{M_t} t  \nabla^{\mathring{m}} \Delta^{ [\frac{m}{2}] } \Si_{p q}  \nabla^{\mathring{m}} \Delta^{ [\frac{m}{2}] }  \Si_{i j}  \ast \n \Hw \, \di \mu_g  \\
    f_5^k={}&  \sum_{m \leq k} \int_{M_t} t  \nabla^{\mathring{m}} \Delta^{ [\frac{m}{2}] } \Si_{p q}  \nabla^{\mathring{m}} \Delta^{ [\frac{m}{2}] }  \Si_{i j}  \ast  \Ew \, \di \mu_g  \\
  &+  \sum_{m \leq k} \int_{M_t} t  \nabla^{\mathring{m}} \Delta^{ [\frac{m}{2}] }  \Si \ast   \nabla^{\mathring{m}} \Delta^{ [\frac{m}{2}] }  \Si \ast  \(\Si, \, \eta \) \ast \(\Si, \, \Ew, \, W, \, \Si \ast \Si\)  \di \mu_g  \\
   &+  \sum_{m \leq k} \int_{M_t}2 t  \nabla^{\mathring{m}} \Delta^{ [\frac{m}{2}] } \Si_{p q} W^{p i q j} \nabla^{\mathring{m}} \Delta^{ [\frac{m}{2}] } \( \Sigma_{i m} \Sigma_{j}^m\) \di \mu_g  \\
 &+  \sum_{m \leq k} \int_{M_t} t \nabla^{\mathring{m}} \Delta^{ [\frac{m}{2}] } \Si \ast W \ast \sum_{a+b=m-1} \n_{I_a} (\n \Si, \, \n \eta) \ast \n_{I_b} \Si \,\di \mu_g.
} 
Similar to $f^k_2$, here $f^k_4$ and $f^k_5$ are both lower-order terms. However, note that $\| \n \Hw\|_{L^\infty}$ is not bounded due to an issue of regularity. Hence, we apply integration by parts to $f_4^k$ so that \[ f_4^k=  \sum_{m \leq k}\int_{M_t}  t  \nabla^{m+1} \Si \ast \nabla^{\mathring{m}} \Delta^{ [\frac{m}{2}] }  \Si   \ast \Hw \, \di \mu_g, \] and then $f^k_4$ and $f^k_5$ can be estimated in a straightforward way,
\ali{ee-f45k}{
|f_4^k| + |f_5^k| \lesssim {} &  t^{-1} \| \Hw \|_{L^\infty}  \|t\Si\|_{H'_{ k}}  \|t\Si\|_{H_{k+1}} +   t^{-1}  \|t\Si\|^2_{H'_{ k}} \| \Ew\|_{L^\infty} \nnb \\
 &+  t^{-1}  \|t\Si\|^2_{H'_{ k}} \| \(\Si, \, \eta \) \ast \(W,\, \Ew, \,\Si, \, \Si \ast \Si\) \|_{L^\infty} \nnb \\
   & +  t^{-1}  \( \|t\Si\|^2_{H_{ k}}  \| \Si \|_{H_N} + \|t\Si\|_{H_{ k}}  \| t \Si \|_{H_N} \|\eta \|_{H_k} \) \|W\|_{L^\infty} \nnb \\
    \lesssim {}& \varepsilon^3 \Ld^3  t^{-2+3\de}. 
}

After all the above estimates, at this stage, we have transformed \eqref{E:id-energy-ho-EHw-1} into the following form
\ali{id-energy-ho-EHw-2}{
& \p_t \( 2  \|  t\nabla^{\mathring{k}} \Delta^{ [\frac{k}{2}] } \Ew\|^2 + \| t \nabla^{\mathring{k}} \Delta^{ [\frac{k}{2}] } \Hw\|^2 \) \nnb \\
& \quad + \p_t \(   \sum_{m<k}  2 C_k^m  (n-3)^{k-m} \| t\nabla^{\mathring{m}} \Delta^{ [\frac{m}{2}] } \Ew \|^2 \) \nnb \\
 & \quad  + \sum_{l \leq k}   4 C_k^l  (n-3)^{k+1-l} t \|\nabla^{\mathring{l}} \Delta^{ [\frac{l}{2}] } \Ew \|^2 \nnb \\
={} & \sum_{m \leq k} 2 C_k^m (n-3)^{k-m} \cdot \p_t   \int_{M_t} \( t^2  \nabla^{\mathring{m}} \Delta^{ [\frac{m}{2}] } \Si_{p q}   \nabla^{\mathring{m}} \Delta^{ [\frac{m}{2}] } \Si_{i j}  W^{p i q j} \) \di \mu_g \nnb \\
& +  f_2^k +  f^k_3+ f_4^k +f^k_5 + f^k_6,
}
where $f_i^k$, $i=2:5$ are defined as before, and $f_6^k$ is the lower-order (in $\Si$) parts in \eqref{E:split-leading-term},  \[ f_6^k = \int_{M_t}  \sum_{m\leq k}\sum_{a+b = m-1} \n_{I_a} \Si \ast \n_{I_b} \n W \ast t \nabla^{\mathring{m}} \Delta^{ [\frac{m}{2}] }  \Ew^{i j} \, \di \mu_g. \]  We find that $f^k_6$ contains the non-decaying $W$ as $f_3^k$,  and therefore shares a similar estimate,
\ali{ee-f6k}{
|f_6^k | \lesssim {}& \sum_{a+b \leq k-1} t \| \n_{I_a} \Si \ast \n_{I_b} \n W \|  \| \Ew\|_{H^\prime_{ k}}  \nnb \\
\lesssim {}& t \|\Si\|_{H_k} \|\n W\|_{H_{N-1}} \|\Ew\|_{H^\prime_{k}}. 
}
In summary, among the nonlinear terms, $f_2^k, \,  f_4^k, \, f_5^k$ are all lower-order terms estimated in \eqref{E:ee-f1k} and \eqref{E:ee-f45k}, while $f_3^k, \, f_6^k$ are the main terms involving $W$ and  their estimates \eqref{E:ee-f2k}, \eqref{E:ee-f6k} are summarized as
\beq\label{ee-f236k}
 |f_3^k| + |f_6^k| \lesssim t \| W\|_{H_{N}} \|\Ew\|_{H'_{k}}   \( \|\Si\|_{H_k}  +  \|\Hw\|_{H_{k-1}}  \).
\eeq
With the same reason as \eqref{E:ee-direct-leading}, \eqref{ee-f236k} is not allowed to be estimated directly, otherwise, the inductive argument would fail. In what follows, we will see that the positive terms on the left side of \eqref{E:id-energy-ho-EHw-2} is crucial for the analysis of \eqref{ee-f236k}.

After integrated over $[t_0, t]$, the energy identity \eqref{E:id-energy-ho-EHw-2} becomes
\ali{ineq-energy-ho-EHw-1}{
& 2  \|  t\nabla^{\mathring{k}} \Delta^{ [\frac{k}{2}] } \Ew\|^2 + \| t \nabla^{\mathring{k}} \Delta^{ [\frac{k}{2}] } \Hw\|^2 \nnb   \\
& \quad + \sum_{l < k}  2 C_k^l  (n-3)^{k-l}  \| t \nabla^{\mathring{l}} \Delta^{ [\frac{l}{2}] } \Ew \|^2  \nnb \\
 & \quad +  \sum_{m \leq k}   4 C_k^m  (n-3)^{k+1-m} \int_{t_0}^t t’ \|\nabla^{\mathring{m}} \Delta^{ [\frac{m}{2}] } \Ew \|^2 \, \di t'  \nnb \\
\lesssim {} &  \sum_{m \leq k} 2 C_k^m (n-3)^{k-m} \cdot \Big| \int_{M_t}  t^{ 2}  \nabla^{\mathring{m}} \Delta^{ [\frac{m}{2}] } \Si_{p q}   \nabla^{\mathring{m}} \Delta^{ [\frac{m}{2}] } \Si_{i j}  W^{p i q j} \di \mu_g \Big|  \nnb \\
&+ \varepsilon^2 I^2_{k+2} + \varepsilon^3 \Ld^3 + \int_{t_0}^t t' \| W\|_{H_{N}} \|\Ew\|_{H'_{k}}   \( \|\Si\|_{H_k}  +  \|\Hw\|_{H_{k-1}}  \) \di t'.
}
By the Cauchy-Schwarz inequality,
\ali{ee-regularity-solve}{
 & \int_{t_0}^t t' \| W\|_{H_{N}} \|\Ew\|_{H'_{k}}   \( \|\Si\|_{H_k}  +  \|\Hw\|_{H_{k-1}}  \) \di t' \nnb \\
\leq {} &  \int_{t_0}^t  a t' \| \Ew\|_{H'_{k}}^2  \, \di t' +  \int_{t_0}^t a^{-1} t'  \| W\|_{H_{N}}^2 \( \|\Si\|_{H_k}^2 +  \|\Hw\|^2_{H_{k-1}}  \) \di t'.
}
We choose the constant $a$ to be properly small so that $\int_{t_0}^t  a t' \| \Ew\|_{H'_{k}}^2 \di t'$ can be absorbed by the extra positive terms $$\sum_{m \leq k}  4  C_k^m  (n-3)^{k+1-m} \int_{t_0}^t t' \|\nabla^{\mathring{m}} \Delta^{ [\frac{m}{2}] } \Ew \|^2 \di t' $$ on the left side of \eqref{E:ineq-energy-ho-EHw-1}, noting that $n \geq 4$. The second term on the right hand side of \eqref{E:ee-regularity-solve} involves only linear terms whose estimates have already been updated.
In practice, by the improvement for $\|W\|_{H_N}$ \eqref{ee-W-N}, the inductive assumption \eqref{induction-Ew-Hw}, and the enhanced estimates for $ \|\Si\|_{H_k}$ ($k\leq N$) \eqref{E:esti-Si} followed, the remaining term in \eqref{E:ee-regularity-solve} admits the bound
\alis{
&  \int_{t_0}^t a^{-1}   t'  \| W\|_{H_{N}}^2 \( \|\Si\|_{H_k}^2 +  \|\Hw\|^2_{H_{k-1}}  \) \di t' \\
\lesssim {}&   \int_{t_0}^t    t^{\prime -1+2\de}  \(  I^2_2 + \vep^2 \Ld^2 \)   \vep \Ld \left(  \varepsilon I_{N+2}  +  \varepsilon^2 \Lambda^2 \right)  \di t' \\
\lesssim {}&  \varepsilon \Ld \( \vep I_{N+2} + \vep^2 \Ld^2\) t^{2 \delta}.
}
In the same way, the first term on the right side of \eqref{E:ineq-energy-ho-EHw-1} is bounded by
\alis{
 &   \sum_{m \leq k} 2 C_k^m (n-3)^{k-m} \Big| \int_{M_t}  t^2  \nabla^{\mathring{m}} \Delta^{ [\frac{m}{2}] } \Si_{p q}   \nabla^{\mathring{m}} \Delta^{ [\frac{m}{2}] } \Si_{i j}  W^{p i q j} \di \mu_g \Big|  \\
  \lesssim {} &  \|t \Si \|_{H_k}^2 \| W \|_{L^\infty} \\
 \lesssim {} & \left(\varepsilon^2  I^2_{k+1}  + \varepsilon^2 \Ld I_{N+2} +  \varepsilon^3 \Lambda^3 \right) t^{2\delta}  \(I_{N+2} + \vep  \Ld  \) \\
 \lesssim {} &  \varepsilon \Ld \( \vep I_{N+2} + \vep^2 \Ld^2\) t^{2 \delta}.
}
As a consequence, we obtain the following energy estimates,
\alis{
& 2  \|  t\nabla^{\mathring{k}} \Delta^{ [\frac{k}{2}] } \Ew\|^2 + \| t \nabla^{\mathring{k}} \Delta^{ [\frac{k}{2}] } \Hw\|^2   \\
& \quad + \sum_{l < k}  2 C_k^l  (n-3)^{k-l}  \| t \nabla^{\mathring{l}} \Delta^{ [\frac{l}{2}] } \Ew \|^2  \\
 & \quad +  \sum_{m \leq k}    C_k^m  (n-3)^{k+1-m} \int_{t_0}^t t^\prime \|\nabla^{\mathring{m}} \Delta^{ [\frac{m}{2}] } \Ew \|^2 \di t^\prime  \\
\lesssim {}&  \vep^2 I_{k+2}^2 +   \varepsilon \Ld \( \vep I_{N+2} + \vep^2 \Ld^2\) t^{2 \delta}.
}
Thus for $k\leq N$, there is a constant $C(N, \, n)$ depending only on $N$ and the dimension $n$ such that
\alis{
&  \|  t \Ew\|^2_{H_k} + \| t  \Hw\|^2_{H_k}  
\leq C(N, \, n) \(   \vep^2 I_{k+2}^2 +   \varepsilon \Ld \( \vep I_{N+2} + \vep^2 \Ld^2\) t^{2 \delta} \).
}
The inductive proof is completed.

\subsection{Estimates for  $\Si$ and the top-order derivative of $\eta$}
\begin{proposition}\label{prop-energy-estimate-2nd} 
We derive from the bootstrap assumptions \eqref{bt-Si-eta}--\eqref{bt-Jw} that
\begin{align*}
t^2 \| \Sigma\|^2_{H_{N+1}}   \lesssim {}&  \varepsilon^2 I_{N+2}^2 + \vep \Ld \left(   \vep I_{N+2}   +  \varepsilon^2 \Lambda^2 \right) t^{2 \delta}, \\
t^2 \| \eta\|^2_{H_{N+2}} \lesssim {}& \varepsilon^2 I^2_{N+2}+ \varepsilon^4 \Lambda^4.
\end{align*}
\end{proposition} 
\begin{proof}
The lower-order estimates of $\| \Sigma\|_{H_{N}}$ and $\|\eta\|_{H_{N+1}}$ have been achieved in \eqref{E:esti-Si} and Lemma \ref{prop-energy-estimate-eta-low} respectively.

The estimate for $\| \Sigma\|^2_{H_{N+1}}$ follows in an analogous fashion as \eqref{E:esti-Si}, except that now we have to substitute \eqref{ee-Sigma} and the updated estimates for $\|\Hw\|_{H_{N}}$ (Proposition \ref{pro-est-Ew-Hw}) and $\|\eta\|_{H_{N+1}}$ (Lemma \ref{prop-energy-estimate-eta-low}) into \eqref{E:esti-Si} with $k=N+1$.

Next, we follow the idea in \cite{Fournodavlos-Luk-Kasner} (referring to \cite[Proposition 3.10]{W-Y-open-M} as well) to retrieve an estimate for $\|\nabla_{I_{N+2}} \eta\|$. Namely, based on the transport equation of $\eta$ \eqref{eq-evolution-eta} and the wave equation for $\Si$ \eqref{E:wave-Sigma}, we use the technic of renormalization and elliptic estimates to improve the regularity of $\eta$.

Applying $\nabla_{I_{N}} \Delta$ on \eqref{eq-evolution-eta} and commuting it with $\p_\tau$, we have
\alis{
 \p_{\tau}  \nabla_{I_{N}} \Delta \eta +   \nabla_{I_{N}} \Delta \eta &=  2 \eta  \nabla_{I_{N}} \Delta \eta +  \frac{2}{n}  \nabla_{I_{N}} \Delta \Sigma_{i j} \cdot \Sigma^{i j}  \nnb \\
& \quad + \sum_{ \substack{a+b\leq N+2\\ a, \, b \leq N+1} } ( \nabla_{I_a} \Sigma, \nabla_{I_a} \eta) * ( \nabla_{I_b} \Sigma, \nabla_{I_b} \eta).
}
As a remark, the term $\frac{2}{n}  \nabla_{I_{N}} \Delta \Sigma_{i j} \cdot \Sigma^{i j}$ on the right side is not bounded due to the restriction of regularity for $\Si$. Fortunately, the wave equation for $\Si$ \eqref{E:wave-Sigma} enables us to reformulate this term as
\alis{
 \nabla_{I_{N}} \Delta \Sigma_{i j} \cdot \Sigma^{i j} = {}& 
\dtau \( \dtau  \nabla_{I_{N}} \Sigma_{i j} \Sigma^{i j} + (n-1)  \nabla_{I_{N}} \Sigma_{i j} \Sigma^{i j} \) \\
&   + n \nabla_{I_{N}} \nabla_i \nabla_j \eta \Sigma^{i j} + f_1, 
}
where under the bootstrap assumptions \eqref{bt-Si-eta}--\eqref{bt-Jw},  \[  \|f_1\|  \lesssim \varepsilon^2 \Lambda^2 t^{-2+2\delta}. \] 
By means of defining the modified variable
\begin{equation}\label{def-ti-eta}
  \tilde \eta_{N+2} =\nabla_{I_N} \Delta \eta - \frac{2}{n} \left( \dtau  \nabla_{I_{N}} \Sigma_{i j} \Sigma^{i j} +  (n-1) \nabla_{I_N} \Sigma_{i j} \Sigma^{i j} \right),
  \end{equation}
we deduce a transport equation for $\tilde \eta_{N+2}$,
\begin{equation}\label{pt-teta}
 \p_{\tau}   \tilde \eta_{N+2} +  \tilde \eta_{N+2}  = 2 \nabla_{I_N} \Delta \eta \cdot \eta  +2 \Sigma \nabla_{I_{N+2}} \eta +   f_2,
\end{equation}
with \[ \|f_2\|  \lesssim \varepsilon^2 \Lambda^2 t^{-2+2\delta}. \] In particular, in \eqref{pt-teta}, the regularity issue arising from $ \nabla_{I_{N}} \Delta \Si$ disappear.
Viewing Corollary \ref{prop-elliptic-Delta-1} and the definition of $\tilde \eta_{N+2}$ \eqref{def-ti-eta}, we have
\ali{compare-eta-teta}{
\|\nabla_{I_{N+2}} \eta\| \lesssim {}&  \|\nabla_{I_{N}} \Delta \eta\|  + \| \eta\|_{H_{N+1} }\nnb \\
\lesssim {} & \|\tilde \eta_{N+2}\|  + \vep t^{-1} + \varepsilon^2 \Lambda^2 t^{-2+2\delta}.
}
Then \eqref{pt-teta} yields the energy inequality,
\ali{energy-inequ-teta}{
t  \|\tilde \eta_{N+2}\| \lesssim {} & \varepsilon I_{N+2} + \varepsilon^2 \Lambda^2 + \int_{t_0}^t \varepsilon \Lambda t^{\prime -2+\delta} \cdot t'\|\tilde \eta_{I_{N+2}}\| \, \di t'.
}
It follows from the Gr\"{o}nwall's inequality that
\[
t\|\tilde \eta_{I_{N+2}} \| \lesssim  \varepsilon I_{N+2}+ \varepsilon^2 \Lambda^2.
\]
Using \eqref{E:compare-eta-teta} again, 
we accomplish the estimate for the top order of $\eta$.
\end{proof}

\begin{remark}\label{rk-ee-other}
Based on the propositions \ref{pro-est-Ew-Hw} and \ref{prop-energy-estimate-2nd}, a further usage of the equations \eqref{eq-evolution-eta}--\eqref{eq-evolution-Si} and \eqref{Gauss-Ricci-hat-k} implies
\begin{align*}
 t^2 \| \dtau g_{ij} \|^2_{H_{N+1}} & \lesssim  \left(  \varepsilon^2 I_{N+2}^2 +  \vep \Ld \left(   \vep I_{N+2}   +  \varepsilon^2 \Lambda^2 \right) \right) t^{2\delta}, \\ 
\|g_{ij} - \gamma_{i j}\|^2_{H_{N+1}} & \lesssim \varepsilon^2 I_{N+2}^2 + \vep \Ld \left(   \vep I_{N+2}   +  \varepsilon^2 \Lambda^2 \right), \\ 
 t^2 \|\dtau \Sigma\|^2_{H_{N}} & \lesssim  \left(  \varepsilon^2 I_{N+2}^2 + \vep \Ld \left(   \vep I_{N+2}   +  \varepsilon^2 \Lambda^2 \right) \right) t^{2\delta}, \\ 
  t^2 \|\dtau \eta \|^2_{H_{N}} & \lesssim \varepsilon^2  I_{N+2}^2 + \varepsilon^4 \Lambda^4, \\ 
t^2 \| R_{ij} + (n-1) g_{ij } \|^2_{H_{N}} & \lesssim  \left(  \varepsilon^2 I_{N+2}^2  +  \vep \Ld \left(   \vep I_{N+2}   +  \varepsilon^2 \Lambda^2 \right) \right) t^{2\delta}. 
\end{align*} 
\end{remark}

\subsection{Estimates for $g-\ga$}
We have derived the bound for $\|g_{ij} - \gamma_{i j}\|_{H_{N+1}}$ (Remark \ref{rk-ee-other}). The estimate for the top-order $\|\nabla_{I_{N+2}} (g- \ga)\|$ follows in the same way as  $\|\nabla_{I_{N+2}}\eta\|$. Observe that with the wave equation for $\Si$ \eqref{E:wave-Sigma}, the evolution equation \eqref{eq-evolution-g} can be rewritten as,
\alis{
\dtau \( \De (g-\ga)_{i j} + 2 \dtau \Si_{i j}  \) ={}& n \n_i \n_j \eta -2\De \eta g + 2 (n-1) \dtau \Sigma_{ij} -4 \Sigma_{ij}\\
& + \Jw\ast \Si + (\dtau \Si, \Si, \eta)\ast (\Si, \eta) + (\Si, \eta)^3.
}
Following the proof of Proposition \ref{prop-energy-estimate-2nd}, we obtain the bound 
\beq\label{ee-g-ga}
\|g-\ga\|^2_{H_{N+2}} \lesssim \varepsilon^2 I_{N+2}^2 + \vep \Ld \left(   \vep I_{N+2}   +  \varepsilon^2 \Lambda^2 \right). 
\eeq

\subsection{Closure of the bootstrap argument}
Combining the conclusions of the propositions \ref{prop-W}, \ref{pro-est-Ew-Hw} and \ref{prop-energy-estimate-2nd}, and the estimate \eqref{ee-g-ga}, we can choose $\Ld > C I_{N+2}$ large and $\vep$ small enough (depending on $I_{N+2}$), so that the bootstrap assumptions \eqref{bt-Si-eta}--\eqref{bt-Jw} hold with the constant $\Lambda$ replaced by $\frac{1}{2} \Lambda$. We thus complete the proof of Theorem \ref{thm-EE}.

\appendix

\section{Local existence theorem}\label{sec-local}

The Einstein equations \eqref{eq-evolution-g}--\eqref{eq-evolution-Si}, \eqref{Gauss-Ricci-trace-hat -k}--\eqref{Codazzi-div-k} in Gaussian normal coordinates \eqref{metric-form-tau} over $(\mathcal{M}, \,\breve g)$ are composed  of  the evolution equations
\begin{subequations}
\begin{align}
\p_t \tilde g_{ij} & = -2 \tilde k_{i j}, \label{pt-g} \\
\p_t \tilde k_{ij}& = \tilde R_{ij} - 2 \tilde k_i^p \tilde k_{jp} + \text{tr}_{\tilde g} \tilde k \tilde k_{ij}, \label{pt- k}
\end{align}
\end{subequations}
and the constraint equations
\begin{subequations}
\begin{align}
\tilde R -|\tilde k|^2 + (\tr_{\tilde g} \tilde k)^2  &= 0, \label{Constrain-Guass} \\
\tilde \nabla^i \tilde k_{ij} - \tilde \nabla_j \tr_{\tilde g} \tilde k &= 0. \label{Constrain-Codazzi}
\end{align}
\end{subequations}
This system implies a wave type equation for $\tilde k$\footnote{Take $\p_t$ derivative on \eqref{pt- k}.} and letting $$h := \tr_{\tilde g} \tilde k$$ be a separate variable, one ends up with the following reduced system \cite{Fournodavlos-Luk-Kasner},
\ali{reduced-sys}{
\p_t \tilde g_{ij} & = -2 \tilde k_{i j}, \nnb \\
\p_t h & = |\tilde k|^2, \nnb \\
-\p_t^2 \tilde k_{ij} + \Delta_{\tilde g} \tilde k_{ij}  &= \tilde \nabla_i \tilde \nabla_j h  - 2 \ti R_{i}{}^{m}{}_{j}{}^{l} \ti k_{m l} + \ti R_{i m} \ti k_{j}^n + \ti R_{j m} \ti k_i^m \nnb \\
& \,\,\, + \p_t \tilde k \ast \tilde k + \tilde k \ast \tilde k \ast \tilde k.
}

Following \cite{Fournodavlos-Luk-Kasner}, one can show that the reduced system \eqref{E:reduced-sys} and  the vacuum Einstein equations \eqref{pt-g}--\eqref{Constrain-Codazzi} are equivalent, if the data of  \eqref{E:reduced-sys} are those induced from \eqref{pt-g}--\eqref{Constrain-Codazzi}.

\begin{lemma}\label{equi-reduce-EKG}
If $(\tilde g, h, \tilde k)$ is any solution of the reduced system \eqref{E:reduced-sys} whose initial data $(\tilde g, h, \tilde k, \p_t \tilde k)|_{t=t_0} =(\tilde g_0, h_0, \tilde k_0, \tilde k_1)$ verify the constraint equations 
\eqref{Constrain-Guass}--\eqref{Constrain-Codazzi}, and $h_0 = \tr_{\tilde g_0} \tilde k_0$,
then $\breve g = -dt^2 + \tilde g$ is a solution to the vacuum Einstein equations \eqref{pt-g}--\eqref{Constrain-Codazzi}.
\end{lemma}

\begin{proof}
Observe that, in the third equation of \eqref{E:reduced-sys}, the term involving curvature $- 2 \ti R_{i}{}^{m}{}_{j}{}^{l} \ti k_{m l} + \ti R_{i m} \ti k_{j}^n + \ti R_{j m} \ti k_i^m$ is trace-free. Consequently, one can take trace on the wave equation of $\tilde k$ and substitute it into the second equation of \eqref{E:reduced-sys} to derive a homogeneous wave equation for $\tr_{\tilde g} \tilde k - h$ and thus infer that $\tr_{\tilde g} \tilde k = h$. 
The rest of proof follows in the same manner as \cite{Fournodavlos-Luk-Kasner}, we omit the details.
\end{proof}

After that, one can use the reduced system \eqref{E:reduced-sys} to prove a local existence theorem  \cite{Fournodavlos-Luk-Kasner}. 
\begin{theorem}\label{thm-local-existence}
Let $(M, \, g_0)$ be a smooth complete Riemannian manifold with positive injective radius and its Ricci curvature is bounded from below. We fix an integer $N > \frac{n}{2}$, and suppose $\ga$ is an Einstein metric on $M$ with negative Einstein constant.
 
Let $(g_0, \, k_0)$ be the data  on $\{t_0\} \times M$, $t_0 >0$, for the rescaled vacuum Einstein equations and decompose the  symmetric $(0,2)$-tensor $k_{0ij}$ into the trace-free and trace parts: $k_{0ij} = \Sigma_{0ij} + \frac{\tr_{g_0} k_0}{n} g_0$, with $\tr_{g_0} k_0 = g_0^{i j} k_{0 ij}$.  
If the data verify that 
\alis{
& g_{0} - \gamma \in H_{N+2} (M, \, g_0), \quad  \Sigma_{0} \in H_{N+1}(M, g_0), \\
&  \tr_{g_0} k_0 + n  \in  H_{N+2}(M, g_0),
}
then there is a unique, local-in-time development $(\mathcal{M}, \, \breve g)$ with 
\[\mathcal{M} = [t_{0}, t_{\ast}] \times M, \quad \breve g = -dt^2 + t^2 g(t),\] 
and $t =t_0$ corresponding to the initial slice $(M, \, t_0^2 \cdot g_0)$. Moreover, denoting $\Sigma_{ij}$, $\tr_g k$, the trace-free and trace parts of $k_{ij}$ respectively,
we have
\alis{
g_{ij}(t) - \gamma_{ ij} & \in C^1([t_0, t_\ast], H_{N+2}(M, g_0)), \\
\Sigma_{ij} (t) & \in C^1([t_0, t_\ast], H_{N+1}(M, g_0)), \\
 \tr_g k  (t) + n & \in C^1([t_0, t_\ast], H_{N+2}(M, g_0)). 
}
\end{theorem}

\section{The density theorem}\label{sec-density}

We recall the density theorem from \cite{Hebey}.
\begin{proposition}\label{prop-density}
Let $(M, g)$ be a smooth, complete Riemannian manifold.  

\begin{itemize}
\item  For any $p \geq 1$, $H_{0,1}^p(M) = H^p_1(M)$.
\item Assume that $(M, g)$ has positive injective radius, and $|\nabla^j R_{m n}|$, $j = 0, \cdots, K-2$, is bounded, where $K \geq 2$ is an integer. Then for any $p \geq 1$, $H_{0, K}^p(M) = H^p_K(M)$.
\end{itemize}
\end{proposition}

We will make use of the above results to establish a density theorem for our purpose.

\begin{proposition}\label{pro-density}
Let $(M, g)$ be an $n$-dimensional smooth, complete Riemannian manifold with positive injective radius and its Ricci curvature is bounded from below. 

Fix an integer $N > \frac{n}{2}$. Suppose $\gamma$ is an Einstein metric on $M$ with negative Einstein constant and assume that \[g  - \gamma  \in H_{N+2} (M, g),  \] then
\alis
{H_{0,k}(M) &= H_{k}(M),  \quad k \leq N+2.
}

\end{proposition}
\begin{proof}
Note that\footnote{We use the notation $\nabla[g]$ ($\nabla[\ga]$) to point out the covariant derivative corresponds to the metric $g$ ($\ga$).} $\nabla[\gamma] g = \nabla[g] \gamma * g$. Then $g - \gamma \in H_{N+2}$, $N > \frac{n}{2}$, implies
\alis{
\sum_{1\leq k \leq N+2}\| \(\nabla[g]\)_{I_k} \gamma_{i j}\|_{L^2} + \sum_{1 \leq k\leq N+2} \| \(\nabla [\gamma]\)_{I_k} g_{i j}\|_{L^2(M, \gamma)} & \leq C.
}
Moreover, for any tensor field $\Ps$ on $M$, the following identities hold
\alis{
\(\n[g]\)_{I_k} \Psi ={}& (\n[\ga])_{I_k} \Psi + \sum_{a+b=k-1} (\n[\ga])_{I_a} \n[\ga] g \ast (\n[\ga])_{I_b} \Psi, \\
\(\n[\ga]\)_{I_k} \Psi ={}& (\n[g])_{I_k} \Psi + \sum_{a+b=k-1} (\n[g])_{I_a} \n[g] \ga \ast (\n[g])_{I_b} \Psi. 
}
Therefore, we can prove, using the Sobolev inequalities, that for any tensor field $\Psi$ on $M$, the two norms are equivalent
\alis{
\| \Psi\|_{H_{k} (M, g)} & \sim \| \Psi\|_{H_{k} (M, \gamma)}, \quad k\leq N+2, 
}
and hence
\alis{
 H_{k} ( M, g) &=  H_{k} (M, \gamma), &&  \quad k\leq N+2,  \\
  H_{0,k} ( M, g) &=  H_{0, k} (M, \gamma), &&   \quad k\leq N+2.
 }
By the density theorem on $(M, \gamma)$ (referring to Proposition \ref{prop-density}), \[H_{0,k} (M, \gamma) =  H_{k} ( M, \gamma), \quad k \leq N+2, \] we conclude the claim.
\end{proof}

 \section{Some identities}
 \subsection{Commuting identity}\label{sec-comm}
Let $\Gamma_{ij}^a$ be the connection coefficient of $\nabla$. Then the Lie derivative $\lie_{\dtau} \Gamma_{ij}^a$ is a tensor field
\begin{equation}\label{dtau-Gamma}
\begin{split}
\lie_{\dtau} \Gamma^a_{ij}& =\frac{1}{2}g^{ab} \left( \nabla_i \lie_{\dtau} g_{jb} + \nabla_j \lie_{\dtau} g_{ib} -\nabla_b \lie_{\dtau} g_{ij} \right).
\end{split}
\end{equation}

A commuting identity between $\nabla$ and $\lie_{\dtau}$ is given as follows.
\begin{lemma}\label{lem-commu-lie}
Let $\Psi$ be an arbitrary $(0, k)$-tensor field on $(M, g)$. The following commuting formula is true:
\begin{equation}\label{commuting-lie-nabla}
\lie_{\dtau} \nabla_j \Psi_{a_1 \cdots a_k} = \nabla_j \lie_{\dtau} \Psi_{a_1 \cdots a_k} - \sum_{i=1}^k \lie_{\dtau} \Gamma^p_{j a_i} \Psi_{a_1 \cdots p \cdots a_k}.
\end{equation}
\end{lemma}
This lemma can be proved by straightforward calculations.
We apply Lemma \ref{lem-commu-lie} to $\nabla_{I_l}  \Psi$, and take \eqref{eq-evolution-g} into account. It then follows that
\begin{lemma}\label{lemma-commuting-application}
Let $l \geq 1$ be an integer. Then 
\begin{align}
  \lie_{\dtau}  \nabla_{I_l} \Psi &=  \nabla_{I_l} \lie_{\dtau}  \Psi +  [ \dtau, \nabla_{I_l}] \Psi, \label{id-commuting-nabla-N-T-l-simplify}
\end{align}
where 
\begin{align}
  [\dtau, \nabla_{I_l}] \Psi &= \sum_{a+1 +b  = l}  \nabla_{I_a}  \left( \nabla \Sigma, \, \nabla \eta \right) * \nabla_{I_{b}}  \Psi, \quad l \geq 1. \label{def-commuting-KN-l}
\end{align}
\end{lemma}

We also present a commuting identity between $\nabla$ and $\Delta$, which can be proved by induction. 
\begin{lemma}\label{lemma-commuting-nabla-laplacian}
 Let  $l \geq 1$ be an integer. For any $(0, r)$-tensor $\Psi_{I_r},$
\begin{align}
 \Delta \nabla_{I_l} \Psi  &= \nabla_{I_l} \Delta \Psi +  \sum_{a+b = l} \nabla_{I_a} R_{imjn} * \nabla_{I_b} \Psi. \label{Def-R-Psi-ij-commute-nabla-laplacian}
\end{align}
In particular,
\ali{comm-2}{
&\nabla_a \Delta \Psi_{I_r} = \Delta \nabla_a \Psi_{I_r} - R_a^p \nabla_p \Psi_{I_r} \nnb\\
& \quad + \sum_{k=1}^r 2 R_{a p i_k}{}^{i_q} \nabla^p \Psi_{i_1 \cdots i_q \cdots i_r} + \sum_{k=1}^r  \nabla^p R_{a p i_k}{}^{i_q} \Psi_{i_1 \cdots i_q \cdots i_r}.
}
\end{lemma}

\subsection{Equivalence between two Sobolev norms}\label{subsec-equi-norms}
We will use the commuting identities in the subsection \ref{sec-comm} to demonstrate the equivalence between $\|\cdot \|_{H_k}$ and $\| \cdot \|_{H'_{k}}$.
\bpf[Proof of Proposition \ref{prop-elliptic-Delta}]
When $k=1$, the two norms $\|\cdot \|_{H_k}$ and $\| \cdot \|_{H'_{k}}$ are identical to each other.

When $k=2$, we take a $(0,1)$-tensor $\Psi$ for instance to illustrate the idea.
 \begin{align*}
& \int_{M} |\nabla^2 \Psi |^2 \, \di \mu_g = -  \int_{M} \nabla^j \Psi^p \Delta \nabla_j  \Psi_p \, \di \mu_g \\
& \stackrel{\eqref{E:comm-2}}{=} -  \int_{M} \nabla^j \Psi^p \left( \nabla_j \Delta \Psi_{p } + R_{j}^i \nabla_i \Psi_p - 2 R_{j a p}{}^{b} \nabla^a \Psi_{b } -  \nabla^a R_{j a p}{}^{b} \Psi_{b} \right) \di \mu_g \\
& =  \int_{M} \( |\Delta \Psi|^2 -  R_{j}^i \nabla^j \Psi  \nabla_i \Psi +  R_{j a p}{}^{b} \nabla^a \Psi_{b } \nabla^j \Psi^p - R_{j a p}{}^{b}  \Psi_{b}  \nabla^a \nabla^j \Psi^p \) \di \mu_g,
 \end{align*}
 where we have used integration by parts in the last identity.
It follows that
  \begin{align*}
\|\nabla^2 \Psi\|^2_{L^2} 
& \leq   \|\Delta \Psi\|^2_{L^2} +  C \|R_{imjn}\|_{L^\infty} \|\Psi\|^2_{H_1} \\
&\quad  +C(  a^{-1} \| R_{imjn} \|^2_{L^\infty}  \| \Psi\|^2_{L^2} + a \| \nabla^2 \Psi\|^2_{L^2} ).
 \end{align*}
 We take the constant $a$ such that $a C <\frac{1}{2}$ and then obtain
   \beq\label{Bochner-2}
\|\nabla^2 \Psi\|^2_{L^2} \lesssim  \|\Delta \Psi\|^2_{L^2} +  ( \|R_{imjn}\|_{L^\infty} + \| R_{imjn} \|^2_{L^\infty} ) \| \Psi\|^2_{H_1}.
 \eeq

 In general, for $k \geq 3$, 
\begin{align*}
 & \quad \int_{M} \nabla_{I_k} \Psi \nabla^{I_k} \Psi \, \di \mu_g  = -  \int_{M} \nabla_{I_{k-1}} \Psi \Delta \nabla^{I_{k-1}} \Psi \, \di \mu_g \\
& \stackrel{\eqref{Def-R-Psi-ij-commute-nabla-laplacian} }{=} -  \int_{M} \nabla_{I_{k-1}} \Psi \left(  \nabla^{I_{k-1}}  \Delta \Psi +  \sum_{a+b= k-1} \nabla_{I_a} R_{imjn} * \nabla_{I_b}  \Psi  \right) \di \mu_g \\
& =   \int_{M} \Delta \nabla_{I_{k-2}} \Psi  \nabla^{I_{k-2}}  \Delta \Psi \, \di \mu_g + \int_M  \sum_{a+b= k-1} \nabla_{I_a} R_{imjn} * \nabla_{I_b}  \Psi * \nabla_{I_{k-1}} \Psi \, \di \mu_g \\
&  \stackrel{\eqref{Def-R-Psi-ij-commute-nabla-laplacian} }{=}   \int_{M} |\nabla^{I_{k-2}}  \Delta \Psi |^2 \, \di \mu_g + \int_M \sum_{a+b= k-2} \nabla_{I_a} R_{imjn} * \nabla_{I_b}  \Psi * \nabla^{I_{k-2}}  \Delta \Psi \, \di \mu_g \\
& \quad + \int_M \sum_{a+b= k-1} \nabla_{I_a} R_{imjn} * \nabla_{I_b}  \Psi * \nabla_{I_{k-1}} \Psi \, \di \mu_g.
\end{align*}
Applying again integration by parts to \[ \int_M \sum_{a+b= k-2} \nabla_{I_a} R_{imjn} * \nabla_{I_b}  \Psi * \nabla^{I_{k-2}}  \Delta \Psi \, \di \mu_g,\] we have
\begin{align*}
\|\nabla_{I_k} \Psi\|^2_{L^2}  
& =  \| \nabla_{I_{k-2}} \Delta \Psi \|^2_{L^2}  + \int_M \sum_{a+b= k-1} \nabla_{I_a} R_{imjn} * \nabla_{I_b}  \Psi * \nabla^{I_{k-3}}  \Delta \Psi \, \di \mu_g \\
& \quad + \int_M \sum_{a+b= k-1} \nabla_{I_a} R_{imjn} * \nabla_{I_b}  \Psi * \nabla_{I_{k-1}} \Psi \, \di \mu_g.
\end{align*}
For terms like $\int_M \nabla_{I_{k-1}} R_{imjn} *  \Psi * \nabla_{I_{k-1}} \Psi \, \di \mu_g$, we apply integration by parts so that
\begin{align*}
 & \int_M \nabla_{I_{k-1}} R_{imjn} *  \Psi * \nabla_{I_{k-1}} \Psi \, \di \mu_g \\
 = & \int_M \nabla_{I_{k-2}} R_{imjn} * \nabla \Psi * \nabla_{I_{k-1}} \Psi + \nabla_{I_{k-2}} R_{imjn} * \Psi * \nabla_{I_{k}} \Psi \, \di \mu_g \\
  = & \int_M \nabla_{I_{k-3}} \n R_{imjn} * \nabla \Psi * \nabla_{I_{k-1}} \Psi + \nabla_{I_{k-3}} \n R_{imjn} * \Psi * \nabla_{I_{k}} \Psi\, \di \mu_g.
\end{align*}
Thus, in summary, we derive, for $k \geq 3$,
\begin{align*}
 \|\nabla_{I_k} \Psi \|^2_{L^2}  \lesssim {}&  \| \nabla_{I_{k-2}} \Delta \Psi \|^2_{L^2}  +  \|R_{imjn}\|_{L^\infty}  \| \Psi \|^2_{H_{k-1}} \\
 & + \sum_{a+b=k-3}  \|\nabla_{I_a} \nabla R_{imjn} \ast \nabla_{I_b} \nabla \Psi \|_{L^2}  \| \Psi \|_{H_{k-1}} \\
 &+ \| \nabla_{I_{k-3}} \n R_{imjn} * \Psi  \|_{L^2}  \| \nabla_{I_{k}} \Psi \|_{L^2}. 
\end{align*}
By Proposition \ref{prop-Multiplication}, it follows that if $N > \frac{n}{2}$, and $1\leq N-(k-3)$, that is, $k \leq N+2$, 
\begin{align}
  \|\nabla_{I_k} \Psi \|^2_{L^2} & \lesssim  \| \nabla_{I_{k-2}} \Delta \Psi \|^2_{L^2} +  \|R_{imjn}\|_{L^\infty}  \| \Psi \|^2_{H_{k-1}} \nnb \\
 & + \|\n R_{imjn} \|_{H_{N-1}} \| \n \Psi \|_{H_{k-2}} \| \Psi \|_{H_{k-1}} \nnb \\
 & + \|\n R_{imjn} \|_{H_{N-1}} \| \Psi \|_{H_{k-2}} \| \nabla_{I_{k}} \Psi \|_{L^2}. \label{ineq-Hk-Hwk}
\end{align}
Noting that 
\alis{
& \|\n R_{imjn} \|_{H_{N-1}} \| \Psi \|_{H_{k-2}} \| \nabla_{I_{k}} \Psi \|_{L^2} \\
 \leq & a^{-1}  \|\n R_{imjn} \|^2_{H_{N-1}} \| \Psi \|^2_{H_{k-2}} + a \| \nabla_{I_{k}} \Psi \|^2_{L^2},
 }
and choosing $a$ to be small so that $a \| \nabla_{I_{k}} \Psi \|^2_{L^2}$ can be absorbed by the left hand side of \eqref{ineq-Hk-Hwk}, we arrive at,
\alis{
  \|\nabla_{I_k} \Psi \|^2_{L^2} & \lesssim  \| \nabla_{I_{k-2}} \Delta \Psi \|^2_{L^2}  + \|R_{imjn}\|_{L^\infty}   \| \Psi \|^2_{H_{k-1}} \nnb \\
 & \quad + \left( \| \nabla R_{imjn} \|_{H_{N-1}} + \| \nabla R_{imjn} \|^2_{H_{N-1}} \right)  \| \Psi \|^2_{H_{k-1}},
} for $3 \leq k\leq N+2$.
By induction, we conclude that, for all $0 \leq k\leq N+2$,
\ali{Bochner-k}
 {\|\nabla_{I_k} \Psi \|^2_{L^2}  \lesssim {}& \| \nabla^{\mathring{k}} \Delta^{ [\frac{k}{2}] } \Psi \|_{L^2} \nnb\\
 &+ C \left( \|R_{imjn}\|_{L^\infty}, \, \| \nabla R_{imjn} \|_{H_{N-1}} \right)  \| \Psi \|^2_{ H'_{ k-1}},
}
where the constant $C \left( \|R_{imjn}\|_{L^\infty}, \, \| \nabla R_{imjn} \|_{H_{N-1}} \right)$ depends on $\|R_{imjn}\|_{L^\infty}$ and $\| \nabla R_{imjn} \|_{H_{N-1}}$.

 \epf      
 
 \subsection{Elliptic estimates}\label{sec-proof-lem-ellip}
Thanks to Proposition \ref{prop-elliptic-Delta}, we will prove Lemma \ref{lemm-elliptic}.
 \bpf[Proof of Lemma \ref{lemm-elliptic}]

In view of the anti-symmetric property for $\Ph$ \eqref{antisymm-Ph}, the equations \eqref{eq-elliptic-Ph} indicate that $\Ph$ obeys a Laplacian equation,
\alis{
 & \De \Ph_{imjn} = \n^p \n_p \Ph_{imjn} \\
 = {}&   - \n^p \( \n_i \Ph_{mpjn} + \n_m \Ph_{pijn} \) \pm \n B \\
 = {}& \n_i \n^p \Ph_{pmjn} - \n_m \n^p \Ph_{pijn} \pm \n B  + R_{imjn} \ast  \Ph \\
 = {}&  \n_i A_{mjn} - \n_m A_{ijn} \pm \n B  + R_{imjn} \ast  \Ph.
 }
Consequently,  \eqref{elliptic-est-Ph} with $k=1$ follows straightforwardly, namely,
\[
\|\n \Ph\|^2 \lesssim  2 \|B\|^2  + \|A\|^2 + \|R_{imjn}\|_{L^\infty} \cdot \|\Ph \|^2.
\]

Moreover,  the Laplacian equation for $\Ph$ together with the inequality \eqref{Bochner-2} 
implies \eqref{elliptic-est-Ph} with $k=2$.

For the general $k \geq 3$, we have for $1 \leq N-(k-3)$, with $N >\frac{n}{2}$,
\alis{
\| \nabla^{\mathring{k}} \Delta^{ [\frac{k}{2}] } \Ph \|^2 \lesssim {}& \|\n_{I_{k-1}} A\|^2 + \|\n_{I_{k-1}} B \|^2 + \| \n_{I_{k-2}} \( R_{imjn} \ast \Ph \) \|^2 \\
 \lesssim{}& \|\n_{I_{k-1}} A\|^2 + \|\n_{I_{k-1}} B \|^2 + \| R_{imjn} \ast \n_{I_{k-2}} \Ph  \|^2 \\
 &+ \| \n_{I_{k-3}} \( \n R_{imjn} \ast \Ph \) \|^2  \\
 \lesssim {} & \|\n_{I_{k-1}} A\|^2 + \|\n_{I_{k-1}} B \|^2 + \|R_{imjn}\|^2_{L^\infty}   \| \Ph \|^2_{H_{k-2}} \\
 &+ \| \nabla R_{imjn} \|^2_{H_{N-1}}  \| \Ph \|^2_{H_{k-2}},
} where in the last inequality, Proposition \ref{prop-Multiplication} is used.  
Combining the above estimates with \eqref{E:Bochner-k}, 
we conclude \eqref{elliptic-est-Ph}  by induction. 

In the end, \eqref{elliptic-est-Ps}  follows in the same manner.
\epf

\subsection{An identity for the Bianchi equations}\label{sec-id-Bianchi}

\begin{proof}[Proof of Lemma \ref{lemma-div-curl}]
Appealing to the commuting identity \eqref{E:comm-2}, we can prove this lemma by induction.
In fact, it suffices to keep track of the principle part of $R_{imjn}$ (without derivatives) in the calculations. By the Gauss equation \eqref{Gauss-Riem-hat-k}, we know that the principle part of $R_{imjn}$ is $-\left(g_{ij} g_{mn} -g_{in} g_{mj} \right)$. In what follows, the notation $\simeq$ refers to equalling up to some divergence forms or error terms containing $O_{ambn}$. For example, $$R_{imjn}\simeq -\left(g_{ij} g_{mn} -g_{in} g_{mj} \right).$$ In addition, we use the following notations as well,
\begin{equation}\label{def-k-prime} 
\mathring{k} =
\begin{cases}
0, &\text{if} \,\, k \,\, \text{is even}, \\
1, &\text{if} \,\, k \,\, \text{is odd};
\end{cases} 
\qquad
k^\prime =
\begin{cases}
1, &\text{if} \,\, k \,\, \text{is even}, \\
0, &\text{if} \,\, k \,\, \text{is odd}.
\end{cases}
\end{equation}

First of all,
\begin{align}\label{pre-case}
& \quad \n^p \nabla^{\mathring{k}} \Delta^{ [\frac{k}{2}] } H_{p j i} \cdot \nabla^{\mathring{k}} \Delta^{ [\frac{k}{2}] } E^{i j}  +    \n_p \nabla^{\mathring{k}} \Delta^{ [\frac{k}{2}] } E_{ij}  \cdot \nabla^{\mathring{k}} \Delta^{ [\frac{k}{2}] } H^{p i j} \nnb \\
& =  g^{ab}  \nabla^p  \nabla^{\mathring{k}}_a \Delta^{ [\frac{k}{2}] } H_{p (i j)} \cdot  \nabla^{\mathring{k}}_b \Delta^{ [\frac{k}{2}] } E^{ij} +  g^{ab} \nabla_p   \nabla^{\mathring{k}}_a \Delta^{ [\frac{k}{2}] } E_{i j} \cdot  \nabla^{\mathring{k}}_b \Delta^{ [\frac{k}{2}] }  H^{ p (i j)}  \nnb \\
&=  \nabla^p \left(g^{ab} \nabla^{\mathring{k}}_a \Delta^{ [\frac{k}{2}] }  H_{p (i j)} \cdot  \nabla^{\mathring{k}}_b \Delta^{ [\frac{k}{2}] } E^{ij} \right),
\end{align}
which concludes \eqref{E:identity} with $k=0$.

To justify the $k=1$ case, we calculate
\alis{
 & \nabla_{i_1} \nabla^p  H_{p j i} \cdot \nabla^{i_1} E^{ij}  + \nabla_{i_1} \nabla_p  E_{i j} \cdot \nabla^{i_1} H^{p ij} \\
 \stackrel{\eqref{E:comm-2}}{=} {}&   \nabla^p \nabla_{i_1} H_{p j i} \cdot \nabla^{i_1} E^{ij} + \nabla_p   \nabla_{i_1} E_{i j} \cdot \nabla^{i_1} H^{ p i j} \\
 & + \( R_{i_1}{}^{p}{}_ {p}{}^{b} H_{b j i} + R_{i_1}{}^{p}{}_ {j}{}^{b} H_{p b i} + R_{i_1}{}^{p}{}_ {i}{}^{b} H_{p j b} \) \cdot \nabla^{i_1} E^{ij} \\
 & + \( R_{i_1 p i}{}^{b} E_{b j} +  R_{i_1 p j}{}^{b} E_{i b} \) \cdot \nabla^{i_1} H^{p i j} \\
\stackrel{ \eqref{pre-case} }{=} {} &   \nabla^p \left( \nabla_{i_1}  H_{p (i j)} \cdot  \nabla^{i_1}  E^{ij} \right)  \\
  &  - \( -(n-1) g^b_{i_1} H_{b j i} + (g_{i_1 j} g^{p b} - g_{i_1}^b g_j^p) H_{p b i} + (g_{i_1 i} g^{p b} - g_{i_1}^b g_i^p)  H_{p j b} \) \cdot \nabla^{i_1} E^{ij} \\
 & - \(  (g_{i_1 i} g_p^{ b} - g_{i_1}^b g_{p i})  E_{b j} + (g_{i_1 j} g_p^{ b} - g_{i_1}^b g_{j p} )  E_{i b} \) \cdot \nabla^{i_1} H^{p i j} \\
  & + O  * H * \nabla E + O * E * \nabla H \\
= {} &    \nabla^p \left( \nabla_{i_1}  H_{p (j i)} \cdot  \nabla^{i_1}  E^{ij} \right)   \\
&  + \( (n-1) H_{i_1 j i}  + H_{j i_1 i} +  H_{i j i_1} \) \cdot \nabla^{i_1} E^{ij}  - \(  g_{i_1 i} E_{p j} + g_{i_1 j}  E_{i p} \) \cdot \nabla^{i_1} H^{p i j} \\
& + O * H * \nabla E + O  * E * \nabla H,
}
where in the last equality, we have used the fact that $H$ is trace-free. Moreover, by the ``Bianchi identity'' of $H$ \eqref{cyclic-H}, $$ H_{j i_1 i} +  H_{i j i_1} = - H_{ i_1 i j}, $$ and then it follows that
\alis{
& \( (n-1) H_{i_1 j i}  + H_{j i_1 i} +  H_{i j i_1} \) \cdot \nabla^{i_1} E^{ij}  - \(  g_{i_1 i} E_{p j} + g_{i_1 j}  E_{i p} \) \cdot \nabla^{i_1} H^{p i j} \\
=& \( (n-1) H_{i_1 j i}  - H_{ i_1 i j}  \) \cdot \nabla^{i_1} E^{ij}  -  E_{p j} \nabla_i H^{p i j}  - E_{i p} \nabla_j H^{p i j} \\
=& (n-1) H_{i_1 j i} \cdot \nabla^{i_1} E^{ij}   - H_{ i_1 i j} \cdot \nabla^{i_1} E^{ij} - \nabla_i ( E_{p j}  H^{p i j} ) +  \nabla_i  E_{p j}  H^{p i j} \\
=& (n-1) H_{i_1 j i} \cdot \nabla^{i_1} E^{ij}   - 2 H_{ i_1 i j} \cdot \nabla^{i_1} E^{ij} - \nabla_i ( E_{p j}  H^{p i j} )   \\
=& (n-3) H_{i_1 j i} \cdot \nabla^{i_1} E^{ij} - \nabla_i ( E_{p j}  H^{p i j} )\\
=& (n-3)  \nabla^{i_1} ( H_{i_1 j i} \cdot E^{ij}) - (n-3) \nabla^{i_1} H_{i_1 j i} \cdot  E^{ij} + \nabla_i ( E_{p j}  H^{i p j} )   \\
=& - (n-3) \nabla^{i_1} H_{i_1 j i} \cdot  E^{ij}  + (n-2)  \nabla^{i_1} ( H_{i_1 (i j)} \cdot E^{ij}),
}
where in the second identity, the fact $E_{i p} \nabla_j H^{p i j}=0$ (by virtue of the antisymmetry of $H_{pij}$ in the first two indices) is used.
Putting all the above calculations together, we deduce
\alis{
 & \nabla_{i_1} \nabla^p  H_{p j i} \cdot \nabla^{i_1} E^{ij}  + \nabla_{i_1} \nabla_p  E_{i j} \cdot \nabla^{i_1} H^{p ij} \\
= {} &    \nabla^p \left( \nabla_{i_1}  H_{p (i j)} \cdot  \nabla^{i_1}  E^{ij} \right)   + (n-2) \nabla^p \( H_{ p (i j)} \cdot E^{i j}  \)  \\
  & - (n-3) \nabla^{p} H_{p j i} \cdot  E^{ij} + O * H * \nabla E + O * E * \nabla H.
}
This confirms \eqref{E:identity} in the case of $k=1$.

The proof for the general $k$ case follows with the same idea. Suppose Lemma \ref{lemm-elliptic} holds for both
\ali{indu-k-1}{
&  \Delta^{ k } \nabla^p   H_{p j i} \cdot  \Delta^{k } E^{ij}   + \Delta^{ k } \nabla_p   E_{i j} \cdot   \Delta^{k }  H^{p ij}  
}
and
\ali{indu-k-2}{
&\quad \nabla_a \Delta^{k } \nabla^p  H_{p j i} \cdot  \nabla^a \Delta^{ k} E^{ij}  + \nabla_a \Delta^{ k } \nabla_p  E_{i j} \cdot  \nabla^a \Delta^{ k}  H^{p ij},
}
with  $k \in \mathbb{Z}_+$ and $E, \, H$ being any two tensors satisfying the assumptions in Lemma \ref{lemm-elliptic}, we will show that it holds with $k$ replaced by $k+1$. 
 
{\bf Step I.} We first give the proof for the case of \eqref{E:indu-k-1} with $k$ replaced by $k+1$. From the commuting identity \eqref{E:comm-2}, we compute 
\alis{
& \Delta^{k+1 } \nabla^p  H_{p j i} \cdot \Delta^{ k+1} E^{ij}  + \Delta^{ k +1} \nabla_p  E_{i j} \cdot \Delta^{ k+1}  H^{p ij} \\
={}& \Delta^{k} \nabla^p \De H_{p j i} \cdot \Delta^{ k+1} E^{ij}  + \Delta^{ k } \nabla_p \De  E_{i j} \cdot \Delta^{ k+1}  H^{p ij}  \\
&- \Delta^k  (2 R^p{}_{a p}{}^{b} \nabla^a H_{b j i} - R^{p b} \nabla_b H_{p j i}) \cdot  \Delta^{k+1 } E^{ij} \\
&- \Delta^k  (  2 R^p{}_{a j}{}^{b} \nabla^a H_{p b i} + 2 R^p{}_{a i}{}^b \n^aH_{pj b}  ) \cdot  \Delta^{k+1 } E^{ij} \\
&-  \Delta^k  (2R_{p a i}{}^{b} \nabla^a E_{b j} + 2R_{p a j}{}^{b} \nabla^a E_{i b} - R_p^b \nabla_b E_{i j}) \cdot  \Delta^{k+1 } H^{p i j} \\
&+ \Delta^k (\nabla R_{a b m n} * H) * \Delta^{k+1} E + \Delta^k (\nabla R_{a b m n} * E) * \Delta^{k+1} H \\
={}&  \Delta^{k} \nabla^p ( \De H_{p j i}) \cdot \Delta^{ k} (\De E^{ij})  + \Delta^{ k } \nabla_p ( \De  E_{i j} ) \cdot \Delta^{ k} ( \De  H^{p ij}) \\
&- \Delta^k  (  2 R^p{}_{a j}{}^{b} \nabla^a H_{p b i} + 2 R^p{}_{a i}{}^b \n^aH_{pj b} + R^{p b} \nabla_b H_{p j i}) \cdot  \Delta^{k+1 } E^{ij} \\
&-  \Delta^k  (2R_{p a i}{}^{b} \nabla^a E_{b j} + 2R_{p a j}{}^{b} \nabla^a E_{i b} - R_p^b \nabla_b E_{i j}) \cdot  \Delta^{k+1 } H^{p i j} \\
&+ \Delta^k ( \nabla O * H) * \Delta^{k+1} E + \Delta^k ( \nabla O * E) * \Delta^{k+1} H\\
={}&  \Delta^{k} \nabla^p ( \De H_{p j i}) \cdot \Delta^{ k} (\De E^{ij})  + \Delta^{ k } \nabla_p ( \De  E_{i j} ) \cdot \Delta^{ k} ( \De  H^{p ij}) \\
&+ \Delta^k  \(  2 (g^p_j g^b_a - g^{pb} g_{a j}) \nabla^a H_{p b i} + 2 (g^p_i g_a^b - g^{p b} g_{a i}) \n^a H_{p j b}  \) \cdot  \Delta^{k+1 } E^{ij} \\
& + (n-1) \Delta^k \nabla^p H_{p j i} \cdot  \Delta^{k+1 } E^{ij} -  (n-1)  \Delta^k \nabla_p E_{i j} \cdot  \Delta^{k+1 } H^{p i j}  \\
&+  \Delta^k \( 2 (g_{p i} g_a^b - g_p^b g_{a i}) \nabla^a E_{b j} + 2 (g_{p j} g_a^b - g_p^b g_{a j}) \nabla^a E_{i b} \) \cdot  \Delta^{k+1 } H^{p i j} \\
&+ \De^k (O  \ast \n H) \De^{k+1} E + \De^k (O \ast \n E) \De^{k+1} H \\
&+ \Delta^k ( \nabla O  * H) * \Delta^{k+1} E + \Delta^k ( \nabla O  * E) * \Delta^{k+1} H.
}
Due to the trace-free property of $H$, the above formulation further reduces to
\alis{
& \Delta^{k+1 } \nabla^p  H_{p j i} \cdot \Delta^{ k+1} E^{ij}  + \Delta^{ k +1} \nabla_p  E_{i j} \cdot \Delta^{ k+1}  H^{p ij} \\
={}&  \Delta^{k} \nabla^p ( \De H_{p j i}) \cdot \Delta^{ k} (\De E^{ij})  + \Delta^{ k } \nabla_p ( \De  E_{i j} ) \cdot \Delta^{ k} ( \De  H^{p ij}) \\
&+ 2 \Delta^k  (   \nabla^b H_{j b i} +   \n^b H_{i j b}  ) \cdot  \Delta^{k+1 } E^{ij} \\
& + (n-1) \Delta^k \nabla^p H_{p j i} \cdot  \Delta^{k+1 } E^{ij} -  (n-1)  \Delta^k \nabla_p E_{i j} \cdot  \Delta^{k+1 } H^{p i j}  \\
&- 2 \Delta^k  \(  \nabla_i E_{p j} +  \nabla_j E_{i p} \) \cdot  \Delta^{k+1 } H^{p i j} \\
&+ \De^k \n (O \ast  H) \ast \De^{k+1} E + \De^k \n (O \ast  E) \ast \De^{k+1} H.
}
Note that in the third and fifth lines of the above identity, the following two terms vanish \[ 2 \Delta^k   \n^b H_{i j b} \cdot  \Delta^{k+1 } E^{ij}=0, \quad - 2 \Delta^k    \nabla_j E_{i p} \cdot  \Delta^{k+1 } H^{p i j}=0, \]  for the anti-symmetry of $H$ in the first two indices. As a result,
\alis{
& \Delta^{k+1 } \nabla^p  H_{p j i} \cdot \Delta^{ k+1} E^{ij}  + \Delta^{ k +1} \nabla_p  E_{i j} \cdot \Delta^{ k+1}  H^{p ij} \\
={}&  \Delta^{k} \nabla^p ( \De H_{p j i}) \cdot \Delta^{ k} (\De E^{ij})  + \Delta^{ k } \nabla_p ( \De  E_{i j} ) \cdot \Delta^{ k} ( \De  H^{p ij}) \\
&- 2 \Delta^k  \nabla^b H_{b j i} \cdot  \Delta^{k+1 } E^{ij} +  2 \Delta^k  \nabla_i E_{p j} \cdot  \Delta^{k+1 } H^{ i p j} \\
& + (n-1) \Delta^k \nabla^p H_{p j i} \cdot  \Delta^{k+1 } E^{ij} -  (n-1)  \Delta^k \nabla_p E_{i j} \cdot  \Delta^{k+1 } H^{p i j}  \\
&+ \De^k \n (O \ast  H) \ast \De^{k+1} E + \De^k \n (O \ast  E) \ast \De^{k+1} H \\
={}&  \Delta^{k} \nabla^p ( \De H_{p j i}) \cdot \Delta^{ k} (\De E^{ij})  + \Delta^{ k } \nabla_p ( \De  E_{i j} ) \cdot \Delta^{ k} ( \De  H^{p ij}) \\
& + (n-3) \Delta^k \nabla^p H_{p j i} \cdot  \Delta^{k+1 } E^{ij} -  (n-3)  \Delta^k \nabla_p E_{i j} \cdot  \Delta^{k+1 } H^{p i j}  \\
&+ \De^k \n (O \ast  H) \ast \De^{k+1} E + \De^k \n (O \ast  E) \ast \De^{k+1} H.
}
In addition, the second line on the right hand of the last equality can be rearranged as
\alis{
& (n-3) \Delta^k \nabla^p H_{p j i} \cdot  \Delta^{k+1 } E^{ij} -  (n-3)  \Delta^k \nabla_p E_{i j} \cdot  \Delta^{k+1 } H^{p i j} \\
={}& (n-3) \n_a \( \Delta^k \nabla^p H_{p j i} \cdot \n^a \Delta^{k } E^{ij} \)  - (n-3) \n_a  \Delta^k \nabla^p H_{p j i} \cdot  \n^a  \Delta^{k } E^{ij}  \\
&-  (n-3) \n_a \( \Delta^k \nabla_p E_{i j} \cdot \n^a \Delta^{k } H^{p i j} \) +  (n-3) \n_a \Delta^k \nabla_p E_{i j} \cdot \n^a \Delta^{k } H^{p i j}. 
}

In summary, we arrive at
\ali{id-induction-2k+2}{
& \Delta^{k+1 } \nabla^p  H_{p j i} \cdot \Delta^{ k+1} E^{ij}  + \Delta^{ k +1} \nabla_p  E_{i j} \cdot \Delta^{ k+1}  H^{p ij}\nnb \\
={}&  \Delta^{k} \nabla^p ( \De H_{p j i}) \cdot \Delta^{ k} (\De E^{ij})  + \Delta^{ k } \nabla_p ( \De  E_{i j} ) \cdot \Delta^{ k} ( \De  H^{p ij}) \nnb\\
& + (n-3) \( \n_a  \Delta^k \nabla^p H_{p j i} \cdot  \n^a  \Delta^{k } E^{ij} + \n_a \Delta^k \nabla_p E_{i j} \cdot \n^a \Delta^{k } H^{p i j} \) \nnb \\
& -2  (n-3) \n_a  \Delta^k \nabla^p H_{p j i} \cdot  \n^a  \Delta^{k } E^{ij} \nnb \\
& + (n-3) \n_a \( \Delta^k \nabla^p H_{p j i} \cdot \n^a \Delta^{k } E^{ij} \) \nnb \\
& -  (n-3) \n_a \( \Delta^k \nabla_p E_{i j} \cdot \n^a \Delta^{k } H^{p i j} \) \nnb \\
&+ \De^k \n (O \ast  H) \ast \De^{k+1} E + \De^k \n (O \ast  E) \ast \De^{k+1} H.
}
By applying the inductive assumption \eqref{E:indu-k-1}--\eqref{E:indu-k-2} to the first two lines of the right hand side of \eqref{E:id-induction-2k+2}, we further achieve
\ali{induction-2k+2-a}{
& \Delta^{k+1 } \nabla^p  H_{p j i} \cdot \Delta^{ k+1} E^{ij}  + \Delta^{ k +1} \nabla_p  E_{i j} \cdot \Delta^{ k+1}  H^{p ij} \nnb \\
\simeq {}& -\sum_{0\leq l < 2k} C_{2k}^l (n-3)^{2k-l} \( \nabla^{\mathring{l}} \Delta^{ [\frac{l}{2}] } \nabla^p ( \De H_{p j i}) \cdot \nabla^{\mathring{l}} \Delta^{ [\frac{l}{2}] } (\De E^{ij}) \) \nnb \\
& - (n-3) \sum_{0\leq l < 2k+1} C_{2k+1}^l (n-3)^{2k+1-l} \( \nabla^{\mathring{l}} \Delta^{ [\frac{l}{2}] } \nabla^p H_{p j i} \cdot  \nabla^{\mathring{l}} \Delta^{ [\frac{l}{2}] } E^{ij}  \) \nnb \\
& -2  (n-3) \n_a  \Delta^k \nabla^p H_{p j i} \cdot  \n^a  \Delta^{k } E^{ij}.
}
In view of the commuting identity \eqref{E:comm-2}, and the trace-free property of $H$, it follows that
\alis{
\n^p \De H_{pji} = {} & \De \n^p H_{p j i} - R^{pb} \n_b H_{p j i} + 2R^p{}_a{}_p{}^b \n^a H_{b j i}  \\
& +2 R^p{}_{a j}{}^b \n^a H_{p b i} + 2 R^p{}_{a i}{}^b \n^a H_{p j b} + \n R_{abmn} \ast H \\
={}&  \De \n^p H_{p j i} + R^{pb} \n_b H_{p j i}  +2 R^p{}_{a j}{}^b \n^a H_{p b i} + 2 R^p{}_{a i}{}^b \n^a H_{p j b} \\
={}&  \De \n^p H_{p j i} - (n-1) \n^p H_{p j i}  - 2 (g^p_j g_a^b - g^{pb} g_{a j} ) \n^a H_{p b i} \\
& - 2 (g^p_i g_a^b - g^{pb} g_{a i}) \n^a H_{p j b} + O \ast \n H + \n O \ast H \\
={}&  \De \n^p H_{p j i} - (n-3) \n^p H_{p j i} - 2 \n^pH_{ijp} + O \ast \n H + \n O \ast H.
}
Substituting the above identity into \eqref{E:induction-2k+2-a}, and noting that \[\nabla^{\mathring{l}} \Delta^{ [\frac{l}{2}] } \nabla^p H_{i j p}  \cdot \nabla^{\mathring{l}} \Delta^{ [\frac{l}{2}] } (\De E^{ij}) =0, \]
 we obtain
\ali{induction-2k+2-b}{
& \Delta^{k+1 } \nabla^p  H_{p j i} \cdot \Delta^{ k+1} E^{ij}  + \Delta^{ k +1} \nabla_p  E_{i j} \cdot \Delta^{ k+1}  H^{p ij}\nnb \\
\simeq {}& -\sum_{0\leq l < 2k} C_{2k}^l (n-3)^{2k-l} \( \nabla^{\mathring{l}} \Delta^{ [\frac{l}{2}] +1 } \nabla^p H_{p j i} \cdot \nabla^{\mathring{l}} \Delta^{ [\frac{l}{2}] +1}  E^{ij} \) \nnb \\
&+ \sum_{0\leq l < 2k} C_{2k}^l (n-3)^{2k+1-l} \( \nabla^{\mathring{l}} \Delta^{ [\frac{l}{2}] } \nabla^p H_{p j i} \cdot \nabla^{\mathring{l}} \Delta^{ [\frac{l}{2}] +1 }  E^{ij} \) \nnb \\
& - \sum_{0\leq l < 2k+1} C_{2k+1}^l (n-3)^{2k+2-l} \( \nabla^{\mathring{l}} \Delta^{ [\frac{l}{2}] } \nabla^p H_{p j i} \cdot  \nabla^{\mathring{l}} \Delta^{ [\frac{l}{2}] } E^{ij}  \) \nnb \\
& -2  (n-3) \n_a  \Delta^k \nabla^p H_{p j i} \cdot  \n^a  \Delta^{k } E^{ij}.
}
The third line in  \eqref{E:induction-2k+2-b} can be further computed via the Leibniz rule as follows
\alis{
& \sum_{0\leq l < 2k} C_{2k}^l (n-3)^{2k+1-l} \( \nabla^{\mathring{l}} \Delta^{ [\frac{l}{2}] } \nabla^p H_{p j i} \cdot \nabla^{\mathring{l}} \Delta^{ [\frac{l}{2}] +1 }  E^{ij} \) \\
={}&  \sum_{0\leq l < 2k} C_{2k}^l (n-3)^{2k+1-l} \n \( \nabla^{\mathring{l}} \Delta^{ [\frac{l}{2}] } \nabla^p H_{p j i} \cdot \nabla^{l^\prime} \Delta^{ [\frac{l+1}{2}]  }  E^{ij} \) \\
&- \sum_{0\leq l < 2k} C_{2k}^l (n-3)^{2k+1-l} \( \nabla^{l^\prime} \Delta^{ [\frac{l+1}{2}] } \nabla^p H_{p j i} \cdot \nabla^{l^\prime} \Delta^{ [\frac{l+1}{2}] }  E^{ij} \).
}
Furthermore, denoting $m=l+1$, we reformulate the last term above as
\alis{
&- \sum_{0 \leq l < 2k} C_{2k}^l (n-3)^{2k+1-l} \( \nabla^{l^\prime} \Delta^{ [\frac{l+1}{2}] } \nabla^p H_{p j i} \cdot \nabla^{l^\prime} \Delta^{ [\frac{l+1}{2}] }  E^{ij} \)\\
={}& - \sum_{1\leq m < 2k+1} C_{2k}^{m-1} (n-3)^{2k+2-m} \( \nabla^{\mathring{m}} \Delta^{ [\frac{m}{2}] } \nabla^p H_{p j i} \cdot \nabla^{\mathring{m}} \Delta^{ [\frac{m}{2}] }  E^{ij} \).
}
Similarly, denoting $m=l+2$, we re-express the second line of \eqref{E:induction-2k+2-b} as below
\alis{
&-\sum_{0\leq l < 2k} C_{2k}^l (n-3)^{2k-l} \( \nabla^{\mathring{l}} \Delta^{ [\frac{l}{2}] +1 } \nabla^p H_{p j i} \cdot \nabla^{\mathring{l}} \Delta^{ [\frac{l}{2}] +1}  E^{ij} \) \\
={}& - \sum_{2 \leq m < 2k+2} C_{2k}^{m-2} (n-3)^{2k+2-m} \( \nabla^{\mathring{m}} \Delta^{ [\frac{m}{2}] } \nabla^p H_{p j i} \cdot \nabla^{\mathring{m}} \Delta^{ [\frac{m}{2}] }  E^{ij} \).
}
Taking all the above rearrangements into account, we manage to show that \eqref{E:induction-2k+2-b} becomes
\alis{
& \Delta^{k+1 } \nabla^p  H_{p j i} \cdot \Delta^{ k+1} E^{ij}  + \Delta^{ k +1} \nabla_p  E_{i j} \cdot \Delta^{ k+1}  H^{p ij} \\
\simeq {}& - \sum_{2 \leq m < 2k+2} C_{2k}^{m-2} (n-3)^{2k+2-m} \( \nabla^{\mathring{m}} \Delta^{ [\frac{m}{2}] } \nabla^p H_{p j i} \cdot \nabla^{\mathring{m}} \Delta^{ [\frac{m}{2}] }  E^{ij} \) \\
& - \sum_{1\leq m < 2k+1} C_{2k}^{m-1} (n-3)^{2k+2-m} \( \nabla^{\mathring{m}} \Delta^{ [\frac{m}{2}] } \nabla^p H_{p j i} \cdot \nabla^{\mathring{m}} \Delta^{ [\frac{m}{2}] }  E^{ij} \) \\
& - \sum_{0\leq m < 2k+1} C_{2k+1}^m (n-3)^{2k+2-m} \( \nabla^{\mathring{m}} \Delta^{ [\frac{m}{2}] } \nabla^p H_{p j i} \cdot  \nabla^{\mathring{m}} \Delta^{ [\frac{m}{2}] } E^{ij}  \)\\
& -2  (n-3) \n_a  \Delta^k \nabla^p H_{p j i} \cdot  \n^a  \Delta^{k } E^{ij}.
}
Due to the formula \[ C_{2k}^{m-2} + C_{2k}^{m-1} + C_{2k+1}^m = C_{2k+1}^{m-1} + C_{2k+1}^m = C_{2k+2}^m, \] it follows that
\alis{
& \Delta^{k+1 } \nabla^p  H_{p j i} \cdot \Delta^{ k+1} E^{ij}  + \Delta^{ k +1} \nabla_p  E_{i j} \cdot \Delta^{ k+1}  H^{p ij} \\
\simeq {}& - \sum_{2 \leq m < 2k+1} C_{2k+2}^{m} (n-3)^{2k+2-m} \( \nabla^{\mathring{m}} \Delta^{ [\frac{m}{2}] } \nabla^p H_{p j i} \cdot \nabla^{\mathring{m}} \Delta^{ [\frac{m}{2}] }  E^{ij} \) \\
& - C_{2k}^{2k-1} (n-3) \( \n_a  \Delta^k \nabla^p H_{p j i} \cdot  \n^a  \Delta^{k } E^{ij} \)  \\
&- C_{2k}^{0} (n-3)^{2k+1} \( \nabla_a \nabla^p H_{p j i} \cdot \nabla^a E^{ij} \) \\
& - C_{2k+1}^1 (n-3)^{2k+1} \( \nabla_a  \nabla^p H_{p j i} \cdot  \nabla^a  E^{ij}  \) \\
&- C_{2k+1}^0 (n-3)^{2k+2} \(   \nabla^p H_{p j i} \cdot E^{ij}  \)\\
& -2  (n-3) \n_a  \Delta^k \nabla^p H_{p j i} \cdot  \n^a  \Delta^{k } E^{ij} \\
= {}& - \sum_{2 \leq m < 2k+1} C_{2k+2}^{m} (n-3)^{2k+2-m} \( \nabla^{\mathring{m}} \Delta^{ [\frac{m}{2}] } \nabla^p H_{p j i} \cdot \nabla^{\mathring{m}} \Delta^{ [\frac{m}{2}] }  E^{ij} \) \\
& - (2k+2) (n-3) \( \n_a  \Delta^k \nabla^p H_{p j i} \cdot  \n^a  \Delta^{k } E^{ij} \) \\
&- (n-3)^{2k+1} \( \nabla_a \nabla^p H_{p j i} \cdot \nabla^a E^{ij} \) \\
& - (2k+1)(n-3)^{2k+1} \( \nabla_a  \nabla^p H_{p j i} \cdot  \nabla^a  E^{ij}  \) \\
&-  (n-3)^{2k+2} \(   \nabla^p H_{p j i} \cdot E^{ij}  \) \\
= {}& - \sum_{2 \leq m \leq 2k+1} C_{2k+2}^{m} (n-3)^{2k+2-m} \( \nabla^{\mathring{m}} \Delta^{ [\frac{m}{2}] } \nabla^p H_{p j i} \cdot \nabla^{\mathring{m}} \Delta^{ [\frac{m}{2}] }  E^{ij} \) \\
& - (2k+2)(n-3)^{2k+1} \( \nabla_a  \nabla^p H_{p j i} \cdot  \nabla^a  E^{ij}  \) -  (n-3)^{2k+2} \(   \nabla^p H_{p j i} \cdot E^{ij}  \)\\
={}&  - \sum_{ 0 \leq m \leq 2k+1} C_{2k+2}^{m} (n-3)^{2k+2-m} \( \nabla^{\mathring{m}} \Delta^{ [\frac{m}{2}] } \nabla^p H_{p j i} \cdot \nabla^{\mathring{m}} \Delta^{ [\frac{m}{2}] }  E^{ij} \).
}
This verifies the case of \eqref{E:indu-k-1} with $k$ replaced by $k+1$.

{\bf Step II.}  Next, we only sketch the proof for the case of \eqref{E:indu-k-2} with $k$ replaced by $k+1$, since it is similar to Step I.
In analogy with \eqref{E:id-induction-2k+2}, we deduce
\alis{
& \nabla_a \Delta^{k+1 } \nabla^p  H_{p j i} \cdot  \nabla^a \Delta^{ k+1} E^{ij}  + \nabla_a \Delta^{ k +1} \nabla_p  E_{i j} \cdot  \nabla^a \Delta^{ k+1}  H^{p i j}  \\
={}& \n_a  \Delta^{k} \nabla^p ( \De H_{p j i}) \cdot \n^a \Delta^{ k} (\De E^{ij})  + \n_a \Delta^{ k } \nabla_p ( \De  E_{i j} ) \cdot \n^a \Delta^{ k} ( \De  H^{p ij}) \\
& + (n-3) \( \Delta^{k+1} \nabla^p H_{p j i} \cdot  \Delta^{k +1} E^{ij} + \Delta^{k+1} \nabla_p E_{i j} \cdot \Delta^{k+1 } H^{p i j} \) \\
& - 2 (n-3) \Delta^{k+1} \nabla^p H_{p j i} \cdot  \Delta^{k +1} E^{ij} \\
& + (n-3) \n^a \(\n_a \Delta^k \nabla^p H_{p j i} \cdot \Delta^{k +1} E^{ij} \) \\
&-  (n-3) \n^a \( \n_a \Delta^k \nabla_p E_{i j} \cdot \Delta^{k +1} H^{p i j} \) \\
&+ \n \De^k \n (O  \ast  H) \ast \n \De^{k+1} E + \n \De^k \n (O  \ast  E) \ast \n \De^{k+1} H.
}
By applying the inductive assumption \eqref{E:indu-k-2} and the result of Step 1 to the first two lines on the right hand side of the above identity, we obtain,
\alis{
& \nabla_a \Delta^{k+1 } \nabla^p  H_{p j i} \cdot  \nabla^a \Delta^{ k+1} E^{ij}  + \nabla_a \Delta^{ k +1} \nabla_p  E_{i j} \cdot  \nabla^a \Delta^{ k+1}  H^{p i j}  \\
\simeq {}& -\sum_{0\leq l < 2k+1} C_{2k+1}^l (n-3)^{2k+1-l} \( \nabla^{\mathring{l}} \Delta^{ [\frac{l}{2}] } \nabla^p ( \De H_{p j i}) \cdot \nabla^{\mathring{l}} \Delta^{ [\frac{l}{2}] } (\De E^{ij}) \) \\
& - (n-3) \sum_{0\leq l < 2k+2} C_{2k+2}^l (n-3)^{2k+2-l} \( \nabla^{\mathring{l}} \Delta^{ [\frac{l}{2}] } \nabla^p H_{p j i} \cdot  \nabla^{\mathring{l}} \Delta^{ [\frac{l}{2}] } E^{ij}  \)\\
& - 2 (n-3) \Delta^{k+1} \nabla^p H_{p j i} \cdot  \Delta^{k +1} E^{ij}.
}
For the rest of proof, we can follow the procedure of Step 1 to achieve
\alis{
& \nabla_a \Delta^{k+1 } \nabla^p  H_{p j i} \cdot  \nabla^a \Delta^{ k+1} E^{ij}  + \nabla_a \Delta^{ k +1} \nabla_p  E_{i j} \cdot  \nabla^a \Delta^{ k+1}  H^{p i j} \\
\simeq {}&  - \sum_{ 0 \leq m < 2k+3} C_{2k+3}^{m} (n-3)^{2k+3-m} \( \nabla^{\mathring{m}} \Delta^{ [\frac{m}{2}] } \nabla^p H_{p j i} \cdot \nabla^{\mathring{m}} \Delta^{ [\frac{m}{2}] }  E^{ij} \).
}
This finishes the proof for the case of \eqref{E:indu-k-2} with $k$ replaced by $k+1$.

\end{proof}
    
    \subsection{Proof of Proposition \ref{prop-decomp-Bianchi}}\label{sec-proof-conn}           
 In this subsection, we make use of the connection formula \eqref{E:eq-connection} to complete the proof of Proposition \ref{prop-decomp-Bianchi}.
\bpf[Proof of Proposition \ref{prop-decomp-Bianchi}]

To justify \eqref{id-DT-W-T},  we denote a $(0,3)$-tensor on $\M$ by \[ \Hwb_{\alpha \mu \nu} :=  \W_{\alpha \mu \nu t}. \]  Note that, $\Hwb$ is not an $M$-tensor. However, the projection of $\Hwb$ onto $M$ is $t \Hw$, which is an $M$-tensor in the sense of subsection \ref{def-M-tensor}. Using of the connection formula \eqref{E:eq-connection}, 
\begin{align*}
\D_{\p_t} \W_{i j l t} ={}&  \lie_{\p_t} \Hwb_{i j l} -  \W (\D_i \p_t, e_j, e_l, \p_t)  -  \W (e_i, \D_j \p_t, e_l, \p_t) \\
&  -  \W (e_i, e_j, \D_l \p_t, \p_t) -  \W (e_i, e_j, e_l,  \D_{\p_t} \p_t)\\
={}&  \lie_{\p_t} \Hwb_{i j l} + \ti k_i^q \W_{q j l t} + \ti k_j^q \W_{i q l t}   + \ti k_l^q  \W _{i j q t}.
\end{align*}
Here we use the notations \[\D_{\p_t} \W_{i j l t} := \D_{\p_t} \W_{\mu \nu \al t} \breve h^\mu_i  \breve h^\nu_j \breve h^{\al}_l, \quad  \lie_{\p_t} \Hwb_{ijl} :=  \lie_{\p_t} \Hwb_{\mu \nu \al} \breve h^{\mu}_i \breve h^{\nu}_j \breve h^{\al}_l.  \]
Since $[\p_t, e_i] =  \D_{\p_t} e_i - \D_i \p_t$ and by the connection formula \eqref{E:eq-connection} \[ [\p_t, \, e_i] =  \D_{\p_t} e_i - \D_i \p_t = \ti \n_{\p_t} e_i - \ti \n_i \p_t, \] it follows that \[ \lie_{\p_t} \Hwb_{ijl}  = \lie_{\p_t} \(t \Hw \) _{ijl}. \] As a consequence,
\[ \D_{\p_t} \W_{i j q t} = \lie_{\p_t} \(t \Hw \) _{i j q} + k_i^p \Hw_{p j q} + k_j^p \Hw_{i p q}    + k_q^p  \Hw_{i j p}, \] 
and then \eqref{id-DT-W-T} follows.
 
In the same way, we calculate 
\begin{align*}
\D_{\p_t} \W_{i p q j} 
={}& \lie_{\p_t} \W_{i p q j} + \ti k_{i}^l \W_{l p q j} + \ti k_{p}^l \W_{i l q j} + \ti k_{q}^l \W_{i p l j} + \ti k_{j}^l \W_{i p q l} \\
={}& \lie_{\p_t} \( t^{2} \Kw\)_{i p q j} + t^{2} \( \ti k_{i}^l \Kw_{l p q j} + \ti k_{p}^l \Kw_{i l q j} + \ti k_{q}^l \Kw_{i p l j} + \ti k_{j}^l \Kw_{i p q l} \)\\
={}& t  \lie_{\dtau} \Kw_{i p q j}  + 2 t \Kw_{i p q j}  + t \( k_{i}^l \Kw_{l p q j} + k_{p}^l \Kw_{i l q j} + k_{q}^l \Kw_{i p l j} + k_{j}^l \Kw_{i p q l} \),
\end{align*}
and
\alis{
\D_{\p_t} \W_{i t j t} ={}& \D_{\p_t} \Ew_{i j} - \W (e_i,\D_{\p_t} \p_t, e_j, \p_t) - \W ( e_i, \p_t, e_j,  \D_{\p_t}  \p_t )\\
={}&  \D_{\p_t} \Ew_{i j} = \lie_{\p_t} \Ew_{i j} - \Ew(\D_i \p_t, e_j) - \Ew(e_i, \D_j \p_t) \\
={}& \lie_{\p_t} \Ew_{i j} + \ti k_i^p \Ew_{p j} + \ti k_j^p \Ew_{i p } \\
={}& t^{-1}  \lie_{\dtau} \Ew_{i j} + t^{-1} \(  k_i^p \Ew_{p j} + k_j^p \Ew_{i p }\).
}
That is, we prove \eqref{id-DT-W-p} and \eqref{id-DT-W-TT}. 


In the end, due to $\D_{i} e_j = \ti \n_i e_j - \ti k_{i j} \p_t$, we have
\begin{align*}
 \D_p \W_{q t i j} = {}& \D_p \Hwb_{ijq} + \ti k_{p}^l \W_{q l i j}  \\
 ={}& \ti\nabla_p \( t \Hw \) _{i j q}+ \ti k_{p}^l \W_{q l i j} + \ti k_{p i} \W_{t j q t } + \ti k_{p j} \W_{ i t q t}   \\
= {}& \nabla_p \( t \Hw \) _{i j q} + t \( k_{p}^l \Kw_{q l i j} - k_{p i} \Ew_{q j} + k_{p j} \Ew_{q i}\).
\end{align*}
Similarly, the following identities
\alis{
 \D_p \W_{i m j n}  ={}& t^2 \ti \nabla_p \Kw_{i m j n} + \ti k_{p i} \W_{t m j n}  + \ti k_{p m} \W_{i t j n} + \ti k_{p j} \W_{i m t n}  + \ti k_{p n} \W_{i m j t} \\
= {}& t^2\( \nabla_p \Kw_{i m j n} - k_{p i} \Hw_{j n m }  + k_{p m} \Hw_{ j n i} - k_{p j} \Hw_{i m n} + k_{p n} \Hw_{i m j}\),
}
and
\alis{
\D_p \W_{i t j t} = {}& \ti \nabla_p \Ew_{i j} + \ti k_{p}{}^{\!q} \W_{i q j t} + \ti k_{p}{}^{\!q}\W_{j q i t }\\
={}& \nabla_p \Ew_{i j} + k_{p}{}^{\!q} \Hw_{i q j} + k_{p}{}^{\!q} \Hw_{j q i},
}
hold as well. Therefore, \eqref{id-Dp-W-T}, \eqref{id-Dp-W-p} and \eqref{id-Dp-W-TT} are concluded.


\epf

\end{document}